\documentclass[a4paper,11pt]{amsart}
\usepackage[margin=3cm]{geometry}

\usepackage{hyperref}
\hypersetup{
    colorlinks,
    citecolor=green,
    filecolor=black,
    linkcolor=blue,
    urlcolor=blue
}
\hypersetup{linktocpage}

\usepackage{cite}
\usepackage{graphicx}

\usepackage{subcaption}
\usepackage{amsbsy}
\usepackage{latexsym}
\usepackage{amsfonts}
\usepackage{amssymb}
\usepackage{upgreek}
\usepackage[usenames]{color}
\usepackage{amsmath,amsthm}
\usepackage{enumerate}

\usepackage{tikz}
\usetikzlibrary{arrows,positioning,automata,calc}
\usepackage{mathdots}
\usepackage{mathtools}

\usepackage{stmaryrd}
\usepackage{dsfont}

\usepackage[foot]{amsaddr}

\DeclareMathOperator{\rank}{rank}

\DeclareMathOperator*{\argmin}{arg\,min}

\providecommand{\abs}[1]{\lvert#1\rvert}
\providecommand{\bigabs}[1]{\bigl\lvert#1\bigr\rvert}
\providecommand{\Bigabs}[1]{\Bigl\lvert#1\Bigr\rvert}
\providecommand{\biggabs}[1]{\biggl\lvert#1\biggr\rvert}
\providecommand{\lrabs}[1]{\left\lvert#1\right\rvert}

\providecommand{\norm}[1]{\lVert#1\rVert}
\providecommand{\bignorm}[1]{\bigl\lVert#1\bigr\rVert}
\providecommand{\Bignorm}[1]{\Bigl\lVert#1\Bigr\rVert}

\providecommand{\lrnorm}[1]{\left\lVert#1\right\rVert}

\providecommand{\inp}[2]{\langle#1,#2\rangle}

\providecommand{\ceil}[1]{\lceil#1\rceil}

\newcommand{\compl}[1]{{#1}^{\mathsf{c}}}

\newtheorem{theorem}{Theorem}
\newtheorem{lemma}[theorem]{Lemma}
\newtheorem{prop}[theorem]{Proposition}
\newtheorem{cor}[theorem]{Corollary}
\newtheorem{ass}{Assumptions}

\theoremstyle{definition}
\newtheorem{definition}[theorem]{Definition}

\theoremstyle{remark}
\newtheorem{remark}[theorem]{Remark}

\newcommand{\Real}{\operatorname{Re}}
\newcommand{\Imag}{\operatorname{Im}}

\newcommand{\rt}{\mathrm{t}}
\newcommand{\rx}{\mathrm{x}}
\newcommand{\ad}{\mathrm{ad}}

\newcommand{\cA}{{\mathcal{A}}}

\newcommand{\cL}{{\mathcal{L}}}

\newcommand{\cF}{{\mathcal{F}}}
\newcommand{\cH}{{\mathcal{H}}}

\newcommand{\cR}{\mathcal{R}}

\newcommand{\cX}{\mathcal{X}}
\newcommand{\cY}{\mathcal{Y}}

\newcommand{\sfr}{\mathsf{r}}
\newcommand{\sfn}{\mathsf{n}}
\newcommand{\sfm}{\mathsf{m}}
\newcommand{\sfM}{\mathsf{M}}
\newcommand{\sfK}{\mathsf{K}}

\newcommand{\bA}{\mathbf{A}}
\newcommand{\bC}{\mathbf{C}}
\newcommand{\bM}{\mathbf{M}}
\newcommand{\bG}{\mathbf{G}}
\newcommand{\bI}{\mathbf{I}}
\newcommand{\bT}{\mathbf{T}}

\newcommand{\bu}{\mathbf{u}}
\newcommand{\bfs}{\mathbf{f}}
\newcommand{\bU}{\mathbf{U}}
\newcommand{\bv}{\mathbf{v}}
\newcommand{\bw}{\mathbf{w}}
\newcommand{\bz}{\mathbf{z}}
\newcommand{\bwkj}{\bw_{k,j}}
\newcommand{\bzkj}{\bz_{k,j}}
\newcommand{\bp}{\mathbf{p}}
\newcommand{\bbf}{\mathbf{f}}
\newcommand{\bB}{\mathbf{B}}
\newcommand{\bD}{\mathbf{D}}

\newcommand{\br}{\mathbf{r}}

\newcommand{\bDX}{\mathbf{D}_{\cX}}

\newcommand{\bDY}{\mathbf{D}_{\cY}}

\newcommand{\bDXtildeNN}{\mathbf{\tilde D}_{\cX,\sfn_1}}

\newcommand{\bDXnut}{\mathbf{D}_{\cX,\nu_\rt}}

\newcommand{\bAtil}{\mathbf{\tilde A}}
\newcommand{\bBtil}{\mathbf{\tilde B}}
\newcommand{\bTtil}{\mathbf{\tilde T}}

\newcommand{\mcX}{\sfm_{\cX}}
\newcommand{\mcXnu}{\sfm_{\cX,\nu_\rt}}
\newcommand{\mcY}{\sfm_{\cY}}
\newcommand{\mcYnu}{\sfm_{\cY,\nu_\rt}}
\newcommand{\hmcX}{\hat{\sfm}_{\cX}}
\newcommand{\hmcXnu}{\hat{\sfm}_{\cX,\nu_\rt}}
\newcommand{\hmcY}{\hat{\sfm}_{\cY}}
\newcommand{\hmcYnu}{\hat{\sfm}_{\cY,\nu_\rt}}

\newcommand{\nut}{\nu_\rt}
\newcommand{\veet}{\vee_\rt}

\newcommand{\dd}{\mathrm{d}} 
\newcommand{\sdd}{\,\mathrm{d}} 
\DeclareMathOperator{\supp}{supp}
\DeclareMathOperator{\range}{range}

\newcommand{\MX}{\sfM_{\cX}}
\newcommand{\MXnu}{M_{\cX,\nu_\rt}}
\newcommand{\MXnuZe}{M_{\cX,0,\nu_\rt}}
\newcommand{\MXZe}{\sfM_{\cX,0}}
\newcommand{\MY}{\sfM_{\cY}}
\newcommand{\MYZe}{\sfM_{\cY,0}}
\newcommand{\mY}{M_{\cY}}
\newcommand{\mYZe}{M_{\cY,0}}

\newcommand{\Chi}{\raise .3ex
\hbox{\large $\chi$}}

\newcommand{\R}{\mathbb{R}}

\newcommand{\N}{\mathbb{N}}
\newcommand{\Z}{\mathbb{Z}}
\newcommand{\C}{\mathbb{C}}

\newcommand{\Pro}{\mathrm{P}}
\newcommand{\Res}{\mathrm{R}}
\newcommand{\Cor}{\mathrm{C}}

\newcommand{\erfcinv}{\operatorname{erfc}^{-1}}
\newcommand{\erf}{\operatorname{erf}}

\newcommand{\recompress}{\textsc{recompress}}
\newcommand{\coarsen}{\textsc{coarsen}}
\newcommand{\apply}{\textsc{apply}}
\newcommand{\rhs}{\textsc{rhs}}

\newcommand{\pluseq}{\mathrel{+}=}

\newcommand{\cAs}{\cA^s}
\newcommand{\piti}{\pi^{(\rt,i)}}
\newcommand{\pitiOne}{\pi^{(\rt,1)}}
\newcommand{\pit}{\pi^{(\rt)}}

\newcommand{\KdR}{\mathrm{K}_d(R)}

\newcommand{\KdOneTwo}{\mathrm{K}_d(1,2,\dots,2)}

\newcommand{\constsvd}{\kappa_{\rm P}}
\newcommand{\constcrs}{\kappa_{\rm C}}

\usepackage[ruled,lined, linesnumbered,algosection]{algorithm2e}
\usepackage{todonotes}

\numberwithin{equation}{section}
\numberwithin{theorem}{section}

\theoremstyle{plain}

\title[A Space-Time Adaptive Low-Rank Method for Parabolic PDEs]{A Space-Time Adaptive Low-Rank Method for High-Dimensional Parabolic Partial Differential Equations}
\author{Markus Bachmayr$^1$} 
\email{bachmayr@igpm.rwth-aachen.de}
\author{Manfred Faldum$^1$}
\email{faldum@igpm.rwth-aachen.de}
\address{$^1$ Institut f\"ur Geometrie und Praktische Mathematik, RWTH Aachen University, Templergraben 55, 52056 Aachen, Germany}
\date{\today}

\thanks{M.B.\ and M.F.\ acknowledge funding by Deutsche Forschungsgemeinschaft (DFG, German Research Foundation) -- project numbers 233630050, 320021702, 442047500 -- TRR 146 ``Multiscale Simulation Methods for Soft Matter Systems'', GRK2326 ``Energy, Entropy, and Dissipative Dynamics (EDDy)'' and SFB 1481 ``Sparsity and Singular Structures''.}

\begin{document}

\maketitle

 \begin{abstract}
   An adaptive method for parabolic partial differential equations that combines sparse wavelet expansions in time with adaptive low-rank approximations in the spatial variables is constructed and analyzed. The method is shown to converge and satisfy similar complexity bounds as existing adaptive low-rank methods for elliptic problems, establishing its suitability for parabolic problems on high-dimensional spatial domains. The construction also yields computable rigorous a posteriori error bounds for such problems. The results are illustrated by numerical experiments.
 \end{abstract}

\section{Introduction}
\label{sec:Intro}

The numerical approximation of parabolic partial differential equations (PDEs) on high-dimensional spaces is of interest in a wide range of applications. In particular, problems of this type arise as Kolmogorov equations associated to stochastic processes, where they provide a deterministic description of the time evolution of densities and expectations under the stochastic dynamics. In the example of many-particle systems, the dimensionality of the problem is then proportional to the number of particles. Problems with similar characteristics, but typically involving additional nonlinearities, arise in mathematical finance and in optimal control.

A variety of specialized methods has been proposed for solving such problems numerically, based, for instance, on sparse expansions \cite{SpaceTimeAdaptive,Dijkema:09,SS:13}, on low-rank tensor approximations \cite{Dolgov:12,Andreev:15,Cho:16,Boiveau:19}, or on deep neural networks. In particular in the case of neural networks, such methods for high-dimensional problems typically exploit connections to stochastic differential equations to approximate point values of solutions by Monte Carlo averaging of sample paths, which can be used to construct approximate solutions by regression, see for example \cite{MR4293960}.

In this work, we aim at methods based on sparse and low-rank representations that offer scalability to high dimensions, but at the same time allow for reliable \emph{deterministic a posteriori control} of numerical errors with respect to the exact solution of the differential equation in the relevant norms. The main concern here, both in the computation of approximations and of corresponding error bounds, is to avoid the curse of dimensionality, that is, to achieve computational costs that ideally have low-order polynomial (rather than exponential) scaling with respect to the the dimensionality.

A typical model problem that we focus on here is the following instationary diffusion equation for the time-dependent function $u$ on the $d$-dimensional unit cube $\Omega = (0,1)^d$ and time interval $[0,T]$ for a $T>0$ with initial data $u_0$,
\begin{equation}\label{eq:modeldiffusion}
   \partial_t u - \nabla_x \cdot ( a \nabla_x u) = f\;\text{ in $(0,T]\times \Omega$}, \quad u|_{t=0} = u_0\;\text{ on $\Omega$},
\end{equation}
where $a$ and $f$ are the given diffusion coefficient and source term, respectively.
For simplicity we assume homogeneous Dirichlet boundary conditions on $\partial \Omega$.

\subsection{Relation to existing results}

 For corresponding stationary elliptic problems of the above type, methods using near-sparsity of solutions in suitable tensor product bases -- such as sparse grids or adaptive variants based on wavelets -- have been shown to be applicable to problems of moderate dimensionality up to $d\approx 20$. As demonstrated for high-dimensional Poisson problems in \cite{Dijkema:09}, however, the approximability of solutions with respect to such bases itself may in general deteriorate exponentially with $d$. 
 This restriction to moderate dimensions is also visible in the case of a non-adaptive treatment of parabolic problems combining time stepping with a spatial discretization by sparse grids in \cite{vPS:04}.

In contrast to methods relying on sparsity with respect to a given basis, low-rank approaches can make problems in higher dimensions computationally accessible by exploiting further structural features. In such methods, the expansion coefficients of solutions with respect to a product basis are represented in suitable low-rank tensor formats. The adaptive solvers of this type developed in \cite{BachmayrNearOptimal,bachmayr2014adaptive} for high-dimensional elliptic problems offer systematic error reduction with near-optimal asymptotic computational complexity and explicitly computable bounds of the $H^1$-error with respect to $u$. They have been shown to avoid the curse of dimensionality in elliptic problems with suitable low-rank approximability, including the test cases of \cite{Dijkema:09}, where in the numerical tests for adaptive low-rank solvers in \cite{bachmayr2014adaptive,BD16}, dimensions up to $d=256$ are treated.

In principle, for extending such adaptive low-rank concepts to time-dependent problems, a variety of basic constructions is possible. One can, for instance, directly rely on a low-rank solver for elliptic problems to implement a time stepping scheme. Such approaches have been proposed, for instance, with fixed spatial discretization in \cite{Dolgov:12,Cho:16}. In this case, however, the evolution of tensor ranks and the computational complexity are difficult to control. 

Dynamical low-rank approximation \cite{KochLubich2007,LubichRohwedderSchneiderVandereycken:13} offers a different strategy for obtaining approximate evolutions on manifolds of fixed-rank tensors. Its application to parabolic equations has been considered in the case $d=2$ in \cite{Conte:20,BEKU:21} and for parametric problems in \cite{Kazashietal:21}.
No methods of this type are known, however, that would allow us to ensure a given solution error tolerance for problems such as \eqref{eq:modeldiffusion}.

Another alternative approach are methods based on space-time variational formulations. Solvers using sparse expansions in terms of Riesz bases in space and time were obtained, for example, based on adaptive wavelet methods in \cite{SpaceTimeAdaptive,CheginiStevenson:11,KSU:16,RS:19} and with sparse polynomial approximation in \cite{SS:13}.
There is a vast literature on finite element-based methods using various different variational formulations, see, for example, \cite{GO:07,Andreev:13,LM:17,MR4320081,SW:21,FK:21,GS:21}. An approach combining wavelets in time and finite elements in space was considered in  \cite{Andreev:16} for preconditioning and in \cite{SvVW:22} in an adaptive solver.

A first natural way of using space-time formulations to obtain low-rank approximations is to treat time as an additional mode in the tensor approximation, which amounts to a low-rank separation between temporal and spatial degrees of freedom. This approach is followed in \cite{Andreev:15,Boiveau:19} with fixed discretizations. A disadvantage is that in general it is not clear whether the sought solutions actually have efficient low-rank approximations in such a format with separated temporal and spatial variables.
This can be an issue, for instance, in problems dominated by convection or with time-dependent sources that move inside the spatial domain. Such features of the problem may lead to structures in the solution that can be resolved only with large ranks in the separation between spatial and temporal variables. In addition, this separation also leads to subtle issues in the interaction of tensor structures with the relevant function spaces, which in \cite{Andreev:15} prevent a fully rigorous treatment of preconditioning in low-rank format. We comment on these difficulties and on how we avoid them in our setting in the following section and in Remark \ref{rem:fullsep}.

\subsection{Novel contributions}

We follow a new approach that combines a \emph{sparse} wavelet expansion in the time variable with a \emph{low-rank} hierarchical tensor approximation in the spatial variables. This means that for each temporal wavelet index in the approximation, we use an independent low-rank tensor representation for the spatial approximation coefficients.
For exploiting the sparsity of solutions as far as possible, the spatial discretization space for each time index is also adapted independently.

Let us consider the form that these approximations take in the case $d=2$, corresponding to two spatial variables, for the example \eqref{eq:modeldiffusion}.
We use the classical space-time variational formulation considered in \cite{SpaceTimeAdaptive}, which we discuss in further detail in Section \ref{sec:Prelim}. In this formulation, we treat  \eqref{eq:modeldiffusion} in weak form on the spatial domain $\Omega = (0,1)^2$, where solutions are sought in the space 
\begin{equation}\label{eq:Xnorm2d}
	{\mathcal{X}} = L_2\bigl(0,T; H^1_0(\Omega)\bigr) \cap H^1\bigl(0,T; H^{-1}(\Omega)\bigr) ,
\end{equation}
and where the initial condition is explicitly enforced in $L_2(\Omega)$.
We assume a suitable Riesz basis $\{ \Phi_{\nu_\rt,\nu_1,\nu_2} \}_{\nu_\rt \in \vee_\rt, \nu_1,\nu_2 \in \vee_1}$ of ${\mathcal{X}}$ with countable index sets $\vee_\rt$ and $\vee_1$ and seek approximations of the form
\begin{equation}\label{eq:approx}
   u(t,x_1,x_2) \approx \sum_{\nu_\rt \in \Lambda_\rt} \sum_{\nu_1 \in \Lambda_{1,\nu_\rt}}\sum_{\nu_2 \in \Lambda_{2,\nu_\rt}}	\bu_{\nu_\rt,\nu_1,\nu_2} \Phi_{\nu_\rt,\nu_1,\nu_2}(t, x_1, x_2) 	,
\end{equation}
with finite index sets $\Lambda_\rt \subset \vee_\rt$ and $\Lambda_{\nu_\rt,1}, \Lambda_{\nu_\rt,2} \subset \vee_1$ that need to be determined for the given approximation error tolerance.
Moreover, for each active time basis index $\nu_\rt \in \Lambda_\rt$, we need to find a rank parameter $r_{\nu_\rt} \in \N$ and vectors $\mathbf{U}^{(i)}_{\nu_\rt,k} \in \R^{\Lambda_{\nu_\rt,i}}$ for $k=1,\ldots, r_{\nu_\rt}$ and $i=1,2$ such that we have a sufficiently accurate low-rank approximation
\begin{equation}\label{eq:temporallowrank2d}
	\bu_{\nu_\rt,\nu_1,\nu_2} = \sum_{k=1}^{r_{\nu_\rt}} \mathbf{U}^{(1)}_{\nu_\rt, k, \nu_1}  \mathbf{U}^{(2)}_{\nu_\rt, k, \nu_2}\qquad \text{for $(\nu_1,\nu_2) \in \Lambda_{\nu_\rt,1} \times \Lambda_{\nu_\rt,2}$.}
\end{equation}

For computing such combined sparse and low-rank approximation, we construct a space-time adaptive solver with properties very similar to the corresponding existing method for elliptic problems from \cite{bachmayr2014adaptive}. In particular, we obtain guaranteed error reduction in $\mathcal{X}$-norm with computable space-time error bounds. At the same time, under natural low-rank approximability assumptions, we again obtain near-optimal asymptotic computational costs that approach the convergence rates of the corresponding underlying one-dimensional approximations. For large $d$, the computational complexity is guaranteed to not grow exponentially in $d$, and thus the curse of dimensionality is avoided. This result requires that the approximability of problem data and solution does not deteriorate too strongly with increasing $d$. However, note that our adaptive solver itself does not use any explicit knowledge on the low-rank approximability of the solution (a property that is also referred to as \emph{universality}), and the $d$-dependence in our numerical tests is in fact substantially more favorable than ensured by our estimates.

For achieving the desired computational complexity by a reduction to lower-dimensional operations, it is crucial that the basis functions have product structure. Specifically, with suitable orthonormal spline wavelet-type bases $\{ \theta_{\nu_\rt} \}_{\nu_\rt \in\vee_\rt}$ and $\{ \psi_{\nu} \}_{\nu \in \vee_1}$ of $L_2(0,T)$ and $L_2(0,1)$, respectively, we take
\[
    \Phi_{\nu_\rt,\nu_1,\nu_2} (t,x_1,x_2) = 
	  \norm{  \theta_{\nu_\rt} \otimes \psi_{\nu_1} \otimes \psi_{\nu_2} }_{\cX}^{-1}\, \theta_{\nu_\rt}(t)\, \psi_{\nu_1}(x_1) \, \psi_{\nu_2}(x_2)\,.
\]
Such $L_2$-orthonormal bases are provided by Donovan-Geronimo-Hardin multiwavelets \cite{multiwavelets}.
Here, the normalization in $\cX$-norm ensures the Riesz basis property, which means that for the coefficient sequence $\bu$ in \eqref{eq:approx}, we have 
 $c\norm{\bu}_{\ell_2} \leq \norm{ u }_{\cX}  \leq C \norm{\bu}_{\ell_2}$ with uniform constants. Since low-rank compressions within the method are computed with respect to the $\ell_2$-norm, this property is crucial for ensuring a total error bound in $\cX$-norm for the computed approximations. However, similarly to the case of elliptic problems in \cite{bachmayr2014adaptive}, one faces the issue that the scaling factor $\norm{  \theta_{\nu_\rt} \otimes \psi_{\nu_1} \otimes \psi_{\nu_2} }_{\cX}^{-1}$ does not have an explicit low-rank form. In order to circumvent this problem, we devise new discretization-dependent low-rank approximations for this diagonal scaling that are adapted to the present space-time setting.

Note that our basic construction is different from the one of \cite{SvVW:22}, where a wavelet discretization in time is combined with a potentially completely different discretization (for example, by finite elements) in space. The method that we consider here is based on a standard product wavelet discretization in space and time as in \cite{SpaceTimeAdaptive}, but rather uses a particular nonlinear parameterization of basis coefficients for this discretization, where time also plays a special role.

\subsection{Conceptual overview and outline}

The outline of this paper is as follows: In Section \ref{sec:Prelim}, we discuss the underlying space-time variational formulation of \cite{SpaceTimeAdaptive} and its wavelets Riesz basis representation, as well as particular requirements on the wavelet bases in the high-dimensional setting.

In Section \ref{sec:ALRbasic}, we turn to basic aspects of the combination of adaptive sparse approximation in time with low-rank approximation in the spatial variables. 
We show that the general framework for low-rank approximations in function spaces via basis representations developed in \cite{BachmayrNearOptimal,bachmayr2014adaptive} can be adapted to the setting of separate low-rank representations for each temporal basis index as in \eqref{eq:temporallowrank2d}. In particular, we obtain analogous results for basis coarsening and rank reduction procedures as in the case of a single low-rank representation in \cite{BachmayrNearOptimal}.

In Section \ref{sec:low-rank-preconditioning}, we analyze new low-rank space-time diagonal preconditioners based on exponential sum approximations. By these low-rank approximations, we account for the lack of separability of the scaling factors arising in the multidimensional Riesz bases. This is also a central issue in low-rank solvers for elliptic problems \cite{bachmayr2014adaptive}, but for space-time formulations we need a new construction. We make crucial use of the structure of separate tensor representations for each temporal basis index,  allowing us to approximate the diagonal entries of the preconditioner independently for each time index. These approximations are then realized by exponential sums based on the inverse Laplace transform of $s\mapsto \sqrt{s}/(s + a)$, where $a > 0$ depends on the corresponding time index.

In Section \ref{sec:Apply}, we use the new low-rank approximations of diagonal scalings  in a scheme for constructing sparse and low-rank approximations of the basis representations of the operators in the space-time formulation. These adaptive operator approximations are subsequently the main constituents in obtaining residual approximations in our adaptive scheme. We devise new techniques for the basis representations of temporal derivatives, which here involve interaction between different low-rank representations, and for the representation of the trace at the initial time. In particular the latter causes new difficulties compared to the elliptic case, since it leads to an additional coupling between operator representation ranks and maximum wavelet levels of the activated basis functions.

In Section \ref{sec:Solver}, the resulting residual approximation scheme is used as the central component of an adaptive method based upon an approximate Richardson iteration in sequence space applied to the least-squares form of the space-time variational formulation. As a first step, we show convergence of the method in the natural norm \eqref{eq:Xnorm2d} to the exact solution of the parabolic PDE. We then analyze the complexity of the method concerning the total number of required elementary operations under typical assumptions on the approximability of solutions. The first crucial ingredient in such complexity bounds are near-optimal estimates for discretization index set sizes and low-rank representation ranks of intermediate results that are provided by the low-rank recompression and basis coarsening procedures. The second are bounds on the complexity of operator approximations. Crucially, to estimate the cumulative effect of several steps in the iterative scheme, we need to deal with the interactions between the sets of active basis indices and the low-rank representations of operator approximations.

Finally, in Section \ref{sec:Numexp} we give some first numerical illustrations of the new method for large $d$, and in Section \ref{sec:Conclusion} summarize our conclusions and give an outlook on further open questions.

\subsection{Notation}

In the remainder of this work, to simplify notation we denote by $\norm{\cdot}$ the $\ell_2$-norm on the respective index set and by $\langle \cdot, \cdot \rangle$ the corresponding inner product. By $A \lesssim B$, we denote $A \leq C B$ with a constant $C>0$; $A \gtrsim B$ is defined as $B \lesssim A$, and $A \eqsim B$ as $A\lesssim B$ and $A \gtrsim B$.

\section{Preliminaries}
\label{sec:Prelim}

\subsection{Problem formulation}

Let $V, H$ be separable Hilbert spaces such that $V$ is densely embedded in $H$. Identifying $H$ with its dual $H^\prime$, we obtain the Gelfand triple $V \hookrightarrow H \hookrightarrow V^\prime$.

Let $0 < T < \infty$ and $I = [0,T]$. We denote for a.e. $t \in I$ by $a(t;\cdot,\cdot)$ a bilinear form on $V \times V$ such that for all $v,\hat{v} \in V$ the function $t \mapsto a(t;v,\hat{v})$ is measurable on $I$. Furthermore, for a.e.\ $t \in I$ we assume
\begin{subequations}
\begin{align}
	\lvert &a(t;v,\hat{v}) \rvert \leq a_{\max}\lVert v \rVert_V \lVert \hat{v} \rVert_V \quad \text{for all } v,\hat{v} \in V \qquad (\text{boundedness}), \\ 
	&a(t;v,v) + \lambda_0 \norm{v}_H^2 \geq a_{\min} \norm{v}_V^2 \quad \text{for all } v \in V \qquad (\text{coercivity}) \label{eq:coercivity} 
\end{align}
\end{subequations}
for some constants $0 < a_{\min} \leq a_{\max} < \infty$ and $\lambda_0 \in \mathbb{R}$. For a.e.\ $t \in I$, we can thus define $A(t) \in \cL(V,V^\prime)$ by
\begin{align*}
	\inp{A(t)v}{\hat{v}}_{V^\prime \times V} = a(t;v,\hat{v})  .
\end{align*}

We consider linear parabolic problems of the form
\begin{equation}\label{eq:parabolic-def}
	\begin{aligned}
		\partial_t u(t) + A(t) u(t) &= g(t) \quad & \text{in } V^{\prime}, \\
		u(0) &= h  \quad & \text{in } H	,
	\end{aligned}
\end{equation}
for given $g \in L_2(I;V^\prime)$ and $h \in H$.

We use the classical space-time weak formulation of the parabolic problem \eqref{eq:parabolic-def} analyzed in the context of adaptive methods in \cite{SpaceTimeAdaptive}; see also \cite[Ch.~XVIII]{MR1156075}.
The trial space for this formulation reads
\begin{align}\label{eq:ansatzX}
	\cX = L_2(I;V) \cap H^1(I;V^{\prime})
	= \left\{ v \in L_2(I;V) :  \partial_t v \in L_2(I;V^\prime) \right\},
\end{align}
the test space is
\begin{align}\label{eq:testY}
	\cY &= L_2(I;V) \times H . 
\end{align}
For $v \in \cX$ and $(w_1,w_2) \in \cY$, the corresponding norms are given by
\begin{equation}\label{eq:normdefs}
\begin{aligned}
	\norm{v}_{\cX} &= \left( \norm{v}_{L^2(I;V)}^2 +\lrnorm{\partial_t v}_{L_2(I;V^\prime)}^2 \right)^{\frac{1}{2}} ,  \\
	\norm{(w_1,w_2)}_{\cY} &= \left( \norm{w_1}_{L^2(I;V)}^2 + \norm{w_2}_H^2 \right)^{\frac{1}{2}}  .
\end{aligned}
\end{equation}
The space-time weak formulation of \eqref{eq:parabolic-def} for $u \in \cX$ now reads
\begin{equation}\label{eq:varform}
	 b(u,v) = f(v) \quad \text{for all } v \in \cY,
\end{equation}
with the bilinear form $b \colon \cX \times \cY \to \R$ given by
\begin{align}\label{eq:bilinear-parabolic}
	b(v,w) = \int\limits_I \inp{\partial_t v(t)}{w_1(t)}_{V^\prime \times V} + a\bigl(t;v(t), w_1(t)\bigr) \, dt + \inp{v(0)}{w_2}_H
\end{align}
and the functional $f \colon \cY \to \R$ given by
\begin{align*}
	f(w) = \int\limits_I \inp{g(t)}{w_1(t)}_H  \, dt + \inp{h}{w_2}_H . 
\end{align*}

\begin{theorem}
	The operator $B \in \cL(\cX,\cY')$ defined by $(Bv)(w) = b(v,w)$ with $b$ as in \eqref{eq:bilinear-parabolic}, $\cX$ as in \eqref{eq:ansatzX} and $\cY$ as in \eqref{eq:testY} is boundedly invertible.
\end{theorem}

For a proof of this theorem and explicit bounds of the norms $\norm{B}$ and $\norm{B^{-1}}$ we refer to \cite{SpaceTimeAdaptive}. Our work specifically addresses problems with coercive spatial part, such as $A(t) = -\Delta$, where \eqref{eq:coercivity} holds with $\lambda_0 = 0$. Note that for this class of problems, the bounds for both $B$ and $B^{-1}$ are in particular independent of the final time $T$, as can be seen from \cite[Thm.~5.1]{SpaceTimeAdaptive}. With a modified norm on $\cX$ that incorporates an initial trace at $t=0$, as shown in \cite{L2ProjectionParabolic} this $T$-independence also holds when $\lambda_0 > 0$, but in what follows we work with the standard norms as in \eqref{eq:normdefs}.

Our construction of adaptive methods is based on an equivalent infinite matrix representation of the problem \eqref{eq:varform} in terms of suitable wavelet Riesz bases. 
Let $\{ \theta_\nu \}_{\nu \in \vee_\rt}$ be a Riesz basis of $L_2(I)$ such that $\{ \norm{\theta_\nu}_{H^1(I)}^{-1} \theta_\nu \}_{\nu \in \vee_\rt}$ is a Riesz basis of $H^1(I)$, and let $\{ \Psi_\nu \}_{\nu \in \vee_\rx}$ be a Riesz basis of $H$ such that $\{ \norm{\Psi_\nu}_V^{-1} \Psi_\nu \}_{\nu \in \vee_\rx}$ is a Riesz basis of $V$ and $\{  \norm{\Psi_\nu}_{V'}^{-1}  \Psi_\nu \}_{\nu \in \vee_\rx}$ is a Riesz basis of $V'$. 

Let $\bar S^\cX$ and $\bar S^\cY$ be real sequences with positive entries on $\vee = \vee_\rt\times \vee_\rx$ and $\vee_\rx$, respectively, that satisfy
\begin{equation}\label{eq:scalinggeneral}
  \bar S^{\cX}_{\nu_\rt, \nu_\rx} \eqsim \Bigl( \norm{ \Psi_{\nu_\rx} }_V^2 + \norm{\theta_{\nu_\rt}}_{H^1(I)}^2 \norm{\Psi_{\nu_\rx}}_{V'}^2  \Bigr)^{-\frac12}, \qquad
  \bar S^{\cY}_{\nu_\rx} \eqsim  \norm{\Psi_{\nu_\rx}}_V^{-1}\,
\end{equation}
uniformly for all $\nu_\rt\in\vee_\rt$, $\nu_\rx \in \vee_\rx$.
Then, as noted in \cite{SpaceTimeAdaptive}, 
\[
  \Sigma_\cX = \Bigl\{  (t,x) \mapsto \bar S^\cX_{\nu_\rt,\nu_\rx}  \theta_{\nu_\rt} (t)\,\Psi_{\nu_\rx}(x)  \colon \nu_\rt \in \vee_\rt, \nu_\rx \in \vee_\rx \Bigr\}
\] 
is a Riesz basis of $\cX$ and
\begin{multline*}
 \Sigma_\cY =  \left\{ (t,x) \mapsto  \left( \bar S^\cY_{\nu_\rx} \theta_{\nu_\rt} (t)\,\Psi_{\nu_\rx}(x) , 0 \right) \colon \nu_\rt \in \vee_\rt, \nu_\rx \in \vee_\rx  \right\}
  \cup \Bigl\{  x\mapsto \bigl( 0, \Psi_{\nu_\rx}(x)\bigr) \colon \nu_\rx \in \vee_\rx  \Bigr\}
\end{multline*}
is a Riesz basis of $\cY$. With $\vee' = \vee \cup \vee_\rx$, we introduce the notation $\Sigma_\cX = \{  X_\nu \}_{\nu \in \vee}$ and $\Sigma_\cY = \{ Y_\mu \}_{\mu \in \vee'}$ for the elements of these collections.
As a consequence, defining
\[
   \mathbf{\bar B} = \begin{bmatrix} \mathbf{\bar B}_1 \\ \mathbf{\bar B}_2  \end{bmatrix}, \quad \mathbf{\bar B}_1 =  \bigl(  b(X_\nu, Y_{\nu'})  \bigr)_{\nu \in \vee, \nu' \in \vee }, \;\mathbf{\bar B}_2 =  \bigl(  b(X_\nu, Y_{\nu_\rx})  \bigr)_{\nu \in \vee, \nu_\rx \in \vee_\rx },
\]
and identifying $\ell_2(\vee')$ with $\ell_2(\vee) \times \ell_2(\vee_\rx)$,
one has that $\mathbf{\bar B}$ defines an isomorphism from $\ell_2(\vee)$ to $\ell_2(\vee')$.
Moreover, with
\[
   \mathbf{\bar f} = \begin{bmatrix} \mathbf{\bar f}_1 \\ \mathbf{\bar f}_2  \end{bmatrix} , \quad \mathbf{\bar f}_1 = \bigl( f (Y_{\nu'})  \bigr)_{\nu' \in \vee}, \;  \mathbf{\bar f}_2 = \bigl( f (Y_{\nu_\rx})  \bigr)_{\nu_\rx \in \vee_\rx}
\]
we have $\mathbf{\bar f} \in \ell_2(\vee')$.
The infinite linear system of equations $\mathbf{\bar B} \mathbf{\bar u} = \mathbf{\bar f}$ thus has a unique solution $\mathbf{\bar u}  \in \ell_2(\vee)$, and $u$ solving \eqref{eq:varform} can be represented as
$ u = \sum_{\nu \in \vee} \mathbf{\bar u}_\nu X_\nu$. 

Let us now consider the structure of $\mathbf{\bar B}$. To this end, we introduce the diagonal scaling matrices
\begin{equation}\label{eq:scaling-matrices}
	\begin{aligned}
	\mathbf{\bar D}_{\cX} &= \bigl(  \bar S^{\cX}_{\nu_\rt, \nu_\rx} \delta_{(\nu_\rt,\nu_\rx),(\nu_\rt',\nu_\rx')}  \bigr)_{ (\nu_\rt,\nu_\rx), (\nu_\rt',\nu_\rx') \in \vee }, \\ 
	\mathbf{\bar D}_{\cY} &= \bigl(  \bar S^{\cY}_{\nu_\rx} \delta_{(\nu_\rt,\nu_\rx,\tilde\nu_\rx),(\nu_\rt',\nu_\rx',\tilde\nu_\rx')}   \bigr)_{(\nu_\rt,\nu_\rx, \tilde\nu_\rx), (\nu_\rt',\nu_\rx',\tilde\nu_\rx) \in \vee'}\,.
	\end{aligned}
\end{equation}
Using the bilinearity of $b$, as well as the definition of the bases of $\mathcal{X}$ and $\mathcal{Y}$, we can rewrite $\mathbf{\bar B}$ in the form
\begin{align}\label{eq:operator-Bhat}
	\mathbf{\bar B} =   \begin{bmatrix}
		\mathbf{\bar D}_\cY & 0 \\
		0 & \bI_\rx
	\end{bmatrix}
	\begin{bmatrix}
		\bT \\ \bT_0
	\end{bmatrix}
	\mathbf{\bar D}_\cX ,
\end{align}
where $\bT$ and $\bT_0$ are given by
\begin{align}\label{eq:opertor-T1T2-def}
	\begin{split}
		\bT &= \Bigl( b\big(\theta_{\nu_\rt}\otimes \Psi_{\nu_\rx}, (\theta_{\nu^\prime_\rt} \otimes \Psi_{\nu^\prime_\rx},0) \big) \Bigr)_{(\nu_\rt,\nu_\rx),(\nu^\prime_\rt,\nu^\prime_\rx) \in \vee}, \\
		\bT_0 &= \Bigl( b\big(\theta_{\nu_\rt} \otimes \Psi_{\nu_\rx}, (0,\Psi_{\nu^\prime_\rx}) \big) \Bigr)_{(\nu_\rt,\nu_\rx) \in \vee, \nu^\prime_\rx \in \vee_\rx} ,
	\end{split}
\end{align}
and where $\mathbf{I}_\rx = (\delta_{\nu,\nu'})_{\nu,\nu' \in \vee_\rx}$.

\begin{remark}\label{rem:alternativeform}
 Alternative weak formulations of \eqref{eq:parabolic-def} with choices of trial and test spaces different from the ones in \eqref{eq:varform} are possible, such as the following one (see for example \cite{CheginiStevenson:11}): find $u \in L_2(I; V)$ such that
\[
	 \int\limits_I -\inp{u(t)}{\partial_t v(t)} + a\bigl(t;u(t), v(t)\bigr)  \, dt 
      =
	 \int\limits_I \inp{g(t)}{v(t)}_H  \, dt + \inp{h}{v(0)}_H  
\]
for all $v \in L_2(I;V) \cap \{ w \in H^1(I;V') \colon w(T) = 0 \}$.
This formulation can be treated by a straightforward adaptation of the techniques developed in this paper, and in fact alleviates some of the technical difficulties associated to the treatment of initial values (see Section \ref{sec:apply-initial-value}). However, since \eqref{eq:varform} yields stronger convergence of approximate solutions -- in the norm of $\cX$ rather than $L_2(I;V)$ -- we use the weak formulation \eqref{eq:varform}.
\end{remark}

\subsection{Second-order problems}\label{sec:scaling-matrix}

In this work, we focus on the case where $A(t)$ is a second-order elliptic operator on $\Omega = (0,1)^d$, where $V = H^1_0(\Omega)$, $H = L_2(\Omega)$, $V' = H^{-1}(\Omega)$, and $A(t) \colon V \to V'$ is given by
\begin{align}\label{eq:second-order-elliptic}
	A(t) v = - \nabla_x \cdot M(t) \nabla_x v + q(t) \cdot \nabla_x v + c(t) v, \quad v \in V,
\end{align}
with suitable $M(t)\in \R^{d\times d}$, $q(t) \in \R^d$, $c(t) \in \R$. 

For this problem posed on a product domain, we use a particular construction of spatial wavelet Riesz bases of $V$ and $V'$ with tensor product structure, based on the following assumptions. 
With a countable index set $\vee_1$, let $\{ \psi_\nu \}_{\nu \in \vee_1}$ be an \emph{orthonormal basis} of $L_2(0,1)$ that is at the same time a Riesz basis of $H^1_0(0,1)$ and of $H^{-1}(0,1)$ with the respective normalizations. For each $\nu \in \vee_1$, we denote by $\abs{\nu}$ the wavelet level of the basis function $\psi_\nu$.
We then set $\vee_\rx = \bigtimes_{i=1}^d \vee_1$ as well as $\Psi_\nu = \bigotimes_{i=1}^d \psi_{\nu_i}$ so that
\[
 \norm{\Psi_\nu}_V = \Bigl(\sum\limits_{i=1}^d \norm{\psi_{\nu_i}}_{H^1_0(0,1)}^2\Bigr)^{\frac{1}{2}}, \quad \nu \in \vee_\rx\,.
\]
We then easily verify the following observations made in similar form in \cite[Sec.~2]{Dijkema:09} and \cite[Sec.~8]{SpaceTimeAdaptive}.

\begin{prop}\label{prop:riesz}
   For $\{ \Psi_\nu\}_{\nu \in \vee_\rx}$ on $\Omega = (0,1)^d$ with $V = H^1_0(\Omega)$, $H = L_2(\Omega)$ as above, which in particular is an orthonormal basis of $H$, we have the following:
   \begin{enumerate}[{\rm(i)}]
	\item $\{\norm{\Psi_\nu}_V^{-1} \Psi_\nu\}_{\nu \in \vee_\rx}$ is a Riesz basis of $V$, where with constants independent of $d$,
	\[  \norm{ \bv } \eqsim \Bignorm{ \sum_{\nu \in \vee_\rx} v_\nu \norm{\Psi_\nu}_V^{-1} \Psi_\nu }_V \quad\text{for all $\bv = (v_{\nu})_{\nu \in \vee_\rx} \in \ell_2(\vee_\rx)$.} \]
	\item\label{it:Vprime-norm} With constants independent of $d$,
	\[
		\norm{ \Psi_{\nu_\rx}}_{V'} \eqsim  \norm{ \Psi_{\nu_\rx} }^{-1}_{V} \quad\text{for $\nu_\rx \in \vee_\rx$.}
	\]
	\item $\{\norm{\Psi_\nu}_V \Psi_\nu\}_{\nu \in \vee_\rx}$ is a Riesz basis of $V'$, where with constants independent of $d$,
	\[  \norm{ \bv } \eqsim \Bignorm{ \sum_{\nu \in \vee_\rx} v_\nu \norm{\Psi_\nu}_V \Psi_\nu }_{V'} \quad\text{for all $\bv = (v_{\nu})_{\nu \in \vee_\rx} \in \ell_2(\vee_\rx)$.} \]
   \end{enumerate}
\end{prop}

Note that for the $d$-independence of the constants in Proposition \ref{prop:riesz}, it is crucial that we start from $L_2(0,1)$-\emph{orthonormal} univariate wavelets $\{ \psi_\nu \}_{\nu \in \vee_1}$, since (as observed in \cite{Dijkema:09}) otherwise the constants would depend exponentially on $d$, eventually leading to exponential scaling of the computational costs of our method. The same restriction in the choice of univariate wavelets applies to the adaptive low-rank methods for elliptic problems treated in \cite{bachmayr2014adaptive}.

Although our adaptive solver can be formulated for quite general choices of basis functions satisfying the above requirements, our complexity analysis requires a further restriction. In what follows, we assume that both $\{ \theta_{\nu_\rt}\}_{\nu_\rt\in\vee_\rt}$ and $\{ \psi_\nu \}_{\nu\in \vee_1}$ are spline wavelet-type bases that are orthonormal in $L_2(0,T)$ and $L_2(0,1)$, respectively, with sufficiently many vanishing moments and such that
\[
   \operatorname{diam} \supp \theta_{\nu_\rt} \lesssim 2^{-\abs{\nu_\rt}}, \quad \operatorname{diam} \supp \psi_{\nu} \lesssim 2^{-\abs{\nu}}
\]
uniformly in $\nu_\rt\in \vee_\rt$ and $\nu \in \vee_1$. These conditions are satisfied by Donovan-Geronimo-Hardin multiwavelets \cite{multiwavelets}. 

 \begin{remark}\label{rem:otherwavelets}  Note that other types of wavelets could be used. For the spatial basis, the most crucial property is $L_2$-orthonormality and sufficient regularity of the univariate wavelets, which is also provided, for example, by standard Daubechies wavelets with boundary adaptation. Our analysis can be applied in this case, but better compressibility of operators and easier computation of matrix entries are achieved with wavelets that are in addition piecewise polynomial. 
 Concerning the temporal basis functions, our analysis can be adapted to other wavelets that provide a sufficiently sparse representation of the identity and of the time derivative operator. In particular, in combination with the alternative variational form discussed in Remark \ref{rem:alternativeform}, the wavelets constructed specifically for this purpose in \cite{CheginiStevenson:11} could also be used; these have the disadvantage, however, of vanishing at $t=0$. This means that the initial condition needs to be resolved by strong adaptive refinement for small times, which in the present setting would lead to unfavorable quantitative performance. \end{remark}

We now introduce sequences $\bar S^\cX$ and $\bar S^\cY$ as in \eqref{eq:scalinggeneral} that are suitable for our purposes.
First, we take $\bar S^\cY_{\nu} = \norm{\Psi_{\nu}}_V^{-1}$ for $\nu \in \vee_\rx$, so that the second relation in \eqref{eq:scalinggeneral} holds with equality.
Thus with 
\begin{equation}
	\mathbf{\bar D} = \bigl(\norm{\Psi_\nu}_V^{-1} \delta_{\nu,\mu} \bigr)_{\nu,\mu \in \vee_\rx} \label{eq:S-def}
\end{equation}
and $\mathbf{I}_\rt = (\delta_{\nu,\nu'})_{\nu,\nu' \in \vee_\rt}$, we have 
\begin{equation}
	\mathbf{\bar D}_\mathcal{Y} = \mathbf{I}_\rt \otimes \mathbf{\bar D}\ . \label{eq:def-DY}
\end{equation}   
Concerning $\bar S^\cX$, for $\theta_{\nu_\rt} \otimes \Psi_{\nu_\rx} \in \cX$, with Proposition \ref{prop:riesz}\eqref{it:Vprime-norm} we obtain
\begin{align*}
	\left( \norm{\Psi_{\nu_\rx}}_{H_0^1(\Omega)}^2 + \norm{\theta_{\nu_\rt}}_{H^1(I)}^2 \norm{\Psi_{\nu_\rx}}_{H^{-1}(\Omega)}^2 \right)^{\frac{1}{2}}   
	& \eqsim \left( \norm{\Psi_{\nu_\rx}}_{H_0^1(\Omega)}^2 + \norm{\theta_{\nu_\rt}}_{H^1(I)}^2 \norm{\Psi_{\nu_\rx}}_{H_0^1(\Omega)}^{-2} \right)^{\frac{1}{2}} \\
	&\eqsim \frac{\norm{\Psi_{\nu_\rx}}_{H_0^1(\Omega)}^2 + \norm{\theta_{\nu_\rt}}_{H^1(I)}}{\norm{\Psi_{\nu_\rx}}_{H_0^1(\Omega)}}
\end{align*}
with constants independent of $d$. Thus for
\begin{equation}\label{eq:omega-t-x}
	\bar S^\cX_{\nu_\rt,\nu_\rx} =  \frac{\norm{\Psi_{\nu_\rx}}_{H_0^1(\Omega)}}{\norm{\Psi_{\nu_\rx}}_{H_0^1(\Omega)}^2 + \norm{\theta_{\nu_\rt}}_{H^1(I)}} ,
\end{equation}
the first relation \eqref{eq:scalinggeneral} also holds true with constants independent of $d$.

In what follows, we assume $A(t)$ to be time-independent, that is, $A(t)=A$ for some $A\colon V\to V'$; under this assumption, by $L_2$-orthonormality of the wavelet bases, the operator $\bT$ has the form
\begin{equation}\label{eq:opstruct}
	\bT = \bI_\rt \otimes \bT_\rx + \bT_\rt \otimes \bI_\rx .
\end{equation}
We define
\begin{align}
	\mathbf{\bar B}_{\rt} &= \mathbf{\bar D}_\cY (\bT_\rt \otimes \bI_\rx) \mathbf{\bar D}_\cX \label{eq:bAct}, \\
	\mathbf{\bar B}_{\rx} &= \mathbf{\bar D}_\cY  (\bI_\rt \otimes \bT_\rx) \mathbf{\bar D}_\cX \label{eq:bAcx}. 
\end{align} 

To avoid technicalities, we restrict ourselves to the case where $A(t) = A$ is a second-order elliptic operator, especially on the model case of the heat equation, where $A(t) = -\Delta$. The extension to the case of general second-order operators with constant coefficients $M$, $q$ and $c$ in \eqref{eq:second-order-elliptic} is then immediate from the results of \cite{bachmayr2014adaptive}, see Remark \ref{rem:generalellipt} for further details. 
Coefficients with spatial and temporal variability can be treated with our approach, where a convergent method can be obtained by a suitably adapted operator compression.
However, depending on the particular assumptions on the coefficients, the complexity analysis of the method can then become substantially more difficult.
To give a concrete example, let us consider replacing coefficients depending smoothly on time by a polynomial approximation in time. This leads again to an operator representation that is a  sum of Kronecker products as in \eqref{eq:opstruct}, but with a number of terms depending on the error tolerance, and the additional terms require a more involved compression procedure. This will be considered in more detail in future work.

\section{Adaptive Low-Rank Approximations}\label{sec:ALRbasic}

\subsection{Hierarchical tensors}
The low-rank approximations considered here rely on \emph{hierarchical tensors} \cite{HackbuschKuehn:09,Hackbusch:12tensorspaces} as a particular format for low-rank representations of higher-order tensors. Since we will be exclusively interested in tensors on the index set $\vee_\rx = \vee_1\times\cdots\times\vee_1$, we state the following basic results for tensors indexed by $\vee_\rx$. For a detailed treatment of hierarchical tensor representations, we refer to \cite[Ch.~11]{Hackbusch:12tensorspaces} and \cite{B:23}.

The starting point for defining the hierarchical format is a \emph{binary dimension tree} $\mathbb{T}_d$, which is a hierarchy of subsets of $\alpha^* = \{1,\ldots,d\}$, which is the root element of $\mathbb{T}_d$. 
For $\alpha \subset \alpha^*$, we write $\compl{\alpha} = \alpha^* \setminus \alpha$.
Starting with the root element $\alpha^* \in \mathbb{T}_d$, for each $\alpha \in \mathbb{T}_d$ with $\#\alpha >1$, there exist precisely two disjoint children $\alpha_1, \alpha_2 \in \mathbb{T}_d$ such that $\alpha = \alpha_1 \cup \alpha_2$. Consequently, $\{1\},\ldots,\{d\} \in \mathbb{T}_d$; these elements are referred to as \emph{leaves}. Moreover, $\mathbb{T}_d$ has tree structure in the sense that for each $\alpha,\beta \in \mathbb{T}_d$ with $\alpha \neq \beta$, either $\alpha \subset \beta$ or $\beta\subset \alpha$ or $\alpha\cap\beta = \emptyset$. 

Hierarchical tensor formats are based on low-rank representations of \emph{matricizations} of a given tensor, which result from arranging its entries in matrix form.
For $\nu \in \vee_\rx$, let $\nu_\alpha = ( \nu_i )_{i \in \alpha}$.
The \emph{$\alpha$-matricization} of $\bv \in \ell_2(\vee_\rx)$ is given by
\[
  \operatorname{mat}_\alpha(\bv) =  \bigl( \bv_{\nu} \bigr)_{\nu_\alpha \in \vee_1^{\#\alpha}, \nu_{\compl{\alpha}} \in \vee_1^{d-\#\alpha}} \,,
\]
and we define $\rank_\alpha(\bv) = \rank \operatorname{mat}_\alpha(\bv)$ with the abbreviation $\rank_i(\bv) = \rank_{\{i\}}(\bv)$ for $i=1,\ldots,d$.
\begin{figure}
	\begin{subfigure}{0.45\textwidth}
		\begin{tikzpicture}[level distance=1.5cm,
			level 1/.style={sibling distance=3cm},
			level 2/.style={sibling distance=1.25cm}]
			\tikzstyle{every node}=[]
			\node[draw=none] (RootShift) {};
			
			\node (Root) at ([xshift=3cm]RootShift)  {$\alpha^* = \{1,2,3,4\}$}
			child {
				node {$\alpha_1^* = \{1,2\}$} 
				child { node {$\{1\}$} }
				child { node {$\{2\}$} }
			}
			child {
				node {$\alpha_2^* = \{3,4\}$}
				child { node {$\{3\}$} }
				child { node {$\{4\}$} }
			};
		\end{tikzpicture}
		\caption*{Dimension tree $\mathbb{T}_d$}
	\end{subfigure}
	\hfill
	\begin{subfigure}{0.45\textwidth}
		\begin{tikzpicture}[level distance=1.5cm,
			level 1/.style={sibling distance=1.25cm},
			level 2/.style={sibling distance=1.25cm}]
			\tikzstyle{every node}=[draw,circle]
			
			\node (Root) {}
			child { node {} edge from parent node[left,draw=none] {$\big[\{1\}\big]$} }
			child { node {} edge from parent node[right,draw=none] {$\big[\{2\}\big]$} };
			
			\node (Root2) at ([xshift=3cm]Root) {}
			child { node {} edge from parent node[left,draw=none] {$\big[\{3\}\big]$} }
			child { node {} edge from parent node[right,draw=none] {$\big[\{4\}\big]$} };
			
			\node[rectangle,draw=none] (RootOriginal) at ($(Root)+(1.5,1.6)$) {$\big[\alpha_1^*\big] = \big[\alpha_2^*\big]$};
			\draw  (Root.north) to [out=67.5,in=180] ($(RootOriginal) + (0,-0.35)$) to [out=0, in=112.5] (Root2.north);
		\end{tikzpicture}
		\caption*{Effective edges $\mathbb{E}_d$}
	\end{subfigure}
	\caption{Example of a binary dimension tree $\mathbb{T}_d$ and its corresponding effective edges $\mathbb{E}_d$ in dimension $d=4$.}
	\label{fig:dimension-tree-edges}
\end{figure}
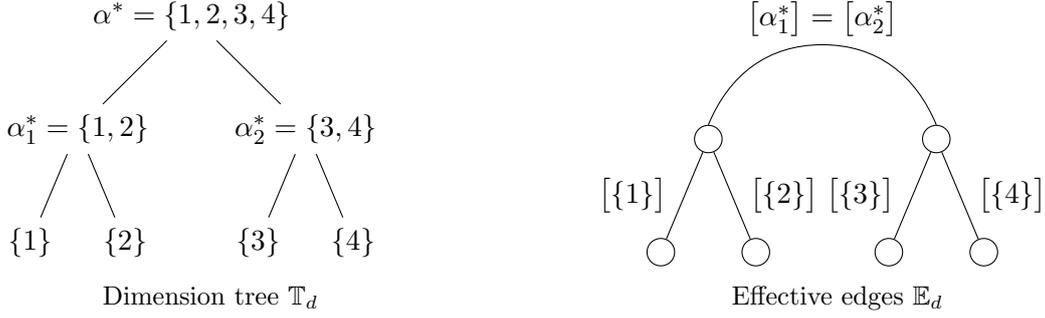

\begin{definition}\label{def:tensorformat}
Let a fixed binary dimension tree $\mathbb{T}_d$ be given.
\begin{enumerate}[{\rm(a)}]
\item The set of \emph{effective edges} of $\mathbb{T}_d$ are the pairs
\[
  \mathbb{E}_d = \bigl\{  \{ \alpha, \compl{\alpha} \} \colon \alpha \in \mathbb{T}_d \setminus \{ \alpha^*\} \bigr\} ,
\]
For each $e \in \mathbb{E}_d$, we define the \emph{representer} $[e]$ as the $\alpha \in e$ such that $\alpha \in \mathbb{T}_d$; if this element is not unique, we make an arbitrary choice of $[e]$.
\item The set of tensors of hierarchical rank at most $\mathsf{r} = (r_e)_{e \in \mathbb{E}_d}$ with $r_e \in \N_0 \cup \{ \infty\}$ for $e \in \mathbb{E}_d$ is then defined as
\[
  \mathcal{H}(\mathsf{r}) = \bigl\{  \bv \in \ell_2(\vee_\rx) \colon  \text{$\rank_{[e]}(\bv) \leq r_e$ for all $e \in \mathbb{E}_d$} \bigr\} .
\]
\item For $\bv \in \ell_2(\vee_\rx)$, we define a notation for the hierarchical rank vector in terms of the matrix ranks of matricizations by
\[
	\rank(\bv) = \bigl(  \rank_{[e]}(\bv) \bigr)_{e \in \mathbb{E}_d}.
\]
\end{enumerate}
\end{definition}

The effective edges $\mathbb{E}_d$ in the above definition correspond to the matricizations that define the hierarchical tensor format associated to $\mathbb{T}_d$. Here matricizations that are identical up to transposition (which correspond to the same rank constraint) are treated as a single effective edge. An illustration of a binary dimension tree $\mathbb{T}_d$ and its corresponding effective edges $\mathbb{E}_d$ for dimension $d=4$ is presented in Figure \ref{fig:dimension-tree-edges}. For the number of effective edges associated to an arbitrary binary dimension tree, one readily verifies $\#\mathbb{E}_d = 2d-3$.

For each $\alpha \in \mathbb{T}_d \setminus \{ \alpha^*\} $, there exists an orthonormal system $\{ \mathbf{U}^\alpha_{k} \}_{k = 1,\ldots, \rank_{\alpha}(\bv)}$ in $\ell_2(\vee_1^{\#\alpha})$, a so-called \emph{mode frame}, that is an orthonormal basis of $\overline{\operatorname{range} \operatorname{mat}_\alpha(\bv)}$.
For $\alpha \in \mathbb{T}_d \setminus \{ \alpha^*\} $ with $\#\alpha >1$, for the children $\alpha_1,\alpha_2 \in \mathbb{T}_d$ of $\alpha$, one has
\begin{equation}\label{eq:nestedness}
  \mathbf{U}^{\alpha}_k = \sum_{\ell_1 = 1}^{\rank_{\alpha_1}\!(\bv)}\;\sum_{\ell_2 = 1}^{\rank_{\alpha_2}\!(\bv)} \mathbf{b}^\alpha_{k, \ell_1, \ell_2} \mathbf{U}^{\alpha_1}_{\ell_1} \otimes \mathbf{U}^{\alpha_2}_{\ell_2}  	 , \quad k = 1,\ldots, \rank_{\alpha}(\bv),	
\end{equation}
with the \emph{transfer tensors} $\mathbf{b}^\alpha$ given by $\mathbf{b}^\alpha_{k, \ell_1, \ell_2} = \langle \mathbf{U}^{\alpha}_k , \mathbf{U}^{\alpha_1}_{\ell_1} \otimes \mathbf{U}^{\alpha_2}_{\ell_2}\rangle$. The nestedness property \eqref{eq:nestedness} implies the restriction $\rank_\alpha(\bv) \leq \rank_{\alpha_1}(\bv)\, \rank_{\alpha_2}(\bv)$ on the possible ranks. With respect to the fixed dimension tree $\mathbb{T}_d$, we denote the set of feasible rank vectors in $(\N_0 \cup \{ \infty\} )^{\mathbb{E}_d}$ of hierarchical tensors by $\mathcal{R}$.

For $\bv \in \ell_2(\vee_\rx)$, in a first step we have the decomposition
\[
    	 \bv =  \sum_{\ell_1 = 1}^{\rank_{\alpha_1^*}(\bv)}\; \sum_{\ell_2 = 1}^{\rank_{\alpha_2^*}(\bv)} \mathbf{b}^{\alpha^*}_{\ell_1, \ell_2} \mathbf{U}^{\alpha_1^*}_{\ell_1}\otimes \mathbf{U}^{\alpha_2^*}_{\ell_2} 
\]
where $\alpha_1^*, \alpha_2^*$ are the children of the root element $\alpha^*$ and where $\mathbf{b}^{\alpha^*} = \bigl( \langle \bv , \mathbf{U}^{\alpha^*_1}_{\ell_1}\otimes  \mathbf{U}^{\alpha^*_1}_{\ell_1} \rangle \bigr)_{\ell_1, \ell_2}$; note that in this particular case, $\rank_{\alpha_1^*}(\bv) = \rank_{\alpha_2^*}(\bv)$, since the children of the root share (up to transposition) the same matricization.
Applying \eqref{eq:nestedness} recursively, we obtain a representation of $\bv$ by $\mathbf{b}^{\alpha^*}$, by the transfer tensors $\mathbf{b}^{\alpha}$ for all $\alpha \in \mathbb{T}_d \setminus \{ \alpha^*\}$ with $\#\alpha >1$, and by the mode frames $\{ \mathbf{U}_k^{\{i \} } \}_{k=1,\ldots, \rank_i(\bv)}$ for $i =1,\ldots, d$. All operations on hierarchical tensors are then performed exclusively on these representation components. 
Note that hierarchical tensors that arise as intermediate results in computations can also be given in terms of non-orthogonal mode frames and transfer tensors with rank parameters larger than the actual hierarchical rank; in this case, we speak of \emph{representation ranks}.

\subsection{Adaptive methods for parabolic problems}\label{sec:recompress-coarsen}

In low-rank methods based on space-time variational formulations, the time variable can be treated as a separate tensor mode in a low-rank decomposition as in \cite{Andreev:15,Boiveau:19}. However, as noted above, this has disadvantages both concerning the basic approximability of solutions and concerning some algorithmic aspects that we comment on in detail in Remark \ref{rem:fullsep}.
Instead, we combine sparse approximation in the time variable with adaptive low-rank approximation in the spatial variables.

The main idea is to use low-rank representations in the spatial variables independently for each time basis index, which leads to approximations with the following structure: 
Denoting by $\mathbf{e}_{\nu_\rt} \in \ell_2(\vee_\rt)$ the Kronecker vector with $\mathbf{e}_{\nu_\rt} = (\delta_{\nu_\rt,\mu_\rt})_{\mu_\rt \in \vee_\rt} $, any $\bu \in \ell_2(\vee) = \ell_2(\vee_\rt \times \vee_\rx)$ can be written uniquely in the form 
\begin{equation}\label{eq:parabolic-lowrank-format}
	\bu = \sum\limits_{\nu_\rt \in \vee_\rt} \mathbf{e}_{\nu_\rt} \otimes \bu_{\nu_\rt} ,
\end{equation}
where $\bu_{\nu_\rt} \in \ell_2(\vee_\rx)$ are to be represented in hierarchical tensor format independently for each $\nu_\rt \in \vee_\rt$, with a shared fixed dimension tree $\mathbb{T}_d$. We say that $\bu$ has \emph{spatial components} $(\bu_{\nut})_{\nu_\rt\in\veet}$ 
and define
\[
	\rank_{\nu_\rt} (\bu) = \rank(\bu_{\nu_\rt}) \,.
\]
To simplify notation, we write
\begin{align*}
	\rank_\infty(\bu) = \lrnorm{\left(	\abs{\rank_{\nu_\rt}(\bu)}_{\infty}\right)_{\nu_\rt}}_{\ell_\infty},
\end{align*}
in other words, $\rank_\infty(\bu)$ denotes the maximum of the maximum rank of the hierarchical tensors associated to the temporal basis indices. Similarly, for $(\sfr_{\nu_\rt})_{\nu_\rt \in \vee_\rt} \in \mathcal{R}^{\vee_\rt}$, we write
\[
	\rank_\infty\bigl(  (\sfr_{\nu_\rt})_{\nu_\rt \in \vee_\rt} \bigr) = \lrnorm{\left(	\abs{\sfr_{\nu_\rt}}_{\infty}\right)_{\nu_\rt}}_{\ell_\infty} \,.
\]

\subsubsection{Near-optimal low-rank recompression}
For the hierarchical tensor format, best approximations for given rank bounds always exist, and truncation of \textit{hierarchical singular value decompositions} (HSVD) yields near-best approximations. In the case of a single hierarchical tensor representation, in the approach of \cite{BachmayrNearOptimal}, the low-rank approximation error is quantified in terms of the maximum entry of the hierarchical rank tuple. In present case of a sequence of tensor representations, we adapt this concept and quantify approximation errors in terms of the $\ell_\infty$-norm of the maximum ranks of the hierarchical tensor representations of each temporal basis index.

Any $\bv \in \ell_2( \vee_\rx )$ has an HSVD representation with mode frames $\{ \mathbf{U}^\alpha_{k} \}_{k = 1,\ldots, \rank_{\alpha}(\bv)}$ that are left singular vectors of $\operatorname{mat}_{\alpha}(\bv)$ for each $\alpha \in \mathbb{T}_d\setminus \{\alpha^*\}$.
The corresponding singular values of these matricizations are denoted by $\sigma^\alpha_k(\bv)$, $k = 1,\ldots, \rank_{\alpha}(\bv)$.
We next consider low-rank approximations in the form \eqref{eq:parabolic-lowrank-format} by termwise truncation of HSVD representations.
To this end, we first introduce notions of minimal ranks for a given target accuracy $\eta$.
Note that in specifying the arising matricizations, we employ Definition \ref{def:tensorformat}(a).

\begin{definition}
	For $\bv \in \ell_2(\vee)$ with spatial components $(\bv_{\nut})_{\nu_\rt \in \vee_\rt}$ and $(\sfr_{\nu_\rt})_{\nu_\rt \in \vee_\rt} \in \mathcal{R}^{\vee_\rt}$ with $\sfr_{\nu_\rt} = ( r_{\nu_\rt,e})_{e \in \mathbb{E}_d}$, we define
	\[
		\lambda\bigl((\sfr_{\nu_\rt})_{\nu_\rt \in \vee_\rt}  ; \bv\bigr) = \bigg( \sum\limits_{\nu_\rt \in \vee_\rt} \sum_{e \in \mathbb{E}_d} \sum_{k > r_{\nu_\rt,e}} \bigabs{\sigma^{[e]}_k(\bv_{\nu_\rt}) }^2 \bigg)^{\frac12} \,.
	\]
	For any $\eta > 0$, we choose $\sfr(\bv,\eta) \in \mathcal{R}^{\vee_\rt}$ with minimal $\rank_\infty\bigl( \sfr(\bv,\eta) \bigr)$ such that 
	\[ \lambda(\sfr(\bv,\eta); \bv) \leq \eta ;\]  
	that is, we choose hierarchical ranks to satisfy this bound such that the maximum of the maximum hierarchical ranks is minimized.
\end{definition}
For $\mathsf{r}\in\mathcal{R}$ and $\hat \bv \in \ell_2(\vee_\rx)$, we denote by $\mathrm{P}_{\hat\bv, \mathsf{r}}$ the linear mapping which applied to $\bv$ yields the HSVD truncation to hierarchical rank $\mathsf{r}$.
With the definition of minimal ranks $\sfr(\bv,\eta)$ for a given target accuracy $\eta$, writing $\sfr(\mathbf{v},\eta) = (\sfr_{\nu_\rt})_{\nu_\rt \in \vee_\rt}$ we define
\begin{equation}\label{eq:HSVDrecompressdef}
	\hat{\mathrm{P}}_{\eta} (\mathbf{v})  =  \sum\limits_{\nu_\rt \in \vee_\rt} \mathbf{e}_{\nu_\rt} \otimes \mathrm{P}_{\mathbf{v}_{\nu_\rt},\mathsf{r}_{\nu_\rt}}\mathbf{v}_{\nu_\rt}
\end{equation}
where $\mathbf{v}$ is in the form \eqref{eq:parabolic-lowrank-format}. Then we have by definition
\begin{align*}
	\norm{\mathbf{v} - \hat{\Pro}_\eta (\mathbf{v})} \leq \lambda(\sfr(\mathbf{v},\eta); \mathbf{v}) \leq \eta, \quad \rank_{\nut}\bigl(\hat{\Pro}_\eta (\mathbf{v})\bigr) = \sfr_{\nut}, \quad \nut \in \veet .
\end{align*}

Based on the representation \eqref{eq:parabolic-lowrank-format}, for $(\sfr_{\nu_\rt})_{\nu_\rt \in \vee_\rt} \in \cR^{\vee_\rt}$, we introduce the class of representations with bounded ranks
\begin{equation*}
	\cF\bigl((\sfr_{\nu_\rt})_{\nu_\rt \in \vee_\rt} \bigr) = \Big\{ \sum\limits_{\nu_\rt \in \vee_\rt} \mathbf{e}_{\nu_\rt} \otimes \bv_{\nu_\rt} : \bv_{\nu_\rt} \in \mathcal{H}(\sfr_{\nu_\rt}) \text{ for all } \nu_\rt \in \vee_\rt \Big\} .
\end{equation*}
As the following proposition shows, the quasi-optimality result for approximation by HSVD truncation carries over to this class of approximations.
\begin{prop}\label{prop:recompress-upper-bound}
	Let $\bv \in \ell_2(\vee)$ and $\kappa_{\Pro} = \sqrt{2d-3}$. Then for $(\sfr_{\nu_\rt})_{\nu_\rt \in \vee_\rt} \in \cR^{\vee_\rt}$, one has
	\begin{align}\label{eq:truncate-upper-bound}
		\lambda\bigl( (\sfr_{\nu_\rt})_{\nu_\rt \in \vee_\rt}; \bv\bigr) \leq \kappa_{\Pro} \inf\limits_{\bw \in \cF((\sfr_{\nu_\rt})_{\nu_\rt \in \vee_\rt})} \norm{\bw - \bv}.
	\end{align}
\end{prop}
\begin{proof}
	For the approximation of $\hat{\bv} \in \ell_2(\vee_\rx)$ by elements of $\cH(\sfr)$ for any $\sfr = (r_e)_{e\in\mathbb{E}_d} \in \cR$, we have
	\begin{align*}
		\sum_{e\in\mathbb{E}_d} \sum_{k > r_e} \bigabs{\sigma_k^{[e]}(\hat\bv)}^2 \leq \constsvd \inf\limits_{\bw \in \mathcal{H}(\sfr)} \norm{\bw - \hat{\bv}}^2 .
	\end{align*}
	Writing $\bv \in \ell_2(\vee)$ in the form \eqref{eq:parabolic-lowrank-format}, we obtain
	\begin{align*}
		\lambda^2\bigl((\sfr_{\nu_\rt})_{\nu_\rt \in \vee_\rt}; \bv\bigr) 
		&\leq \sum\limits_{\nu_\rt \in \vee_\rt} \kappa_{\Pro}^2 \inf\limits_{\bw_{\nu_\rt} \in \cH(\sfr_{\nu_\rt})} \norm{\bw_{\nu_\rt} - \bv_{\nu_\rt}}^2 \\
		&=\kappa_{\Pro}^2 \sum\limits_{\nu_\rt \in \vee_\rt}  \inf\limits_{\bw_{\nu_\rt} \in \cH(\sfr_{\nu_\rt})} \norm{\mathbf{e}_{\nu_\rt} \otimes \bw_{\nu_\rt} - \mathbf{e}_{\nu_\rt} \otimes \bv_{\nu_\rt}}^2 \\
		&= \kappa_{\Pro}^2 \inf\limits_{\bw \in \cF((\sfr_{\nu_\rt})_{\nu_\rt \in \vee_\rt})} \norm{\bw - \bv}^2,
	\end{align*}
	for any choice of $(\sfr_{\nu_\rt})_{\nu_\rt \in \vee_\rt}$, where we used that the entries on each sum are disjoint.
\end{proof}

For $r \in \N_0$, we define the best approximation errors with maximum rank $r$ by
\begin{multline*}
	\sigma_r(\bv) = \inf \Big\{ \norm{\bv - \bw} : \bw \in \cF\bigl((r_{\nu_\rt})_{\nu_\rt}\bigr)  \text{ with } (r_{\nu_\rt})_{\nu_\rt \in\vee_\rt} \in \cR^{\vee_\rt},  \rank_\infty\bigl( (r_{\nu_\rt})_{\nu_\rt} \bigr) \leq r  \Big\}.
\end{multline*}
For each $\eta > 0$, we introduce corresponding best approximation ranks by choosing $\bar{\sfr}(\bv,\eta) \in \mathcal{R}^{\vee_\rt}$ such that 
\[
	\rank_\infty\bigl( \bar{\sfr}(\bv,\eta) \bigr) = \min\{ r \in \N_0 \colon \sigma_r(\bv) \leq \eta \}\,.
\]
The following lemma provides an analogue of \cite[Lemma 2]{BachmayrNearOptimal} and is proved in the same manner using the Proposition \ref{prop:recompress-upper-bound}. For the reader's convenience, the proof can be found in Appendix \ref{sec:aux}.

\begin{lemma}\label{lem:recompress}
	Fix any $\alpha > 0$. For any $\bu,\bv, \eta$ satisfying $\norm{\bu - \bv} \leq \eta$, one has
	\begin{align}\label{eq:recompress-norm}
		\norm{\bu - \hat{\Pro}_{\kappa_{\Pro}(1+\alpha)\eta} (\bv)} \leq (1 + \kappa_{\Pro} (1+\alpha)) \eta
	\end{align}
	while
	\begin{equation}\label{eq:recompress-rank-upper-bound}
		\rank_\infty\bigl(\hat{\Pro}_{\kappa_{\Pro}(1+\alpha)\eta} (\bv)\bigr) \leq \rank_\infty \bigl( \bar{\sfr}(\bu,\alpha \eta) \bigr).
	\end{equation}
\end{lemma}

We now consider corresponding approximation classes in the same way as for a single low-rank representation in \cite[Definition 4]{BachmayrNearOptimal}. For simplicity, we adopt the same notation.

\begin{definition}
	We call a positive, strictly increasing sequence $\gamma = \big(\gamma(n)\big)_{n \in \N_0}$ with $\gamma(0) = 1$ and $\gamma(n) \to \infty$ as $n \to \infty$ a \textit{growth sequence}. For a given growth sequence $\gamma$, we define
	\begin{equation*}
		\norm{\bv}_{\cA(\gamma)} = \sup\limits_{r \in \N_0} \gamma(r) \sigma_r (\bv) , \qquad
		\cA (\gamma) = \left\{ \bv \in \ell_2(\vee) :\norm{\bv}_{\cA(\gamma)} < \infty \right\} .
	\end{equation*}
	We call the growth sequence $\gamma$ \textit{admissible} if
	\begin{equation*}
		\rho_{\gamma} = \sup\limits_{n \in \N} \frac{\gamma(n)}{\gamma(n-1)} < \infty,
	\end{equation*}
	which corresponds to a restriction of at most exponential growth. By $\gamma^{-1}\colon \R^+ \to \N_0$, we denote the left-continuous inverse of $\gamma$.
\end{definition}

Note that $\bv \in \cA(\gamma)$ means that a target accuracy $\varepsilon$ can be realized with maximum ranks of the size $\gamma^{-1}(\norm{\bv}_{\cA(\gamma)} / \varepsilon)$, so that a rank bound of the form $\gamma^{-1}(C \norm{\bv}_{\cA(\gamma)} / \varepsilon)$, where $C$ is any constant, is near-optimal.

We close this section with a final result about the HSVD recompression operator defined in \eqref{eq:HSVDrecompressdef}, where we assume that the approximand $\bu$ is an element of an approximation class $\cA(\gamma)$. 

\begin{theorem}\label{thm:recompress}
	Let $\kappa_{\Pro} = \sqrt{2d-3}$ and $\alpha > 0$. Assume that $\bu \in \cA(\gamma)$ with an admissible growth sequence $\gamma$ and that $\bv \in \ell_2(\vee)$ satisfies $\norm{\bu - \bv} \leq \eta$ for $\eta>0$. Then defining $\bw_{\eta} = \hat{\Pro}_{\kappa_{\Pro}(1+\alpha)\eta} (\bv)$, one has
	\begin{equation*}
		\norm{\bu - \bw_{\eta}} \leq (1 + \kappa_{\Pro}(1+\alpha))\eta, \qquad
		\rank_\infty(\bw_{\eta}) \leq \gamma^{-1} \big(\rho_{\gamma} \norm{\bu}_{\cA(\gamma)} / (\alpha \eta)\big), 
	\end{equation*}
	and
	\[
		\norm{\bw_{\eta}}_{\cA(\gamma)} \leq \bigl(\alpha^{-1}(1+\kappa_{\Pro}(1+\alpha)) + 1\bigr) \norm{\bu}_{\cA(\gamma)} .
	\]
\end{theorem}
Using Lemma \ref{lem:recompress}, the statement can be proved exactly in the same way as for a single low-rank representation, see \cite[Theorem 6]{BachmayrNearOptimal}. Note that here, the restriction to admissible $\gamma$ is not essential for obtaining rank bounds using Lemma \ref{lem:recompress}, but for $\gamma$ of faster growth (corresponding to faster than exponential decay of matricization singular values) one obtains only weaker information on quasi-optimality of ranks than provided by Theorem~\ref{thm:recompress}.

\subsubsection{Contractions and coarsening} 

In addition to the near-optimal low-rank recompression operator, we need a way to select finitely many basis indices for approximation in each mode frame. For this mechanism we adapt the concept of lower-dimensional \emph{contractions} used in \cite{BachmayrNearOptimal,bachmayr2014adaptive}.

The basic idea of the coarsening operator is to reduce the complexity of given coefficient sequences in tensor representations by discarding basis indices of all coefficients of sufficiently small absolute value. At the same time, in our present setting we also need to preserve the Cartesian product structure of spatial index sets, separately for each time index. We use the following standard notions of best $N$-term approximation.

\begin{definition}
	For any countable index set $\hat\vee$, $\Lambda \subset \hat{\vee}$, and $\bv \in \ell_2(\hat{\vee})$, we define the \emph{restriction} $\Res_{\Lambda} \bv$ to be equal to $\bv$ on $\Lambda$, and zero on $\hat\vee\setminus\Lambda$. For $s > 0$, we define 
	\begin{equation}\label{eq:Asnormdef}
		\norm{\bv}_{\cAs(\hat{\vee})} = \sup\limits_{N \in \N_0} (N+1)^s \inf\limits_{\substack{\Lambda \subset \hat{\vee} \\ \# \Lambda \leq N}} \norm{\bv - \Res_{\Lambda} \bv}
	\end{equation}
	and the approximation class
	\begin{align*}
		\cAs(\hat{\vee}) = \left\{ \bv \in \ell_2(\hat{\vee}) : \norm{\bv}_{\cAs(\hat{\vee})} < \infty \right\} .
	\end{align*}
\end{definition}
It can be shown that $\cAs(\hat{\vee})$ is a quasi-Banach space with the quasinorm defined in \eqref{eq:Asnormdef}; that is, $\norm{\cdot}_{\cAs(\hat{\vee})}$ satisfies the properties of a norm except for the triangle inequality. Where no confusion can arise, we write $\cAs = \cAs(\hat{\vee})$.
The following statements are direct consequences of the definition of the $\cAs$-quasinorm.
\begin{prop}\label{prop:cAs-properties}
	For $\bv,\bw \in \cAs(\hat \vee)$, the $\cAs$-quasinorm has the following properties:
	\begin{enumerate}[{\rm(i)}]
		\item\label{prop:cAs-triangle} $\norm{\bv + \bw}_{\cAs} \leq 2^s (\norm{\bv}_{\cAs} + \norm{\bw}_{\cAs})$.
		\item If $\# \supp(\bv) \leq N$, we have $\norm{\bv}_{\cAs} \leq N^s \norm{\bv}$.
		\item If $\bv_N$ is a best $N$-term approximation of $\bv$, then
		$\norm{\bv}_{\cAs} \geq (N+1)^s \norm{\bv - \bv_N}$.
		\item\label{prop:restr-ineq} $\norm{\Res_{\Lambda} \bv}_{\cAs} \leq \norm{\bv}_{\cAs}$ for any $\Lambda \subset \hat\vee$.
	\end{enumerate}
\end{prop}

In our present setting, we apply this notion of approximation classes separately to the lower-dimensional sets of basis indices in each tensor mode. As in case of single tensor representations used for elliptic problems in \cite{BachmayrNearOptimal,bachmayr2014adaptive}, we use the concept of \textit{near best} $N$-term approximations based on \emph{tensor contractions}. These contractions as defined below can be interpreted as lower-dimensional densities of coefficient tensors.
Here we extend the concept of contractions to our combined spatial low-rank and temporal sparse approximation by using the set of contractions for each low-rank approximation of each time basis index.
\begin{definition}\label{def:contractions}
	For $\bv \in \ell_2(\vee_\rx)$, where $\vee_\rx = \bigtimes_{i=1}^d \vee_{1}$, and $i=1,\ldots,d$, we define the (spatial) contractions $\pi^{(i)}(\bv) = \bigl( \pi^{(i)}_{\mu}(\bv) \bigr)_{\mu\in\vee_1} \in \ell_2(\vee_1)$ by their entries
	\[
		\pi^{(i)}_{\mu} (\bv)
		=  \biggl( \sum_{\nu_1 \in \vee_1} \cdots \sum_{\nu_{i-1}\in\vee_1} \sum_{\nu_{i+1}\in\vee_1} \cdots \sum_{\nu_d\in\vee_1} \abs{\bv_{\nu_1,\ldots,\nu_{i-1},\mu,\nu_{i+1},\ldots,\nu_d}}^2 \biggr)^{\frac{1}{2}}, \quad \mu \in \vee_1.
	\]
	Let $\bw \in \ell_2(\vee)$ with spatial components $(\bw_{\nut})_{\nut\in\veet}$, where $\vee = \vee_\rt \times \vee_\rx$. For $i =1,\ldots,d$, we define the spatio-temporal contractions
	\begin{align*}
		\piti(\bw) = \bigl( \pi^{(i)}_{\mu} (\bw_{\nu_\rt}) \bigr)_{(\nu_\rt,\mu) \in \vee_\rt \times \vee_{1}}
	\end{align*}
	as well as the temporal contractions
	\begin{align*}
		\pit(\bw) = \Bigl(\pit_{\nu_\rt}(\bw)\Bigr)_{\nu_\rt \in \vee_\rt} = \bigl(\norm{\bw_{\nu_\rt}}\bigr)_{\nu_\rt \in \vee_\rt} .
	\end{align*}
\end{definition}
The direct computation of these quantities would involve high-dimensional summations. The observations from \cite{BachmayrNearOptimal} show that this can be avoided due to orthogonality properties of the tensor formats. These also apply in the present setting, because the summations are evaluated independently for each time index.
In our case, we have the following analogous basic properties of the contractions.
\begin{prop}\label{prop:error-restrict-contr}
	Let $\bv \in \ell_2(\vee)$ with spatial components $(\bv_{\nut})_{\nut\in\veet}$.
	\begin{enumerate}[{\rm(i)}]
		\item We have $\norm{\bv} = \norm{\piti(\bv)}, i=1,\dots,d$.
		\item\label{it:restr-norm-contr} Let $\Lambda^{(i)}_{\nu_\rt} \subseteq \vee_1$ for each $i=1,\dots,d$ and $\nu_\rt \in \vee_\rt$ and let
		\begin{align*}
			\Lambda = \bigcup\limits_{\nu_\rt \in \vee_\rt} \{\nu_\rt\} \times \Lambda^{(1)}_{\nu_\rt} \times \cdots \times \Lambda^{(d)}_{\nu_\rt} ,
		\end{align*}
		then we have
		\begin{align}\label{eq:bound-restriction-contractions}
			\norm{\bv - \Res_{\Lambda} \bv} \leq \biggl( \sum\limits_{\nu_\rt \in \vee_\rt} \sum\limits_{i=1}^d \sum\limits_{\nu_\rx \in \vee_1 \setminus \Lambda^{(i)}_{\nu_\rt}} \bigabs{\piti_{\nu_\rt,\nu_\rx}(\bv)}^2 \biggr)^{\frac{1}{2}} .
		\end{align}
	\item\label{prop:contr-hsvd} For $i=1,\ldots,d$, let in addition $\bU^{(i)}_{\nu_\rt}$ be the mode frames and $\sigma_k^{(i,\nu_\rt)}$ be the sequence of mode-$i$ singular values of an HSVD of $\bv_{\nu_\rt}$. Then
	\begin{align*}
		\piti_{\nu_\rt,\nu_\rx}(\bv) = \biggl( \sum\limits_k \bigabs{\bigl(\bU^{(i)}_{\nu_\rt}\bigr)_{\nu_\rx,k}}^2 \bigabs{\sigma_k^{(i,\nu_\rt)}}^2 \biggr)^{\frac{1}{2}} .
	\end{align*}
	\end{enumerate}
\end{prop}
At certain points we need a slightly modified notation of Proposition \ref{prop:error-restrict-contr}\eqref{it:restr-norm-contr}. We define 
\begin{align*}
	\Lambda^{(i)} = \bigcup_{\nu_\rt \in \vee_\rt} \bigl\{ (\nu_\rt, \nu_i) : \nu_i \in \Lambda^{(i)}_{\nu_\rt} \bigr\}
\end{align*}
and from \eqref{eq:bound-restriction-contractions} obtain
\begin{equation}\label{eq:bound-restriction-contraction-mod}
		\norm{\bv - \Res_{\Lambda} \bv } \leq  \biggl(\sum\limits_{i=1}^d \norm{\piti(\bv) - \Res_{\Lambda^{(i)}} \piti(\bv)}^2\biggr)^{\frac{1}{2}} 
		\leq  \sum\limits_{i=1}^d \norm{\piti(\bv) - \Res_{\Lambda^{(i)}} \piti(\bv)}.
\end{equation}

Additionally we have the following subaddivity property, which is an immediate consequence of the triangle inequality.
\begin{prop}
	Let $N \in \N$ and $\bv_n \in \ell_2(\vee), n=1,\dots,N$. Then for $i=1,\ldots,d$ and $\nu_\rx \in \vee_1$, $\nu_\rt \in \vee_\rt$, we have
	\begin{align*}
		\piti_{\nu_\rt,\nu_\rx} \Bigl( \sum\limits_{n=1}^N \bv_n \Bigr) \leq \sum\limits_{n=1}^N \piti_{\nu_\rt,\nu_\rx} (\bv_n) .
	\end{align*}
\end{prop}
As in the case of a single tensor representation, we perform coarsening of $\bv$ with spatial components in low-rank representation by selecting index sets by best $N$-term approximations of the contractions $\piti(\bv)$. The resulting spatial basis index sets have Cartesian product structure separately for each temporal basis index. We determine these index sets by rearranging the entire set of all contractions of all mode frames of all temporal indices indices, given by
\begin{align*}
	\bigl\{\piti_{\nu_\rt,\nu_\rx} (\bv) : \nu_\rx \in \vee_1, \nu_\rt \in \vee_\rt, i=1,\dots,d\bigr\},
\end{align*}
to a non-increasing sequence
\begin{align*}
	\pi^{(\rt,i(1))}_{\nu_\rt^*(1),\nu_\rx^*(1)}(\bv) \geq \pi^{(\rt,i(2))}_{\nu_\rt^*(2),\nu_\rx^*(2)}(\bv) \geq \ldots \geq \pi^{(\rt,i(j))}_{\nu_\rt^*(j),\nu_\rx^*(j)}(\bv) \geq \ldots, 
\end{align*}
where $i(j) \in \{1,\dots,d\}, \nu_\rt^*(j) \in \vee_\rt$ and $\nu_\rx^*(j) \in \vee_1$ for each $j \in \N$.
We retain only the $N$ largest from this ordering and redirect them to the respective dimension bins and time indices,
\begin{align*}
	\Lambda^{(i)}_{\nu_\rt} (\bv,N) = \bigl\{ \nu_\rx^*(j) : i(j) = i, \nu_\rt^*(j) = \nu_\rt, j=1,\dots,N \bigr\}\,.
\end{align*}
Now for each time index $\nu_\rt \in \vee_\rt$, we consider the Cartesian product index set
\begin{align*}
	\Lambda_{\nu_\rt}(\bv,N) = \bigtimes_{i=1}^d \Lambda^{(i)}_{\nu_\rt}(\bv,N)
\end{align*}
and denote by $\Lambda(\bv,N)$ the union of these product sets combined with their respective time indices, that is,
\begin{align}\label{eq:contraction-set-defs}
	\Lambda(\bv,N) = \bigcup\limits_{\nu_\rt \in \vee_\rt} \bigl( \{\nu_\rt\} \times \Lambda_{\nu_\rt}(\bv,N) \bigr) \,\subset\,\vee.
\end{align}
By construction, 
\begin{align*}
	\sum\limits_{\nu_\rt \in \vee_\rt} \sum\limits_{i=1}^d \# \Lambda^{(i)}_{\nu_\rt} (\bv,N) \leq N 
\end{align*}
and
\begin{align}\label{eq:contractions-minimum}
	\sum\limits_{\nu_\rt \in \vee_\rt} \sum\limits_{i=1}^d \sum\limits_{\nu_\rx \in (\vee_1 \setminus \Lambda^{(i)}_{\nu_\rt}(\bv,N))} \abs{\piti_{\nu_\rt,\nu_\rx}(\bv)}^2 
	= \min\limits_{\hat{\Lambda}} \biggl\{  \sum\limits_{\nu_\rt \in \vee_\rt} \sum\limits_{i=1}^d \sum\limits_{\nu_\rx \in (\vee_1 \setminus \hat{\Lambda}^{(i)}_{\nu_\rt})} \abs{\piti_{\nu_\rt,\nu_\rx}(\bv)}^2 \biggr\},
\end{align}
where $\hat{\Lambda}$ ranges over all sets that can be written in the form
\begin{align}\label{eq:parabolic-indexset-struct}
	\hat{\Lambda} = \bigcup\limits_{\nu_\rt \in \vee_\rt} \{\nu_\rt\} \times \bigtimes\limits_{i=1}^d \hat{\Lambda}^{(i)}_{\nu_\rt}, 
	\quad \text{where}\quad  \sum\limits_{\nu_\rt \in \vee_\rt} \sum\limits_{i=1}^d \# \hat{\Lambda}^{(i)}_{\nu_\rt} \leq N\,.
\end{align}
\begin{prop}
	For any $\bv \in \ell_2(\vee)$ with spatial components $(\bv_{\nut})_{\nut\in\veet}$ we have
	\begin{equation}\label{eq:mu-N}
		\norm{\bv - \Res_{\Lambda(\bv,N)} \bv} \leq \mu_N(\bv) ,
	\end{equation}
	where
	\begin{equation}\label{eq:mudef}
		\mu_N(\bv)  = \biggl( \sum\limits_{\nu_\rt \in \vee_\rt} \sum\limits_{i=1}^d \sum\limits_{\nu_\rx \in (\vee_1 \setminus \Lambda^{(i)}_{\nu_\rt}(\bv,N))} \abs{\piti_{\nu_\rt,\nu_\rx}(\bv)}^2  \biggr)^{\frac{1}{2}} ,
	\end{equation}
	and for any $\hat{\Lambda}$ satisfying \eqref{eq:parabolic-indexset-struct}, we have
	\begin{align}\label{eq:coarsen-inequality-other-indexset}
		\norm{\bv - \Res_{\Lambda(\bv,N)} \bv} \leq \mu_N(\bv) \leq \sqrt{d} \norm{\bv - \Res_{\hat{\Lambda}} \bv}.
	\end{align}
\end{prop}
\begin{proof}
	The bound \eqref{eq:mu-N} is an immediate consequence of \eqref{eq:bound-restriction-contractions}. Let now be $\hat{\Lambda}$ be as in the assumption, then by using \eqref{eq:mu-N} and \eqref{eq:contractions-minimum} we obtain
	\begin{align*}
		\norm{\bv - \Res_{\Lambda(\bv,N)} \bv}^2 &\leq \sum\limits_{\nu_\rt \in \vee_\rt} \sum\limits_{i=1}^d \sum\limits_{\nu_\rx \in (\vee_1 \setminus \hat{\Lambda}^{(i)}_{\nu_\rt})} \abs{\piti_{\nu_\rt,\nu_\rx}(\bv)}^2 \\
		&= \sum\limits_{\nu_\rt \in \vee_\rt} \norm{\bv_{\nu_\rt} - \Res_{\hat{\Lambda}^{(1)}_{\nu_\rt} \times \vee_1 \times \cdots  \times \vee_1} \bv_{\nu_\rt}}^2 + \ldots  + 
		\norm{\bv_{\nu_\rt} - \Res_{\vee_1 \times \cdots \times \vee_1 \times \hat{\Lambda}^{(d)}_{\nu_\rt}} \bv_{\nu_\rt}}^2 \\
		&\leq \sum\limits_{\nu_\rt \in \vee_\rt} d \norm{\bv_{\nu_\rt} - \Res_{\hat{\Lambda}_{\nu_\rt}} \bv_{\nu_\rt}}^2 \\
		&= d \norm{\bv - \Res_{\hat{\Lambda}} \bv }^2 ,
	\end{align*}
	where we used in the last line that $\mathbf{e}_{\nu_\rt} \otimes \bv_{\nu_\rt}$ have disjoint support for each $\nu_\rt \in \vee_\rt$.
\end{proof}

The sorting of the set of all contractions can be replaced by a \emph{quasi-sorting} by binary binning, which corresponds to the one-dimensional case of the coarsening operator, see \cite{Metselaar:02,Barinka:05,Dijkema:09}. For each $N \in \N$, we can compare $\Lambda(\bv,N)$ to the index set $\bar{\Lambda}(\bv,N)$ that satisfies the conditions in \eqref{eq:parabolic-indexset-struct} as well as
\begin{align}\label{eq:coarsen-best-optimality}
	\norm{\bv - \Res_{\bar\Lambda(\bv,N)} \bv} = \min\limits_{\sum_i \# \supp(\piti(\bw)) \leq N} \norm{\bv - \bw}.
\end{align}
As a consequence of \eqref{eq:coarsen-inequality-other-indexset}, 
\begin{align}\label{eq:coarsen-best-ineq}
	\norm{\bv - \Res_{\Lambda(\bv,N)} \bv} \leq \mu_N(\bv) \leq \kappa_{\Cor} \norm{\bv - \Res_{\bar\Lambda(\bv,N)}  \bv}, \quad \kappa_{\Cor} = \sqrt{d} ,
\end{align}
with $\mu_N(\bv)$ from \eqref{eq:mudef}.

In our adaptive scheme we need a corresponding procedure that yields the smallest index set such that the computable error bound $\mu_N(\bv)$ is below a given threshold $\eta$.
To this end, we define $N(\bv,\eta) = \min\{N : \mu_N(\bv) \leq \eta\}$. The corresponding thresholding procedure is then given by
\begin{align}\label{eq:coarsen-threshold}
	\hat{\Cor}_{\eta} (\bv) = \Res_{ \Lambda(\bv, N(\bv,\eta) ) } \bv .
\end{align}

\subsubsection{Combination of Tensor Recompression and Coarsening}

We complete this section by the following main result on the combination of low-rank recompression and basis coarsening. It is an extension of Theorem \ref{thm:recompress} and shows that both reduction techniques combined are optimal up to uniform constants and stable with respect to $\norm{\cdot}_{\cAs}$ and $\norm{\cdot}_{\cA(\gamma)}$ for the the mode frames and the low-rank approximability, respectively.
\begin{theorem}\label{thm:combinedcoarsen}
	Let $\bu, \bv \in \ell_2(\vee)$ with $\bu \in \cA(\gamma)$, where $\gamma$ is an admissible growth sequence, $\piti \in \cAs$ for $i=1,\dots,d$ and $\norm{\bu - \bv} \leq \eta$. As before let $\kappa_{\Pro} = \sqrt{2d-3}$ and $\kappa_{\Cor} = \sqrt{d}$. Then for
	\begin{align*}
		\bw_{\eta} = \hat{\Cor}_{\kappa_{\Cor}(\kappa_{\Pro}+1)(1+\alpha)\eta} \left( \hat{\Pro}_{\kappa_{\Pro}(1+\alpha)\eta} (\bv) \right),
	\end{align*}
	we have
	\begin{align}\label{eq:recompress-coarsen-norm}
		\norm{\bu - \bw_{\eta}} \leq (1 + \kappa_{\Pro}(1+\alpha) + \kappa_{\Cor}(\kappa_{\Pro}+1)(1+\alpha))\eta
	\end{align}
	as well as
	\begin{align}\label{eq:recompress-coarsen-rank}
		\begin{split}
			\rank_\infty(\bw_{\eta})  &\leq \gamma^{-1} \big(\rho_{\gamma} \norm{\bu}_{\cA(\gamma)} / (\alpha \eta)\big), \\
			\norm{\bw_{\eta}}_{\cA(\gamma)} &\leq C_1 \norm{\bu}_{\cA(\gamma)},
		\end{split}
	\end{align}
	where $C_1 = \alpha^{-1}(1+\kappa_{\Pro}(1+\alpha)) + 1$ and
	\begin{align}\label{eq:recompress-coarsen-supp}
		\begin{split}
			\sum\limits_{i=1}^d \# \supp(\piti(\bw_{\eta})) &\leq 2 \eta^{-\frac{1}{s}} d \alpha^{-\frac{1}{s}} \left(\sum\limits_{i=1}^d \norm{\piti(\bu)}_{\cAs} \right)^{\frac{1}{s}}, \\
			\sum\limits_{i=1}^d \norm{\piti(\bw_{\eta})}_{\cAs} &\leq C_2 \sum\limits_{i=1}^d \norm{\piti(\bu)}_{\cAs}	
		\end{split}
	\end{align}
	with $C_2 = 2^s(1+2^s) + 2^{3s} \alpha^{-1} (1+\kappa_{\Pro}(1+\alpha) + \kappa_{\Cor}(\kappa_{\Pro}+1)(1+\alpha))d^{\max\{1,s\}}$.
\end{theorem}

To prove this theorem we follow the lines of \cite[Theorem 7]{BachmayrNearOptimal}. For the convenience of the reader, the proof is given in Appendix \ref{sec:aux}.

\begin{remark}\label{rem:hsvd-complexity}
	The coarsening as well as the recompression routine require a HSVD of their inputs. For each finitely supported $\bv \in \ell_2(\vee)$, using the bound for the costs of each component $\bv_{\nu_\rt} \in \ell_2(\vee_\rx)$, $\nu_\rt \in \vee_\rt$, given in hierarchical format, the total number of operations required for computing these decompositions for $\bv$ is bounded by a fixed multiple of
	\begin{multline*}
		 \sum\limits_{\nu_\rt \in \vee_\rt} \Bigl(d \abs{\rank_{\nu_\rt}(\bv)}_{\infty}^4 + \abs{\rank_{\nu_\rt}(\bv)}_{\infty}^2 \sum\limits_{i=1}^d \# \supp(\pi^{(i)}(\bv_{\nu_\rt}))\Bigr) \\
		\leq d \rank^4_\infty(\bv) \# \supp(\pi^{(\rt)}(\bv)) + \rank^2_\infty(\bv) \sum\limits_{i=1}^d \# \supp(\piti(\bv)) ,
	\end{multline*}
	and the right-hand side can be estimated by \[  2 \rank^4_\infty(\bv) \sum\limits_{i=1}^d \# \supp(\piti(\bv)) . \] 
	When the HSVD needs to be computed for redundant hierarchical tensor representations with representation ranks greater than $\rank_{\nu_\rt}(\bv)$ for $\nu_\rt\in \vee_\rt$, the latter need to be replaced by the respective representation ranks in the above bound.
\end{remark}

\section{Low-Rank Preconditioning}\label{sec:low-rank-preconditioning}
The scaling matrices $\mathbf{\bar D}_\cX$ and $\mathbf{\bar D}_\cY$ from \eqref{eq:scaling-matrices} and \eqref{eq:def-DY} play an important role, similarly to the elliptic case treated in \cite{bachmayr2014adaptive}, in the sense that the operator $\mathbf{\bar B}$ defined in \eqref{eq:operator-Bhat} is boundedly invertible with condition number independent of the spatial dimension $d$. But even though the operators $\bT$ and $\bT_0$ from \eqref{eq:opertor-T1T2-def} have an explicit low-rank representation, since the scaling matrices do not have finite-rank representations, the operator $\mathbf{\bar B}$ generally still has unbounded ranks. 
To solve this problem we define scaling operators $\bDX$ and $\bDY$ that are equivalent to $\mathbf{\bar D}_\cX$ and $\mathbf{\bar D}_\cY$ in the sense that
\begin{align}\label{eq:norm-equivalence-scaling-matrices}
	\norm{\bDX \mathbf{\bar D}_\cX^{-1}} , \norm{\bDX^{-1} \mathbf{\bar D}_\cX}<\infty, \quad \norm{\bDY \mathbf{\bar D}_\cY^{-1} } ,  \norm{\bDY^{-1} \mathbf{\bar D}_{\cY} } < \infty,
\end{align}
but that at the same time can be approximated by separable operators in an efficient and quantifiable way.
For practical purposes, the corresponding bounds in \eqref{eq:norm-equivalence-scaling-matrices} should not be too large.

With such equivalent substitute scaling matrices $\bDX$ and $\bDY$, and with \[ \mathbf{g}_1 = ( f(\theta_{\nut} \otimes \Psi_{\nu_\rx},0))_{(\nut,\nu_\rx)\in\vee}, \quad \mathbf{g}_2 = ( f(0, \Psi_\nu))_{\nu \in \vee_\rx},\] 
we obtain a modified system
\begin{equation}\label{eq:Bfdef}
	\bB \bu = \bfs, \quad \bB = \begin{bmatrix} 
		\bB_1 \\ \bB_2
	\end{bmatrix}, \quad
	\bfs = \begin{bmatrix}
		\bDY \mathbf{g}_1 \\ \mathbf{g}_2
	\end{bmatrix} ,
\end{equation}
where
\begin{equation}\label{eq:Bpartsdef}
  \bB_1 = \bDY \bT \bDX  , \quad \bB_2 = 	\bT_0 \bDX\,.
\end{equation}
Here we gain low-rank approximability based on the fact that the scaling matrices $\bDX$ and $\bDY$ can be efficiently approximated by separable operators. Furthermore, based on \eqref{eq:norm-equivalence-scaling-matrices}, the bi-infinite matrix $\bB$ is still boundedly invertible.

In the remainder of this section, we give a construction of appropriate $\bDX$ and $\bDY$. For $\bDY$, we can make use of the fact that $\mathbf{\bar D}$ is well understood in the sense of low-rank approximability \cite[Section 4.1]{bachmayr2014adaptive}. For the matrix $\bDX$, we develop a new type of low-rank approximation.

\subsection{Scaling matrix $\bDX$}

With $\mathbf{E}_{\nu_\rt} = \mathbf{e}_{\nu_\rt}  \mathbf{e}_{\nu_\rt}^\intercal$ for $\nut \in \veet$, we construct $\bDX$ in the form
\begin{align*}
	\bDX = \sum\limits_{\nu_\rt \in \vee_\rt} \mathbf{E}_{\nu_\rt} \otimes \bDXnut \,.
\end{align*}
The construction of each $\bDXnut$ follows a similar idea as \cite{bachmayr2014adaptive} based on exponential sum approximations with bounds on the relative error. Here we obtain substitutes for the expressions in \eqref{eq:omega-t-x} for each fixed $\nu_\rt$ that have efficient low-rank approximations. In the following result of this type, an important role is played by the \emph{Dawson function} $F$, which is defined as
\begin{equation}\label{eq:dawsondef}
	F(x) = e^{-x^2} \int\limits_0^x e^{s^2} ds.
\end{equation}

\begin{theorem}\label{thm:approx-scaling}
	For each $a > 0$, let
	\[
	\alpha_a(x) = a^{-1} e^x, \qquad w_a(x) = \pi^{-1/2} a^{-1/2} \big(1-2e^{x/2} F(e^{x/2})\big) e^{x/2} \,.
	\]
	For an arbitrary but fixed $\delta \in (0,1)$, let $h_a$ be chosen such that
	\begin{align}\label{eq:h-epxonential-sum}
		0 < h_a \leq \frac{2 \pi^2}{3 \ln\left(1 + 10 \frac{1 + a}{\delta}\right)}
	\end{align}
	and set
	\begin{align}\label{eq:nplus-eponential-sum}
		n^+_a = n^+_a(\delta) = \left\lceil 2 h_a^{-1} \ln\bigl(\sqrt{a} \erfcinv(\delta/2)\bigr) \right\rceil.
	\end{align}
	Then, defining
	\begin{align*}
		\Phi_{a,n} (s) = \sum\limits_{k=-n}^{n^+} h_a w_a(kh_a) e^{-\alpha_a(kh_a)s}, \quad \Phi_{a,\infty} (s) = \lim\limits_{n \to \infty} \Phi_{a,n} (s),
	\end{align*}
	one has
	\begin{align}\label{eq:scaling-diff-DX-DXtilde}
		\lrabs{\frac{\sqrt{s}}{s + a} - \Phi_{a,\infty}(s)} \leq \delta \frac{\sqrt{s}}{s+a} \quad \text{for all } s \in [1,\infty).
	\end{align}
	For any $\eta > 0$ and $K > 1$, provided that 
	\begin{align*}
		n \geq \min \Bigl\{ n \in \N_0 : \max\bigl\{f_a(1,h_a n),f_a(K,h_a n)\bigr\} \leq \min\{\delta/2,\eta\}  \Bigr\}
	\end{align*}
	where $f_a$ is defined by
	\begin{align*}
		f_a(s,y) = 2 \sqrt{\tfrac{a}{s\pi}} e^{-\frac{s}{a} e^{-y}} F\left(e^{-\frac{y}{2}}\right) + \erf\left(\sqrt{\tfrac{s}{a}} e^{-\frac{y}{2}}\right),
	\end{align*}
	one has in addition 
	\begin{align}
		\lrabs{\frac{\sqrt{s}}{s + a} - \Phi_{a,n}(s)} &\leq \delta \frac{\sqrt{s}}{s+a} \quad \text{for all } s \in [1,K], \label{eq:scaling-diff-DX}\\
		\lrabs{\Phi_{a,\infty}(s) - \Phi_{a,n}(s)} &\leq \eta \frac{\sqrt{s}}{s+a} \quad \text{for all } s \in [1,K]. \label{eq:scaling-diff-DXhat}
	\end{align}
\end{theorem}
The proof of this theorem is given in Appendix \ref{app:expsum}.
We define
$S_{\min} = \bigl( \min\limits_{\nu \in \vee_\rx} \norm{\Psi_\nu}_V \bigr)^{-1}$
and set
\begin{align}\label{eq:scaling-time-a}
	a_{\nu_{\rt}} = S_{\min}^2 \norm{ \theta_{\nu_\rt}}_{H^1(I)}, \quad \nu_\rt \in \vee_\rt.
\end{align}
In addition, let the values of $\delta$, $h_{a_{\nu_\rt}}$ and $n_{a_{\nu_\rt}}^+$ for each $\nu_\rt \in \vee_\rt$ be fixed according to Theorem \ref{thm:approx-scaling}, where $h_{a_{\nu_\rt}}$ are chosen as the upper bound in \eqref{eq:h-epxonential-sum}.
Then, for any  $\nu_\rt \in \vee_\rt$ and $\sfn = (n_{\nu_\rt})_{\nu_\rt \in \vee_\rt} \subset \N_0$, we define
\[
	 S^\cX_{n,\nu_\rt,\nu_\rx} = S_{\min} \Phi_{a_{\nut},n} \bigl(  S_{\min}^2 \norm{\Psi_{\nu_\rx}}_V^2  \bigr) 
\]
as well as
\[
	\mathbf{D}_{\cX,\nut,\sfn} \bv_{\nut} = (\mathbf{D}_{\cX,\sfn} \bv)_{\nu_\rt} = \left( S^\cX_{n_{\nu_\rt},\nu_\rt,\nu_\rx} v_{\nu_\rt,\nu_\rx} \right)_{\nu_\rx \in \vee_\rx}.
\]
As a consequence of this definition, $\bDXnut$ can be approximated as a sum of $1 + n_{a_{\nut}}(\delta) + n_{a_{\nut}}$ separable terms  for each time index $\nu_\rt$.
In the limit $n \to \infty$, we obtain the reference scaling
\begin{align*}
	\bD_{\cX,\nut} \bv_{\nut} = (\bDX \bv)_{\nu_\rt} = \left( S^\cX_{\nu_\rt,\nu_\rx} v_{\nu_\rt,\nu_\rx} \right)_{\nu_\rx \in \vee_\rx}
\end{align*}
where
\begin{align*}
	{S}^\cX_{\nu_\rt,\nu_\rx} = \lim\limits_{n \to \infty} S^\cX_{n,\nu_\rt,\nu_\rx} =  S_{\min} \Phi_{a_{\nut},\infty} \bigl(  S_{\min}^2 \norm{\Psi_{\nu_\rx}}_V^2  \bigr)  .
\end{align*}

In the next step we rephrase the statements \eqref{eq:scaling-diff-DX} and \eqref{eq:scaling-diff-DXhat} in terms of the scaling matrix $\mathbf{D}_{\cX,\sfn}$. Note that for any $K>1$, the bounds \eqref{eq:scaling-diff-DX} and \eqref{eq:scaling-diff-DXhat} hold true for spatial indices in the subset 
\begin{equation}\label{eq:LambdaKdef}
	\Lambda_K = \bigl\{ \nu_\rx \in  \vee_\rx : S_{\min}^2 \norm{\Psi_{\nu_\rx}}_V^2 \leq  K\bigr\} .
\end{equation}
In addition, we define
\begin{align*}
	\MXnu(\eta;K) &= \min \Bigl\{ n \in \N_0 : \max\bigl\{f_{a_{\nut}}(1,h_{a_{\nut}} n),f_{a_{\nut}}(K,h_{a_{\nut}} n)\bigr\} \leq \min\{\delta/2,\eta\}  \Bigr\}, \\
	\MXnuZe(K) &= \MXnu(\delta/2,K),
\end{align*}
where $a_{\nut}$ is given by \eqref{eq:scaling-time-a}. Then for all $\nu_\rt \in \vee_\rt$, $\eta > 0$ and $K > 1$, we obtain for $n \geq M_{\nu_\rt}(\eta;K)$ the estimates
\begin{align*}
	\lrabs{\bigl(\bar S^\cX_{\nu_\rt,\nu_\rx}\bigr)^{-1} \bigl(\bar S^\cX_{\nu_\rt,\nu_\rx} -  S^\cX_{n,\nu_\rt,\nu_\rx}\bigr)} \leq \delta, \quad \lrabs{\bigl(\bar S^\cX_{\nu_\rt,\nu_\rx}\bigr)^{-1} \bigl( S^\cX_{\nu_\rt,\nu_\rx} -  S^\cX_{n,\nu_\rt,\nu_\rx}\bigr)} \leq \eta.
\end{align*}
These estimates are a direct consequence of \eqref{eq:scaling-diff-DX} and \eqref{eq:scaling-diff-DXhat} as well as \eqref{eq:scaling-time-a} and the definition of $\bar S^\cX$ in \eqref{eq:omega-t-x}.

In our present setting, we generally have different spatial index sets for each time index. With this in mind, we define for $\sfK = (K_{\nu_\rt})_{\nu_\rt \in \vee_\rt}$ the index set
\begin{equation}\label{eq:LambdaKKdef}
	\Lambda_{\sfK} = \bigcup\limits_{\nu_\rt \in \vee_\rt} \{\nu_\rt\} \times \Lambda_{K_{\nu_\rt}}
\end{equation}
and set
\begin{equation*}
	\MX(\eta;\sfK) = \left(\MXnu(\eta;K_{\nu_\rt})\right)_{\nu_\rt \in \vee_\rt}, \qquad
	\MXZe(\sfK) = (\MXnuZe(K_{\nu_\rt}))_{\nu_\rt \in \vee_\rt}.
\end{equation*}
Then for any $\eta > 0$ and $\sfK > 1$, provided that $\sfn \geq \MX(\eta;\sfK)$ (with these inequalities to be understood componentwise), we have
\begin{align}\label{eq:scaling-X-error}
	\norm{\mathbf{\bar D}_\cX^{-1}(\mathbf{\bar D}_\cX - \mathbf{D}_{\cX,\sfn}) \Res_{\Lambda_\sfK}} \leq \delta, \quad \norm{\mathbf{\bar D}_\cX^{-1}(\bDX - \mathbf{D}_{\cX,\sfn}) \Res_{\Lambda_\sfK}} \leq \eta\,.
\end{align}
Furthermore, as a consequence of \eqref{eq:scaling-diff-DX-DXtilde} we have
\begin{equation}\label{eq:scaling-etry-normal-to-equiv}
	1 - \delta \leq \bigl(\bar S^\cX_{\nut,\nu_\rx}\bigr)^{-1}   S^\cX_{n,\nut,\nu_\rx} \leq 1 + \delta, \quad (\nut,\nu_\rx) \in \vee .
\end{equation}
The function $w_{a_{\nut}}$ is positive on all of its domain. However, we can use 
\[
	\tfrac{dF}{dx}(e^{\frac{x}{2}})= 1-2e^{x/2} F(e^{x/2})
\]
together with the fact that the Dawson function is strictly increasing in the interval $[0,\frac{3}{4}]$ \cite{oeis} to conclude that $w_{a_{\nut}}(-kh_{a_{\nut}}) \geq 0$ for all $k > \frac{2}{h_{a_{\nut}}} \ln(\frac{4}{3})$. Therefore, for $n_{\nut} \geq \frac{2}{h_{a_{\nut}}} \ln(\frac{4}{3})$ we have by definition
\[
	{S}^\cX_{n,\nu_\rt,\nu_\rx} \leq {S}^\cX_{\nu_\rt,\nu_\rx}
\]
and by \eqref{eq:scaling-etry-normal-to-equiv}, $\bigl({S}^\cX_{\nu_\rt,\nu_\rx}\bigr)^{-1} {S}^\cX_{n_{\nu_\rt},\nu_\rt,\nu_\rx}\leq 1 + \delta$. Under the additional restrictions $\nu_\rx \in \Lambda_{K_{\nu_\rt}}$ and $n \geq M_{0,\nu_\rt}(K_{\nu_\rt})$, we have the corresponding lower bound $\bigl({S}^\cX_{\nu_\rt,\nu_\rx}\bigr)^{-1} {S}^\cX_{n_{\nu_\rt},\nu_\rt,\nu_\rx} \geq 1 - \delta$.
We summarize the above observation as follows.
\begin{remark}\label{rem:scaling-DX-delta}
	For the diagonal operators $\mathbf{\bar D}_\cX, \bDX, \mathbf{D}_{\cX,\sfn}$ we have
	\begin{equation}\label{eq:scaling-X-operator-change}
		\norm{\mathbf{\bar D}_\cX^{-1} \bDX} \leq 1 + \delta, \quad \norm{\bDX^{-1} \mathbf{\bar D}_{\cX}} \leq (1-\delta)^{-1} ,
	\end{equation}
	which means that the spectral condition number of $\bDX^{-1} \mathbf{\bar D}_{\cX}$ is bounded by $(1+\delta)/(1-\delta)$. Moreover, for $\sfn = (n_{\nu_\rt})_{\nu_\rt \in \vee_\rt}$ with $n_{\nu_\rt} \geq \frac{2}{h_{a_{\nut}}} \ln(\frac{4}{3})$, we have
	\begin{align*}
		\norm{\mathbf{\bar D}_\cX^{-1} \mathbf{D}_{\cX,\sfn}} \leq 1 + \delta
	\end{align*}
	and for any $\sfK = (K_{\nu_\rt})_{\nu_\rt \in \vee_\rt} > 1$ and $\eta > 0$,
	\begin{align*}
		(1-\delta) \norm{\mathbf{\bar D}_{\cX} \bv} \leq \norm{\mathbf{D}_{\cX,\sfn} \bv} \leq (1+\delta) \norm{\mathbf{\bar D}_{\cX} \bv} \quad \text{when} \quad \supp (\bv) \subset \Lambda_{\sfK}.
	\end{align*}
\end{remark}

To bound the ranks in the iteration, we will need an upper bound for $\MXnu(\eta;K)$ for given $\eta > 0, K \in \R$ and $\nut \in \veet$.
\begin{lemma}
	Let $0 < \eta < 1, K \in \R$ and $\nut \in \veet$. Then we have
	\begin{multline}\label{eq:MX-bound}
		\MXnu(\eta;K) \leq 1+ 2 h_{a_{\nut}}^{-1} \Bigl(C + \abs{\ln(\min\{\delta/2,\eta\})}  \\ + \tfrac{1}{2}\ln(a_{\nut}) + \max\{0,\tfrac{1}{2}\ln(K)-\ln(a_{\nut})\} \Bigr) ,
	\end{multline}
	where $C = \ln(\tfrac{4}{\sqrt{\pi}})$.
\end{lemma}
\begin{proof}
	We define
	\[
		g_{a_{\nut}}(y) = \tfrac{4}{\sqrt{\pi}} e^{-\frac{y}{2}} \sqrt{a_{\nut}} \max\left\{1,\tfrac{\sqrt{K}}{a_{\nut}}\right\} .
	\]
	Using that $\frac{dF}{dx}(x) = 1-2xF(x) \leq 1$ and $F(0) = 0$, one obtains $F(x) \leq x$ for all $x \geq 0$. Analogously, we have $\erf(x) \leq \frac{2}{\sqrt{\pi}} x$ for all $x > 0$. Combining this with $c + c^{-1} \leq 2 \max\{c,c^{-1}\}$ yields $f_{a_{\nut}}(s,y) \leq g_{a_{\nut}}(y)$ for all $y \in \R, s \in [1,K]$.
	Choosing the smallest $n \in \N$ such that $g_{a_{\nut}}(nh_{a_{\nut}}) \leq \eta$ yields
	\[
		n-1 \leq 2 h_{a_{\nut}}^{-1} \left(C + \abs{\ln(\min\{\delta/2,\eta\})} + \tfrac{1}{2}\ln(a_{\nut}) + \max\{0,\tfrac{1}{2}\ln(K)-\ln(a_{\nut})\} \right)
	\]
	with $C = \ln(\tfrac{4}{\sqrt{\pi}})$.
	Due to $f_{a_{\nut}}(s,nh_{a_{\nut}}) \leq g_{a_{\nut}}(nh_{a_{\nut}}) \leq \eta$ for all $s \in [1,K]$, we have $\MXnu(\eta;K) \leq n$, from which \eqref{eq:MX-bound} follows.
\end{proof}
Note that the upper bound does not only depend on the time index $\nut$ via $a_{\nut}$ but also on the spatial support at this given time index via $K_{\nut}$.

\begin{remark}
	The step size $h_{a_{\nut}}$ is not fixed, but depends by $a_{\nut}$ on the time index $\nut$. By definition, one has $a_{\nut} \eqsim 2^{\abs{\nu_\rt}}$. If we set $h_{a_{\nut}}$ to the upper bound in \eqref{eq:h-epxonential-sum}, we get $h_{a_{\nut}} \eqsim \abs{\nu_\rt}^{-1}$.
	Based on the result of $h_{a_{\nut}}$, one can easily see that $n^+_{a_{\nut}}$ grows quadratically in $\abs{\nu_\rt}$.
	For the lower bound for $n$, we obtain $n \geq \frac{2}{h_{a_{\nut}}} \ln(\frac{4}{3})$.
\end{remark}

\begin{remark}\label{rem:fullsep}
	In the present approach, $\bDXnut$ is approximated by a finite sum of separable terms independently for each $\nu_\rt \in \veet$, and the approximation of each $\bDXnut$ acts on precisely one low-rank representation associated to the time index $\nut$. 
	This depends on our particular approximation format with separate low-rank representations for each each $\nut$.
	When instead aiming for low-rank approximations treating the time-dependence as an additional tensor mode, one would instead need a direct low-rank approximation of $\mathbf{\bar D}_\cX$, which appears to be a substantially more difficult problem, and we are not aware of suitable error-controlled constructions of this type; see \cite{Andreev:15}, however, for heuristic approaches for obtaining low-rank approximations of a rescaling equivalent to $\mathbf{\bar D}_\cX$.
\end{remark}

\subsection{Scaling matrix $\bDY$}
The scaling operator $\mathbf{\bar D}_{\cY}$ has the form $\mathbf{\bar D}_{\cY} = \bI_\rt \otimes \mathbf{\bar D}$ with $\mathbf{\bar D}$ as in \eqref{eq:S-def}.
An equivalent scaling operator $\bD$ for $\mathbf{\bar D}$, which can be approximated by a finite sum of separable operators, is given in \cite[Section 4.1]{bachmayr2014adaptive}. In this chapter we will recapitulate the relevant results. Subsequently, we adapt these results to obtain $\bDY$. Therefore we use that we have the identity in the time and can approximate the scaling matrix $\mathbf{\bar D}$ independently for each time index corresponding to its associated spatial support.

We first recall the construction of $\bD$ and $\bD_n$. Let $\delta \in (0,1)$ and
\begin{align*}
	h \in \Big(0, \tfrac{\pi^2}{5(\abs{\ln \frac{\delta}{2}} + 4)} \Big], \quad n^+ &= \ceil{h^{-1} \max\{4\pi^{-\frac{1}{2}}, \sqrt{\abs{\ln(\delta/2)}}\}}, \\
	\alpha(x) = \ln^2(1+e^x), \quad w(x) &= 2\pi^{-\frac{1}{2}} (1+e^x)^{-1} .
\end{align*}
With the exponential sums
\begin{align*}
	\varphi_{h,n}(t) = \sum\limits_{k=-n}^{n^+} h w(kh) e^{-\alpha(kh)t}, \quad \varphi_{h,\infty}(t) = \lim\limits_{n \to \infty} \varphi_{h,n}(t)
\end{align*}
we define $\bD$ and $\bD_n$ by
\begin{align*}
	\bD_n \bv &= \bigl(S^V_{n,\nu}  \bv_{\nu}\bigr)_{\nu \in \vee_\rx}, \quad \text{where} \quad {S}^V_{n,\nu} = S_{\min}  \varphi_{h,n}\bigl( S_{\min}^2 \norm{\Psi_\nu}_V^2 \bigr), \\
	\bD \bv &= \bigl(S^V_{\nu} \bv_{\nu}\bigr)_{\nu \in \vee_\rx}, \quad \text{where} \quad \tilde{S}^V_{\nu} = S_{\min} \varphi_{h,\infty}\bigl(S_{\min}^2\norm{\Psi_\nu}_V^2 \bigr).
\end{align*}
Next, we state the relevant results on the equivalent scaling matrix, where we refer to \cite[Section 4.1]{bachmayr2014adaptive} for more details.
For $\eta > 0$, $h \in \R$, $0<\delta<1$ and $K \in \R$, let
\begin{equation}\label{eq:MY-bound}
	\begin{aligned}
	\mY(\eta;K) &= \lceil h^{-1} (\ln(2\pi^{-\frac{1}{2}}) + \abs{\ln(\min\{\delta/2,\eta\})} + \tfrac{1}{2} \ln(K)) \rceil ,\\
	\mYZe(K) &= \mY(\delta/2,K) .
	\end{aligned}
\end{equation}
Provided that $n \geq \mY(\eta;K)$, we have
\begin{align}\label{eq:scaling-S}
	\norm{\mathbf{\bar D}^{-1}(\mathbf{\bar D} -\bD_n) \Res_{\Lambda_K}} \leq \delta, \quad \norm{\mathbf{\bar D}^{-1}(\bD-\bD_n) \Res_{\Lambda_K}} \leq \eta .
\end{align}

\begin{remark}[{see \cite[Remark 11]{bachmayr2014adaptive}}]\label{rem:scaling-S}
	For the diagonal operators $\mathbf{\bar D}, \bD$ and $\bD_n$, we have
	\begin{align*}
		\norm{\mathbf{\bar D}^{-1} \bD_n}, \norm{\mathbf{\bar D}^{-1} \bD} \leq 1+\delta, \quad n \in \N_0, \quad \norm{\bD^{-1}\mathbf{\bar D}} \leq (1-\delta)^{-1} ,
	\end{align*}
	and in particular, the spectral condition number of $\bD^{-1}\mathbf{\bar D}$ is bounded by $(1+\delta)/(1-\delta)$. Moreover, for any $K > 1$ and $n \geq \mYZe(K)$,
	\begin{align*}
		(1-\delta) \norm{\mathbf{\bar D} \bv} \leq \norm{\bD_n \bv} \leq (1+\delta) \norm{\mathbf{\bar D} \bv} \quad \text{when} \quad \supp(\bv) \subseteq \Lambda_K .
	\end{align*}
\end{remark}
Based on this knowledge, we are able to define an equivalent scaling operator for $\mathbf{\bar D}_\cY$. We set $\bDY = \bI_\rt \otimes \bD$, which still has unbounded ranks. For the rank-truncated version we allow a different rank for each temporal index: for $\sfn \in \N_0^{\vee_\rt}$, we define
\begin{align*}
	\bD_{\cY,\sfn} = \sum\limits_{\nu _\rt \in \vee_\rt} \mathbf{E}_{\nu_\rt} \otimes \bD_{n_{\nu_\rt}} .
\end{align*}
For $\sfK = (K_{\nu_\rt})_{\nu_\rt \in \vee_\rt}$, we define
\begin{equation*}
	\MY(\eta;\sfK) = (\mY(\eta;K_{\nu_\rt}))_{\nu_\rt \in \vee_\rt}, \qquad
	\MYZe(\sfK) = (\mYZe(K_{\nu_\rt}))_{\nu_\rt \in \vee_\rt} .
\end{equation*}
Then if $\sfn \geq \MY(\eta;\sfK)$, by \eqref{eq:scaling-S}, we have
\begin{equation*}
	\norm{\mathbf{\bar D}_{\cY}^{-1} (\mathbf{\bar D}_{\cY} - \mathbf{D}_{\cY,\sfn}) \Res_{\Lambda_\sfK}} \leq \delta, \quad \norm{\mathbf{\bar D}_\cY^{-1}(\bDY - \bD_{\cY,\sfn}) \Res_{\Lambda_\sfK}} \leq \eta.
\end{equation*}

\begin{remark}\label{rem:scaling-DY-delta}
	For the diagonal operators $\mathbf{\bar D}_{\cY}$, $\bD_\cY$ and $\mathbf{D}_{\cY,\sfn}$, as a direct consequence of Remark \ref{rem:scaling-S}, we have
	\begin{align}\label{eq:scaling-Y-operator-change}
		\norm{\mathbf{\bar D}_\cY^{-1} \mathbf{D}_{\cY,\sfn}}, \norm{\mathbf{\bar D}_\cY^{-1} \bDY} \leq 1+\delta, \quad \sfn \in \N_0^{\vee_\rt}, \quad \norm{\bDY^{-1} \mathbf{\bar D}_\cY} \leq (1-\delta)^{-1},
	\end{align}
	and for any $\sfK > 1$ and $\sfn \geq \MYZe(\sfK)$,
	\begin{align*}
		(1-\delta) \norm{\mathbf{\bar D}_\cY \bv} \leq \norm{\mathbf{D}_{\cY,\sfn} \bv} \leq (1+\delta) \norm{\mathbf{\bar D}_\cY \bv} \quad \text{when} \quad \supp(\bv) \subseteq \Lambda_\sfK .
	\end{align*}
\end{remark}

\section{Adaptive approximation of rescaled low-rank operators}\label{sec:Apply}
Given the equivalent scaling operators and their approximations constructed in the previous section, we now turn to the adaptive application of operators in low-rank representation. To achieve this for $\bB_1$, we need to consider operators of the form
\begin{align}\label{eq:operator-sep-form}
	\bT = \bT_\rt \otimes \bI_\rx + \bI_\rt \otimes \bT_\rx ,
\end{align} 
where with a certain rank parameter $R$ and $\KdR = \bigtimes_{i=1}^d \{1,\dots,R\}$, $\bT_\rx$ has the form
\begin{align}\label{eq:operator-low-rank-form}
	\bT_\rx = \sum\limits_{\sfn \in \KdR} c_\sfn \bigotimes_{i} \bT^{(i)}_{n_i},
\end{align}
with the component tensor $( c_{\sfn} )_{\sfn \in \KdR}$ in hierarchical format with ranks at most $R$. 

\begin{remark}\label{rem:laplaceT}
For the heat equation, where $\mathbf{T}_\rx$ corresponds to the representation of the Laplacian, we obtain the one-dimensional operators
\begin{align*}
	\bT^{(i)}_1 = \left( \inp{\psi_{\nu}}{\psi_{\mu}}\right)_{\nu,\mu \in \vee_{1}} = \bI_i , 
	\quad 
	\bT^{(i)}_2 = \left( \inp{\psi_{\nu}^\prime}{\psi_{\mu}^\prime}\right)_{\nu,\mu \in \vee_{1}}, \quad \bT_\rt = \left( \inp{\psi_{\nu}^\prime}{\psi_{\mu}}\right)_{\nu,\mu \in \vee_\rt}.
\end{align*}
In this case, we have $R=2$, and $c_{\sfn} = 1$ if $(n_1,\dots,n_d)$ is a permutation of $(2,1,\dots,1)$ and $c_{\sfn} = 0$ otherwise. It is natural to choose the same wavelet basis for each spatial dimension, which allows us to write $\bT_2 = \bT_2^{(i)}$. Hence in this case the operator $\bT_{\rx}$ takes the form
\begin{align}\label{eq:laplace-operator-representation}
	\bT_{\rx} = \bT_2 \otimes \bI_2 \otimes \cdots \otimes \bI_d + \cdots + \bI_1 \otimes \cdots \otimes \bI_{d-1} \otimes \bT_2 .
\end{align}
\end{remark}

In the first parts of this section we construct for a given $\bv \in \ell_2(\vee)$ of finite hierarchical ranks and any given tolerance $\eta > 0$ an approximation $\bw_{\eta} \in \ell_2(\vee)$, which satisfies the estimate $\norm{\bB_1 \bv - \bw_{\eta}} \leq \eta$ and satisfies bounds on hierarchical ranks and lower-dimensional support sizes $\# \supp (\piti(\bw_{\eta}))$ that are quasi-optimal in relation to $\eta$. In the last part of the section we turn to analogous results for $\bB_2$.

The approach for $\bB_1$ is to examine the spatial and temporal operator separately and subsequently combine the results. For both operators we make use of the properties of the scaling operators from Section \ref{sec:low-rank-preconditioning} and transfer the problem to operators with explicit low-rank format. The spatial operator can then be treated similarly to the elliptic case as in \cite{BachmayrNearOptimal,bachmayr2014adaptive}, whereas for the temporal operator some new concepts are needed due to interaction between the different hierarchical tensor representations associated to different temporal basis functions.

Let us first recall the following standard notion from \cite{CDD1} for the adaptive compressibility of operators. 

\begin{definition}\label{def-compressible}
	Let $\hat\vee$ be a countable index set, and let $s^* > 0$. An operator $\bM \colon \ell_2(\hat\vee) \to \ell_2(\hat\vee)$ is called \emph{$s^*$-compressible} if for any $0 < s < s^*$ there exist summable positive sequences $(\alpha_j)_j, (\beta_j)_j$ such that for each $j \in \N_0$ there exists $\bM_j$ with at most $\alpha_j 2^j$ nonzero entries per row and column satisfying $\norm{\bM- \bM_j} \leq \beta_j 2^{-sj}$. We denote the arising sequences for a given $s^*$-compressible operator by $\alpha(\bM)$ and $\beta(\bM)$. 
\end{definition}

In the following the compressibility is always dependent on the combination with the scaling matrices. First, we define the lower-dimensional scaling operators $\mathbf{\hat D}^{\tau} \colon \R^{\vee_1} \to \R^{\vee_1}$ and $\mathbf{\hat D}_\rt^{\tau} \colon \R^{\vee_\rt} \to \R^{\vee_\rt}$ by
\begin{align*}
	\mathbf{\hat D}^\tau \bv = \left( \norm{\psi_{\nu}}_{H^1_0(0,1)}^{-\tau} \bv_{\nu} \right)_{\nu \in \vee_1}, \quad  &\mathbf{\hat D} = \mathbf{\hat D}^1, \\
	\mathbf{\hat D}_\rt^\tau \bv = \left( \norm{ \theta_{\nu_\rt} }_{H^1(I)}^{-\tau} \bv_{\nu_\rt} \right)_{\nu_\rt \in \vee_\rt},  \quad  & \mathbf{\hat D}_\rt = \mathbf{\hat D}_\rt^1.
\end{align*}
The operators
\begin{align}\label{eq:operator-onedim-scaling}
	\bC_2 = \mathbf{\hat D} \bT_2 \mathbf{\hat D}, \quad \bC_{\rt} = \bT_\rt \mathbf{\hat D}_\rt 
\end{align}
are bounded for sufficiently regular time and spatial basis functions. Additionally, they are $s^*$-compressible: for each $s < s^*$, there exist sequences of approximations $(\bT_{2,j})_{j \in \N_0}$, $(\bT_{\rt,j})_{j \in \N_0}$ with
\begin{align}\label{eq:compressible-one-dim-scaling}
	\begin{split}
		\norm{\mathbf{\hat D} (\bT_2 - \bT_{2,j}) \mathbf{\hat D}} &\leq \beta_j(\bC_2) 2^{-sj}, \\
		\norm{(\bT_\rt - \bT_{\rt,j}) \mathbf{\hat D}_\rt} &\leq \beta_j(\bC_{\rt}) 2^{-sj}.
	\end{split}
\end{align}

\subsection{Spatial operator}\label{sec:Apply-spatial}
In the following we consider as in \eqref{eq:bAcx} an operator of the form
\begin{equation}\label{eq:Bxdef}
	\mathbf{B}_\rx = \bDY (\bI_\rt \otimes \bT_\rx) \bDX,
\end{equation}
where $\bT_\rx$ is of the form \eqref{eq:operator-low-rank-form}. 
We assume that the coefficients $\bT^{(i)}_{n_i}$ are $s^*$-compressible in the sense of \eqref{eq:compressible-one-dim-scaling} and $\bT^{(i)}_{n_i,j}$ are the corresponding approximations. For the adaptive application of the operator $\mathbf{B}_\rx$ on a given $\bv \in \ell_2(\vee)$ we want to combine, analogously to the one-dimensional and the high-dimensional elliptic case, the available a priori information on $\mathbf{B}_\rx$ ($s^*$-compressibility) with a posteriori information on $\bv$. We describe how to construct approximations $\bw_J$ to $\mathbf{B}_\rx \bv$ for refinement parameters $J \in \N_0$ and then express these results in terms of error tolerances. Our approach follows the lines of the high-dimensional elliptic case in \cite{bachmayr2014adaptive}, but uses the spatio-temporal contractions $\piti(\bv)$ from Definition \ref{def:contractions} to control the approximations. 

Let $\bar{\Lambda}^{(i)}_j(\bv)$ be the support of the best $2^j$-term approximation of the contractions $\piti(\bv)$. From this we extract the support for each time index $\nu_\rt$ and each space dimension $i$ by
\begin{align}\label{eq:best2jpitinu}
	\bar{\Lambda}^{(i)}_{\nu_\rt,j}(\bv) = \bigl\{\nu_\rx : (\nu_\rt,\nu_\rx) \in \bar{\Lambda}^{(i)}_j(\bv) \bigr\}, \quad j=0,\ldots,J.
\end{align}
With the indices for each time index we can proceed exactly as in the elliptic case \cite{BachmayrNearOptimal}. If $\bT^{(i)}_{n_i} = \bI_i$, we need no approximation and simply set $\bTtil^{(i)}_{n_i} = \bI_i$. Otherwise, for each $\nu_\rt \in \vee_\rt$, let $\bar{\Lambda}^{(i)}_{\nu_\rt,-1}(\bv) = \emptyset$ and
\begin{align}\label{eq:operator-space-restriction-sets}
	\Lambda^{(i)}_{\nu_\rt,[p]}(\bv) = \begin{cases}
		\bar{\Lambda}^{(i)}_{\nu_\rt,p}(\bv) \setminus \bar{\Lambda}^{(i)}_{\nu_\rt,p-1}(\bv), \quad &p = 0,\dots,J, \\
		\vee_{1} \setminus \bar{\Lambda}^{(i)}_{\nu_\rt,J}(\bv), \quad &p = J+1, \\
		\emptyset, \quad &p > J+1.
	\end{cases}
\end{align}
Correspondingly, we define the set including all respective time indices by
\begin{align}\label{eq:lambda-all-time-indices}
	\Lambda^{(\rt,i)}_{[p]}(\bv) = \bigl\{ (\nu_\rt,\nu_i) : \nu_\rt \in \vee_\rt, \nu_i \in \Lambda^{(i)}_{\nu_\rt,[p]}(\bv) \bigr\} .
\end{align}
Furthermore, let
\begin{align}\label{eq:operator-space-A}
	\bTtil^{(i)}_{n_i,[p]} = \begin{cases}
		\bT^{(i)}_{n_i,J-p}, \quad &p=0,\dots,J, \\
		0, \quad &p > J.
	\end{cases}
\end{align}

We now define the approximate application of the operator to a $\bw \in \ell_2(\vee)$ given in the form \eqref{eq:parabolic-lowrank-format} by
\begin{align*}
	\begin{split}
		 \bTtil_{\rx,J}[\bv] \bw &=  \sum\limits_{\nu_\rt \in \vee_\rt} \mathbf{e}_{\nu_\rt} \otimes \left(\bTtil_{\rx,\nu_\rt,J}[\bv]\bw_{\nu_\rt}\right) \\
		 &= \sum\limits_{\nu_\rt \in \vee_\rt} \mathbf{e}_{\nu_\rt} \otimes \Bigg(\sum\limits_{\sfn \in \KdR} c_\sfn \bigotimes_{i=1}^d \bTtil^{(i)}_{\nu_\rt,n_i}[\bv] \Bigg)\bw_{\nu_\rt} ,
	\end{split}
\end{align*}
where
\begin{align*}
	\bTtil^{(i)}_{\nu_\rt,n_i}[\bv] = \sum\limits_{p \in \N_0} \bTtil^{(i)}_{n_i,[p]} \Res_{\Lambda^{(i)}_{\nu_\rt,[p]}(\bv)} .
\end{align*}
Thus we approximate the operator $\bT_{\rx}$ in a different way for each time index, depending on the corresponding contractions of $\bv$. 
To avoid technicalities, we give the proof of the operator approximation error estimates for the Laplacian, that is, for $\bT_\rx$ of the form \eqref{eq:laplace-operator-representation}.

\begin{remark}\label{rem:generalellipt}
The following estimates can be adapted to the case of more general second-order elliptic operators with constant coefficients as in \cite[Section 6.2]{bachmayr2014adaptive}. 
Since in $\bB_\rx$, the spatial part of the operator can be treated independently for each temporal index, one can immediately apply the bounds on terms containing mixed derivatives in \cite[Thm.~34, Lemmas 35 and 36]{bachmayr2014adaptive} to obtain a result analogous to Lemma \ref{lem:operator-apply-J-space-can-scale} below also for second-order operators with constant diffusion tensors, where the ranks of the hierarchical tensor representations of the diffusion tensors enter in the corresponding error bounds. For example, in the case of constant tridiagonal diffusion tensors, these ranks are bounded by five, see \cite[Example 5]{bachmayr2014adaptive}.
\end{remark}

We define the approximation
\begin{align}\label{eq:operator-approx-canoical-rescale-space}
	\bBtil_{\rx,J}(\bv) = \bDY  \bTtil_{\rx,J}[\bv] \bDX \bv ,
\end{align}
where $\bTtil_{\rx,J}$ depends on $\bv$ via the partitions \eqref{eq:operator-space-restriction-sets}.
To simplify notation, according to \eqref{eq:operator-onedim-scaling} we define 
\begin{align*}
	\mathbf{\tilde C}_2 = \mathbf{\hat D} \bTtil_2 \mathbf{\hat D}. 
\end{align*}
Before we can analyze the approximation we need a relation between the effects of high-dimensional and one-dimensional scaling matrices.
\begin{lemma}\label{lem:scaling-to-one-dim}
	For $\bM_\rx \in \R^{\vee_{1} \times \vee_{1}}, \bM_\rt \in  \R^{\vee_\rt \times \vee_\rt}$ and $\nu_\rt \in \vee_\rt$, one has
	\begin{align}
		\lrnorm{ \bD [\bI_1 \otimes \cdots \otimes \bI_{i-1} \otimes \bM_\rx \otimes \bI_{i+1} \otimes \cdots \otimes \bI_d] \bDXnut} &\leq C_\delta \norm{\mathbf{\hat D} \bM_\rx \mathbf{\hat D}}, \label{eq:scaling-to-onedim-x}\\
		\lrnorm{\bDY [\bM_\rt \otimes \bI_\rx] \bDX} &\leq C_\delta \lrnorm{\bM_\rt \mathbf{\hat D}_\rt} , \label{eq:scaling-to-onedim-t}
	\end{align}
	with the canonical interpretation for $i=1$ and $i=d$, and where $C_\delta = (1+\delta)^2$.
\end{lemma}
\begin{proof}
	The estimate \eqref{eq:scaling-to-onedim-x} follows from 
	\begin{align*}
		\norm{ [\bI_1 \otimes \cdots \otimes \bI_{i-1} \otimes \mathbf{\hat D}^{-1} \otimes \bI_{i+1} \otimes \cdots \otimes \bI_d] \bDXnut} &\leq  1+\delta , \\
		\norm{\bD [\bI_1 \otimes \cdots \otimes \bI_{i-1} \otimes \mathbf{\hat D}^{-1} \otimes \bI_{i+1} \otimes \cdots \otimes \bI_d] } &\leq  1+\delta,
	\end{align*}
	where we used \eqref{eq:scaling-X-operator-change}, \eqref{eq:scaling-Y-operator-change} and the definition of $\mathbf{\bar D}_{\cX}$ and $\mathbf{\bar D}_{\cY}$ respectively.
	Due to the particular structure of the operator, 
	\begin{align*}
		\bDY (\bM_\rt \otimes \bI_\rx) = (\bI_\rt \otimes \bD) (\bM_\rt \otimes \bI_\rx) = (\bM_\rt \otimes \bI_\rx) (\bI_\rt \otimes \bD) = (\bM_\rt \otimes \bI_\rx) \bDY .
	\end{align*}
	Combining this fact with \eqref{eq:scaling-X-operator-change}, \eqref{eq:scaling-Y-operator-change} and $\norm{(\mathbf{\hat D}_\rt^{-1} \otimes \bI_\rx) \mathbf{\bar D}_{\cY} \mathbf{\bar D}_{\cX}} \leq 1$  yields \eqref{eq:scaling-to-onedim-t}.
\end{proof}

Now we are able to analyze the adaptive application with respect to the canonical scaling matrices. 
\begin{lemma}\label{lem:operator-apply-J-space-can-scale}
	Let $\mathbf{B}_{\rx} = \bDX (\bI_\rt \otimes \bT_\rx) \bDY$ be defined by \eqref{eq:Bxdef}. Assume that \eqref{eq:compressible-one-dim-scaling} holds for any $s < s^*$. Additionally, let $\bv, \bw \in \ell_2(\vee)$ with $\piti(\bv) \in \cAs$ for $i=1,\dots,d$. Then for each $J \in \N_0$ and $\bBtil_{\rx,J}$ defined in \eqref{eq:operator-approx-canoical-rescale-space} with the $\bv$-dependent partitions \eqref{eq:operator-space-restriction-sets}, we have the a posteriori bound
	\begin{align}\label{eq:res-op-canocial-scaling-space}
		\norm{\mathbf{B}_{\rx} \bw - \bDY \bTtil_{\rx,J}[\bv]\bDX \bw} \leq  e_{\rx,J}[\bv](\bw),
	\end{align}
	with 
	\begin{align*}
		e_{\rx,J}[\bv](\bw) = C_\delta \sum\limits_{i=1}^d  \sum\limits_{p=0}^J 2^{-s(J-p)} \beta_{J-p}(\bC_2) \Bignorm{\Res_{\Lambda^{(t,i)}_{[p]}(\bv)} \piti(\bw)} + \norm{\bC_2} \Bignorm{\Res_{\Lambda^{(t,i)}_{[J+1]}(\bv)} \piti(\bw)}
	\end{align*}
	and $C_\delta$ as in Lemma \ref{lem:scaling-to-one-dim}, as well as the a priori estimate
	\begin{align}\label{eq:res-op-canocial-scaling-space-As}
		\norm{\mathbf{B}_{\rx} \bv - \bBtil_{\rx,J}(\bv)} \leq 2^s C_\delta 2^{-sJ}  \left( \lrnorm{\bC_2} + \lrnorm{\beta(\bC_2)}_{\ell_1}\right) \sum\limits_{i=1}^d  \lrnorm{\piti(\bv)}_{\cAs} .
	\end{align}
	Moreover, for $i=1,\ldots,d$, we have the support bound
	\begin{align}\label{eq:supp-op-canonical-scale-space}
		\# \supp\left(\piti\left(\bBtil_{\rx,J} (\bv) \right)\right) \leq 2 \norm{\alpha(\bC_2)}_{\ell_1} 2^J .
	\end{align}
\end{lemma}
\begin{proof}
We start with the error bound \eqref{eq:res-op-canocial-scaling-space}. Due to the structure \eqref{eq:laplace-operator-representation} of $\bT_\rx$, 
\begin{align*}
	 &\norm{\mathbf{B}_{\rx} \bv - \bDY \bTtil_{\rx,J}[\bv]\bDX \bw} \\
	\leq &\lrnorm{\sum\limits_{\nu_\rt \in \vee_\rt}  \bDY \Bigl(\bI_\rt \otimes  (\bT_{2} - \bTtil^{(1)}_{\nu_\rt,2}[\bv]) \otimes \bI_2 \otimes \cdots \otimes \bI_d \Bigr) \bDX (\mathbf{e}_{\nu_\rt} \otimes \bw_{\nu_\rt})} 
	+ \cdots + \\
	 &\quad \lrnorm{\sum\limits_{\nu_\rt \in \vee_\rt} \bDY \Bigl( \bI_\rt \otimes \bI_1 \otimes \cdots \otimes \bI_{d-1} \otimes (\bT_{2} - \bTtil^{(d)}_{\nu_\rt,2}[\bv])\Bigr) \bDX (\mathbf{e}_{\nu_\rt} \otimes \bw_{\nu_\rt})} .
\end{align*}
By definition of the one-dimensional approximations, we have
\begin{align*}
		&\lrnorm{\sum\limits_{\nu_\rt \in \vee_\rt}  \bDY \Bigl(\bI_\rt \otimes  (\bT_{2} - \bTtil^{(1)}_{\nu_\rt,2}[\bv]) \otimes \bI_2 \otimes \cdots \otimes \bI_d  \Bigr) \bDX (\mathbf{e}_{\nu_\rt} \otimes \bw_{\nu_\rt})} \\
		\leq & \sum_{p \in \N_0} \lrnorm{\sum\limits_{\nu_\rt \in \vee_\rt}  \Bigl(\bI_\rt \otimes \bD \Bigl( (\bT_{2} - \bTtil^{(1)}_{2,[p]}) \Res_{\Lambda^{(1)}_{\nu_\rt,[p]}(\bv)} \otimes \bI_2 \otimes \cdots \otimes \bI_d  \Bigr) \bDXnut \Bigr) (\mathbf{e}_{\nu_\rt} \otimes \bw_{\nu_\rt})}
	\end{align*}
and the analogous estimates for $i=2,\dots,d$.

Using that the support of $\mathbf{B}_\rx(\mathbf{e}_{\nu_\rt} \otimes \bw_{\nu_\rt})$ is pairwise disjoint for each time index $\nu_\rt$, the fact that diagonal operators commutes, and \eqref{eq:scaling-to-onedim-x} from Lemma \ref{lem:scaling-to-one-dim}, we obtain
\begin{align*}
	&\lrnorm{\sum\limits_{\nu_\rt \in \vee_\rt}  \Bigl(\bI_\rt \otimes \bD \Bigl( (\bT_{2} - \bTtil^{(1)}_{2,[p]}) \Res_{\Lambda^{(1)}_{\nu_\rt,[p]}(\bv)} \otimes \bI_2 \otimes  \cdots \otimes \bI_d \Bigr) \bDXnut \Bigr) (\mathbf{e}_{\nu_\rt} \otimes \bw_{\nu_\rt})}^2 \\
	\leq &\sum\limits_{\nu_\rt \in \vee_\rt}  \lrnorm{\bD \Bigl( (\bT_{2} - \bTtil^{(1)}_{2,[p]}) \otimes \bI_2 \otimes \cdots \otimes \bI_d \Bigr) \bDXnut}^2  \lrnorm{\Res_{\Lambda^{(1)}_{\nu_\rt,[p]}(\bv)} \pi^{(1)}(\bw_{\nut})}^2 \\
	\leq & C_\delta \lrnorm{\mathbf{\hat D}(\bT_2 - \bT_{2,J-p}) \mathbf{\hat D}}^2  \lrnorm{\Res_{\Lambda^{(\rt,1)}_{[p]}(\bv)} \pi^{(\rt,1)}(\bw)}^2 . 
\end{align*}
Furthermore, by $s^*$-compressibility we have $\norm{\mathbf{\hat D}(\bT_2 - \bT_{2,J-p}) \mathbf{\hat D}}\leq \beta_{J-p}(\bC_2) 2^{-s(J-p)}$ for $p=0,\dots,J$. By our construction, we obtain $\norm{\bT_{2} - \bTtil^{(i)}_{2,[J+1]}} = \norm{\bT_{2}}$ and $\norm{\Res_{\Lambda^{(\rt,i)}_{[p]}} \piti(\bv)} = 0$ for $p > J+1$. Combining this with the previous estimates yields \eqref{eq:res-op-canocial-scaling-space}.

By the choice of the index sets $\Lambda^{(\rt,i)}_{[p]}(\bv)$ and the definition of $\norm{\cdot}_{\cA^s}$, we obtain 
\[ 
	\norm{\Res_{\Lambda^{(\rt,i)}_{[p]}(\bv)} \piti(\bv)} \leq 2^{-s(p-1)} \norm{\piti(\bv)}_{\cA^s} , \quad p=0,\dots,J+1,
\]
which confirms \eqref{eq:res-op-canocial-scaling-space-As}. 
		
Since the scaling operators are diagonal, they leave the supports of approximations unchanged. Therefore, as in the one-dimensional case one has
\begin{equation}\label{eq:suppbound}
	\# \supp\left(\piti(\bBtil_{\rx,J}( \bv))\right) \leq 2 (\alpha_J(\bC_2) 2^J 2^0 + \cdots + \alpha_0(\bC_2) 2^0 2^J) \leq 2 \norm{\alpha(\bC_2)}_{\ell_1} 2^J,
\end{equation}
which shows \eqref{eq:supp-op-canonical-scale-space}; the additional factor of two in \eqref{eq:suppbound} results from the hierarchical representation rank of the Laplacian.
\end{proof}

We have arrived at estimates of a similar nature as in the elliptic case in \cite{bachmayr2014adaptive}, which is natural in the sense that $\bB_\rx$ does not involve interactions between the tensor representations of different temporal basis indices.

\subsection{Temporal operator}\label{sec:Apply-temporal}
We now consider the adaptive approximation of the temporal operator 
\begin{equation}\label{eq:Btdef}
	\bB_{\rt} = \bDY (\bT_\rt \otimes \bI_{\rx}) \bDY.
\end{equation}
Due to the possible rank increase based on the interaction between hierarchical tensor formats of different temporal basis indices, we need a slightly more restrictive type of compressibility of the (rescaled) representation of the time derivative $\bC_{\rt}$ introduced in \eqref{eq:operator-onedim-scaling}.

\begin{definition}\label{def:super-compressible}
	Let $\hat{\vee}, \tilde{\vee}$ be countable index sets. An operator $\bM \colon \ell_2(\hat{\vee}) \to \ell_2(\tilde{\vee})$ is called \emph{super-compressible} if there exist a summable positive sequence $(\beta_j)_{j \geq 0}$ and a $c>0$ such that for each $j \in \N_0$, there exists $\bM_j$ with $\norm{\bM - \bM_j} \leq \beta_j 2^{-j}$ and for each $j \in \N_0$ the operator $\bM_{j+1} - \bM_j$ has only $c$ non-zero entries per row and column. 
\end{definition}
\begin{remark}\label{rem:super-compressible}
	By modifying the sequence $(\beta_j)_j$ and the constant $c > 0$, one can match any rate $s$ in Definition \ref{def-compressible}. For super-compressible $\bT_\rt$, thus assume without loss of generality that $\norm{(\bT_\rt - \bT_{\rt,j})\mathbf{\hat D}_\rt} \leq \beta_j 2^{-sj}$ with $s$ as in the compressibility of the spatial operator $\bC_2$ in \eqref{eq:compressible-one-dim-scaling}. Additionally we set $\bT_{\rt,0} = 0$ and $\beta_0 = \norm{\bT_{\rt}\mathbf{\hat D}_\rt}$. Then, as immediate consequence of the definition, $\bT_{\rt,J}$ has at most $cJ$ non-zero entries in each row and column for each $J \in \N_0$.
	To simplify notation, we write $\beta^{(\rt)} = \beta(\bC_t)$.
\end{remark}

\begin{remark}\label{rem:supercompressible-realization}
	This new notion of compressiblity of operator representations is significantly more restrictive, but sufficient for our purposes: since the Donovan-Geronimo-Hardin multiwavelets that we use are spline functions with sufficiently many vanishing moments, we obtain this type of compressibility by observations studied in a general setting of piecewise smooth wavelets in \cite{Stevenson:04}; specifically, for the temporal multiwavelet basis $\{ \theta_\nu\}_{\nu\in\vee_\rt}$ and $\nu,\nu'\in\vee_\rt$ such that $\theta_{\nu}$ and $\theta_{\nu'}$ vanish at the boundary points of $I$, we have
	\[
	   \int_I \theta_{\nu}' \theta_{\nu'} \sdd x = -\int_I \theta_{\nu} \theta_{\nu'}' \sdd x  = 0
	\]
	whenever $\supp \theta_{\nu}$ does not contain any node of the spline $\theta_{\nu'}$ and vice versa.
	Due to the use of spline wavelets, the entries of the matrix $\bT_{\rt}$ can be evaluated exactly by numerical integration. Therefore we can assume that our approximation sequence is nested in the sense that for any $j_1 < j_2$ and $\nu,\mu$ with $(\bT_{\rt,j_1})_{\nu,\mu} \neq 0$, one has $(\bT_{\rt,j_1})_{\nu,\mu} = (\bT_{\rt, j_2})_{\nu,\mu}$. We thus obtain successive approximations $\bT_{\rt,j}$ and $\bT_{\rt,j+1}$ that differ only in $c$ entries per row and column. 
	Note that this property could also be achieved by other wavelets, see Remark \ref{rem:otherwavelets}.
\end{remark}

To construct an approximation of $\bB_\rt \bv$, as in the case of the spatial operator $\bB_\rx$, we use the spatio-temporal contractions $\piti(\bv)$. In what follows, we assume that $\bT_\rt$ is super-compressible as in Remark \ref{rem:super-compressible}.

We first subdivide the basis indices of the input $\bv$, where the basic idea is to combine the best $2^j$-term approximations of the individual spatial dimensions. We thus set 
\begin{equation}\label{eq:timeindexsets}
\begin{aligned}
	\bar{\Lambda}^{(\rt)}_j(\bv) &= \bigcup\limits_{\nu_\rt \in \vee_\rt} \{\nu_\rt\} \times \bar{\Lambda}_{\nu_\rt,j}^{(1)}(\bv) \times \cdots \times \bar{\Lambda}_{\nu_\rt,j}^{(d)}(\bv), \\
	\bv_j &= \Res_{\bar{\Lambda}_j^{(\rt)}(\bv)} \bv
\end{aligned}
\end{equation}
for each $j=0,\dots,J$ with $\bar{\Lambda}_{\nu_\rt,j}^{(i)}(\bv)$ from \eqref{eq:best2jpitinu}. We define the approximation by
\[
	\bBtil_{\rt,J} (\bv) = \bDY (\bT_{\rt,J} \otimes \bI_\rx) \bDX \bv_0 + \sum\limits_{j=1}^J \bDY(\bT_{\rt,J-j} \otimes \bI_\rx) \bDX (\bv_j - \bv_{j-1}).
\]
Due to the construction of $\bv_j$, which uses restrictions in all dimension, for the practical realization it is more convenient to use the equivalent formulation
\begin{equation}\label{eq:time-approx-def}
	\bBtil_{\rt,J} (\bv)= \sum\limits_{j=0}^{J-1} \bDY((\bT_{\rt,J-j} - \bT_{\rt,J-j-1}) \otimes \bI_\rx) \bDX \bv_j + \bDY(\bT_{\rt,0} \otimes \bI_\rx) \bDX \bv_J .
\end{equation}
Here we can make use of the super-compressibility, since according to Remark \ref{rem:super-compressible}, $(\bT_{\rt,j+1} - \bT_{\rt,j})$ has only $c$ non-zero entries in each row and column for each $j \in \N_0$. For later reference, we also note that 
\[
	\bBtil_{\rt,J} (\bv) = \bDY \mathbf{\tilde T}_{\rt,J}[\bv] \bDX \bv,
\]
where the operator $\mathbf{\tilde T}_{\rt,J}[\bv]$ is defined by
\[
	\mathbf{\tilde T}_{\rt,J}[\bv]  = 
	\sum\limits_{j=0}^{J-1} \bigl( (\bT_{\rt,J-j} - \bT_{\rt,J-j-1}) \otimes \bI_\rx \bigr) \Res_{\bar{\Lambda}_j^{(\rt)}(\bv)} + (\bT_{\rt,0} \otimes \bI_\rx) \Res_{\bar{\Lambda}_J^{(\rt)}(\bv)} \,.
\]

Note that $\bar{\Lambda}^{(\rt)}_{j}(\bv) \subseteq \bar{\Lambda}^{(\rt)}_{j+1}(\bv)$ for each $j \in \N_0$. Combining this with \eqref{eq:bound-restriction-contraction-mod}, we arrive at
\begin{equation}\label{eq:time-operator-vj-diff}
	\begin{aligned}
		\Bignorm{\Res_{\bar{\Lambda}^{(\rt)}_{j+1}(\bv)}\bw - \Res_{\bar{\Lambda}^{(\rt)}_{j}(\bv)} \bw} &= \Bignorm{\Res_{\bar{\Lambda}^{(\rt)}_{j+1}(\bv)}\bw - \Res_{\bar{\Lambda}^{(\rt)}_{j}(\bv)} \Big(\Res_{\bar{\Lambda}^{(\rt)}_{j+1}(\bv)}\bw\Big)} \\
		&\leq \sum\limits_{i=1}^d \Bignorm{\Res_{\bar{\Lambda}^{(i)}_{j+1}(\bv)} \piti(\bw) - \Res_{\bar{\Lambda}_{j}^{(i)}(\bv)} \piti(\bw)} \\
		&\leq \sum\limits_{i=1}^d \Bignorm{\Res_{\Lambda_{[j+1]}^{(\rt,i)}(\bv)} \piti(\bw)} 
	\end{aligned}
\end{equation}
for $\bw \in \ell_2(\vee)$, 
where $\Lambda_{[j+1]}^{(\rt,i)}(\bv)$ is defined in \eqref{eq:lambda-all-time-indices}.
\begin{remark}
	As an alternative, the index sets $\bar{\Lambda}^{(\rt)}_j(\bv)$ in \eqref{eq:timeindexsets} can also be defined by the index sets from the coarsening routine $\Lambda(\bv,N)$ from \eqref{eq:contraction-set-defs} with $N=2^j$. In the present setting, this leads to essentially the same results.
\end{remark}

\begin{lemma}\label{lem:operator-apply-J-time}
	Let $\bB_{\rt} = \bDY (\bT_\rt \otimes \bI_\rx) \bDX$ be defined by \eqref{eq:Btdef} and assume that $\bT_\rt$ is super-compressible. Given any $\bv, \bw \in \ell_2(\vee)$ and $J \in \N_0$, let $ \bBtil_{\rt,J} (\bv )$ be defined by \eqref{eq:time-approx-def}.
	Then whenever $\piti(\bv) \in \cAs$ for some $0 < s < s^*$ and $i=1,\dots,d$, one has the a posteriori error bound
	\begin{align}\label{eq:operator-apply-temporal-error_J-posteriori}
		\norm{\bB_{\rt} \bw -  \bDY \mathbf{\tilde T}_{\rt,J}[\bv] \bDX \bw} \leq e_{\rt,J}[\bv](\bw),
	\end{align}
	where
	\begin{align*}
		e_{\rt,J}[\bv](\bw) = C_{\delta}  \sum\limits_{i=1}^d \sum\limits_{j=0}^J 2^{-s(J-j)} \beta^{(\rt)}_{J-j} \Bignorm{\Res_{\Lambda_{[j]}^{(\rt,i)} (\bv) } \piti(\bw) } + \norm{\bC_{\rt}}  \Bignorm{\Res_{\Lambda_{[J+1]}^{(\rt,i)} (\bv)} \piti(\bw) }
	\end{align*}
	as well as the a priori error bound
	\begin{align}\label{eq:operator-apply-temporal-error_J-priori}
		\norm{\bB_{\rt} \bv -  \bBtil_{\rt,J} (\bv )} \leq 2^s C_{\delta} 2^{-sJ} (\norm{\beta^{(\rt)}}_{\ell_1} + \norm{\bC_{\rt}}) \sum\limits_{i=1}^d \norm{\piti(\bv)}_{\cAs} .
	\end{align}
	Moreover, one has the support estimate
	\begin{align}\label{eq:operator-apply-temporal-supp-J}
		\# \supp \left(\piti( \bBtil_{\rt,J} (\bv ))\right) \leq 2c 2^J, \quad i=1,\dots,d.
	\end{align}
\end{lemma} 
\begin{proof}
	We start with statement \eqref{eq:operator-apply-temporal-error_J-posteriori}. For simplicity, we define $\Lambda^{(\rt)}_{-1}(\bv)= \emptyset$. One has
	\begin{align*}
		 &\quad \norm{\bB_{\rt} \bw -  \bDY \mathbf{\tilde T}_{\rt,J}[\bv] \bDX \bw}  
		 \\ &\leq \sum\limits_{j=0}^J \norm{ \bDY((\bT_{\rt,J-j} - \bT_\rt)\otimes \bI_\rx)\bDX} \Bignorm{\Res_{\bar{\Lambda}^{(\rt)}_{j}(\bv)}\bw - \Res_{\bar{\Lambda}^{(\rt)}_{j-1}(\bv)} \bw} + \norm{\bB_{\rt}} \Bignorm{\bw - \Res_{\bar{\Lambda}^{(\rt)}_{J}(\bv)} \bw} \\
		 &\leq  C_{\delta}  \sum\limits_{j=0}^J \norm{(\bT_{\rt,J-j} - \bT_\rt) \hat{\bD}_\rt} \Bignorm{\Res_{\bar{\Lambda}^{(\rt)}_{j}(\bv)}\bw - \Res_{\bar{\Lambda}^{(\rt)}_{j-1}(\bv)} \bw} + C_{\delta} \Bignorm{\bw - \Res_{\bar{\Lambda}^{(\rt)}_{J}(\bv)}\bw} ,
	\end{align*}
	where we used \eqref{eq:scaling-to-onedim-t} from Lemma \ref{lem:scaling-to-one-dim} in the last inequality.
	Combining this with the super-compressibility and \eqref{eq:time-operator-vj-diff} results in \eqref{eq:operator-apply-temporal-error_J-posteriori}. 
	The a priori error bounds \eqref{eq:operator-apply-temporal-error_J-priori} follows exactly as for the spatial operator.
	
	As noted above, $(\bT_{\rt,j+1} - \bT_{\rt,j})$ has only $c$ non-zero entries in each column. In view of \eqref{eq:time-approx-def}, we thus have
	\begin{align*}
		\# \supp (\piti(\bv_j)) \leq 2^j
	\end{align*}
	for each $j=0,\dots,J$. Since the diagonal scaling matrices $\bDX$ and $\bDY$ leave supports unchanged, the statement \eqref{eq:operator-apply-temporal-supp-J} follows.
\end{proof}

\subsection{Combination of the operators}\label{sec:adaptive-operator-comb}
In this section we give the main results of the combination of the spatial operator $\bB_\rx$ and temporal operator $\bB_\rt$ with respect to a given error tolerance and the rank-truncated scaling operators. 
The following result is a simple consequence of Lemma \ref{lem:operator-apply-J-space-can-scale} and Lemma \ref{lem:operator-apply-J-time}.

\begin{cor}\label{cor:operator-apply-J-full}
	Under the assumptions of Lemma \ref{lem:operator-apply-J-space-can-scale} and Lemma \ref{lem:operator-apply-J-time} on $\bT_\rt$ and $\bT_\rx$, let $\bv \in \ell_2(\vee)$ have finite support and let 
	\begin{equation}\label{eq:B1Jdef}
		 \bBtil_{1,J} (\bv) = \bBtil_{\rt,J} (\bv) + \bBtil_{\rx,J} (\bv) 
	\end{equation}
	and 
	\[
	\mathbf{\tilde T}_J[\bv] = \mathbf{\tilde T}_{\rt,J}[\bv] +  \mathbf{\tilde T}_{\rx,J}[\bv].
	\]
	Then for $\bw \in \ell_2(\vee)$ and $e_J[\bv](\bw) = e_{\rt,J}[\bv](\bw) + e_{\rx,J}[\bv](\bw)$, one has the a posteriori bound
	\begin{align*}
		\norm{\bB_1 \bw - \bDY 	\mathbf{\tilde T}_J[\bv] \bDX \bw} \leq e_J[\bv](\bw),
	\end{align*}
	as well as the a priori error bound
	\begin{align*}
		\norm{\bB_1 \bv - \bBtil_{1,J} (\bv)} \leq 2^s C_{\delta}\tilde{C} 2^{-sJ} \sum\limits_{i=1}^d  \bignorm{\piti(\bv)}_{\cAs}
	\end{align*}
	with $\tilde C = \norm{\beta^{(\rt)}}_{\ell_1} + \norm{\bC_\rt} + \norm{\bC_2} + \norm{\beta(\bC_2)}_{\ell_1}$. Furthermore
	\begin{align}\label{eq:apply-approx-J-supp}
		\# \supp \Bigl(\piti\bigl(\bBtil_{1,J} (\bv)\bigr)\Bigr) \leq 2(c + \norm{\alpha(\bC_2)}_{\ell_1}) 2^J .
	\end{align}
\end{cor}
For notational simplicity we write $e_J(\bv) = e_J[\bv](\bv)$.
\begin{remark}
	Whenever $\bv$ is finitely supported, there exists a $p(\bv) \in \N_0$ such that
	$\Lambda^{(\rt,i)}_{[p]}(\bv) = \emptyset$ for $i=1,\dots,d$ for all $p > p(\bv)$. The quantity $e_J(\bv)$ can be computed for each $J \in \N_0$, because all entries of the sum over $p$ are zero for $J \geq p(\bv)$. Increasing of $J \geq p(\bv)$ further will then lead to a decrease in all summands on the right-hand side. 
	Hence, for a fixed $s < s^*$, for each $\eta > 0$ there exists a minimal $J \in \N_0$ such $e_J(\bv) \leq \eta$.
\end{remark}

The scaling matrices $\bDX$ and $\bDY$ still have unbounded ranks. Thus further approximation of the scaling matrices by $\mathbf{D}_{\cX,\sfn_1}$ and $\mathbf{D}_{\cY,\sfn_2}$ with bounded ranks for each temporal basis index are required.

\begin{lemma}\label{lem:scaling-to-truncated}
	Let $\mathbf{T}$ be defined as in \eqref{eq:operator-sep-form}, let $\mathbf{\tilde T}\in \R^{\vee\times \vee}$, $\bv \in \ell_2(\vee)$ and $\sfK_1,\sfK_2 \in \N_0^{\vee_\rt}$ be such that
	\begin{align*}
		\supp(\bv) \in \Lambda_{\sfK_1}, \quad \supp(\bDY \bTtil \bDX \bv) \in \Lambda_{\sfK_2},
	\end{align*}
	and let $\bG = \bDY (\bTtil - \bT) \bDX$.
	 Then whenever $\sfn_1 \geq \MX(\eta,\sfK_1)$, $\sfn_2 \geq \MY(\eta,\sfK_2)$, we have
	\begin{multline}
		\lrnorm{(\bDY \bT \bDX - \mathbf{D}_{\cY,\sfn_2} \bTtil \mathbf{D}_{\cX,\sfn_1})\bv}  \leq \norm{{\bG}\bv} + \norm{{\bG} (\bI - \bDX^{-1}\mathbf{D}_{\cX,\sfn_1}) \bv} \\ 
	+ \tfrac{\eta}{1-\delta}\norm{{\bG}(\bDX^{-1}\mathbf{D}_{\cX,\sfn_1})\bv} + \tfrac{2\eta}{1-\delta} \norm{\bB_1} \norm{\bv} .
	\end{multline}
\end{lemma}
The proof can be done analogously to \cite[Lemma 32]{bachmayr2014adaptive} using the properties of the scaling matrices established in Section \ref{sec:low-rank-preconditioning}. An important point to mention here is that $\sfn_1$ is chosen such that each diagonal element of $\bDX^{-1}\mathbf{D}_{\cX,\sfn_1}$ is bounded from above by one.

We now prove an error bound for the rank-truncated scaling matrices, where for a given $\bv \in \ell_2(\vee)$ and a tolerance $\eta > 0$, we choose the parameters
\begin{align*}
	J(\eta;\bv) = \min \bigl\{J \in \N_0  :e_J(\bv) \leq \eta/4 \bigr\}, \quad c(\bv) \eta = \tfrac{\eta (1-\delta)}{4\norm{\bB_1} \norm{\bv}}
\end{align*}
and
\begin{align*}
	\mcX(\eta;\bv) = \MX\bigl(c(\bv)\eta;\sfK(J(\eta;\bv);\bv)\bigr), \quad 
	\mcY(\eta;\bv) = \MY\bigl(c(\bv)\eta;\sfK(J(\eta;\bv);\bv)\bigr)
\end{align*}
with 
\begin{align}\label{eq:K-def-LamdaK}
	\sfK(J;\bv) = \min \{ \sfK > 0 : \supp(\bv) \cup \supp(\bTtil_J[\bv] \bv) \subseteq \Lambda_{\sfK} \} .
\end{align}
Here the minimum is to be understood element-wise.

\begin{theorem}\label{thm:error-rescaled-trunc}
	Let $\bv \in \ell_2(\vee)$ and $0 < \eta < 2 \norm{\bB_1} \norm{\bv}$. We fix
	\[
		J = J(\eta;\bv),\ \sfK = \sfK(J(\eta;\bv),\bv),\ \sfn_1 = \mcX(\eta;\bv),\ \sfn_2 = \mcY(\eta;\bv),\ \zeta = c(\bv)\eta.
	\]
	Then with $\bw_{\eta} = \mathbf{D}_{\cY,\sfn_2} \bTtil_J[\bv] \mathbf{D}_{\cX,\sfn_1} \bv$, we have
	\begin{equation}\label{eq:error-op-truncate-eta}
		\norm{\bB_1 \bv - \bw_{\eta}} \leq \eta\,.
	\end{equation}
\end{theorem}
\begin{proof}
	To simplify notation, we write $\bTtil_J = \bTtil_J[\bv]$.
	 Lemma \ref{lem:scaling-to-truncated} yields, with $\mathbf{\tilde d} =  (\bI - \bDX^{-1}\mathbf{D}_{\cX,\sfn_1}) \bv$,  $\mathbf{\tilde v} = (\bDX^{-1}\mathbf{D}_{\cX,\sfn_1})\bv$,
	\begin{align*}
		\norm{\bB_1 \bv - \mathbf{D}_{\cY,\sfn_2} \bTtil_{J} \mathbf{D}_{\cX,\sfn_1} \bv} 
		&\leq \norm{\bB_1 \bv - \bBtil_{1,J}(\bv)} + \norm{\bB_1 \mathbf{\tilde d} - \bDY \bTtil_{J}\bDX\mathbf{\tilde d}} \\ 
		&\quad + \frac{\zeta}{1-\delta}\norm{\bB_1 \mathbf{\tilde v} - \bDY \bTtil_{J}\bDX \mathbf{\tilde v}} + \frac{2\zeta}{1-\delta} \norm{\bB_1} \norm{\bv} \\
		&\leq e_J(\bv) + e_J[\bv](\mathbf{\tilde d}) + \frac{\zeta}{1-\delta} e_J[\bv](\mathbf{\tilde v})  + \frac{2\zeta}{1-\delta} \norm{\bB_1} \norm{\bv}.
	\end{align*}
	 By the choice of $\sfK$ and $\sfn_1$, we have $(\bDX^{-1}\mathbf{D}_{\cX,\sfn_1})_{\nu} \leq 1$ for $\nu \in \Lambda_\sfK$. Additionally, we get by \eqref{eq:scaling-X-operator-change} and \eqref{eq:scaling-X-error} the estimate
	\begin{align*}
		\lrabs{(\bI - (\bDX^{-1}\mathbf{D}_{\cX,\sfn_1})_{\nu}} \leq (1-\delta)^{-1} \zeta .
	\end{align*}
	We conclude
	\begin{align*}
		\norm{\bB_1 \bv -\mathbf{D}_{\cY,\sfn_2} \bTtil_J \mathbf{D}_{\cX,\sfn_1} \bv} &\leq  e_J(\bv) \left(1+\frac{2\zeta}{1-\delta} \right) + \frac{2\zeta}{1-\delta} \norm{\bB_1} \norm{\bv} \\
		&\leq 2 e_J(\bv)  + \frac{2\zeta}{1-\delta} \norm{\bB_1} \norm{\bv} .
	\end{align*}
	By the definition of $J$ and $\zeta$ as well as $\eta \leq 2 \norm{\bB_1} \norm{\bv}$, we arrive at \eqref{eq:error-op-truncate-eta}.
\end{proof}
Note that the condition $\eta < 2 \norm{\bB_1} \norm{\bv}$ is no restriction, since otherwise the solution can be approximated by zero.
Based on the above error bound, in the next step we give estimates on support sizes, ranks and $\cAs$-norms of approximations. Lemma \ref{lem:operator-apply-J-space-can-scale} and Lemma \ref{lem:operator-apply-J-time} yield support size estimates in dependence on the approximation parameter $J$. We now reformulate the estimates for a given tolerance $\eta$.

\begin{lemma}\label{lem:apply-quantity-estimates}
	Under the assumptions of Theorem \ref{thm:error-rescaled-trunc}, we have
	\begin{align}\label{eq:supp-apply-eta}
		\# \supp(\piti(\bw_{\eta})) \leq \hat{\alpha}(\bB_1) 4^{1+\frac{1}{s}} C_\delta^{\frac{1}{s}} \eta^{-\frac{1}{s}} \tilde C^{\frac{1}{s}} \Big( \sum\limits_{i=1}^d  \lrnorm{\piti(\bv)}_{\cAs}\Big)^{\frac{1}{s}} 
	\end{align}
	with $\hat{\alpha}(\bB_1) = 2(c+\norm{\alpha(\bC_2)}_{\ell_1})$,
	as well as
	\begin{align}\label{eq:apply-eta-As-norm}
		\norm{\piti(\bw_\eta)}_{\cAs} \leq  4^s C_\delta (\hat C_1 + \check{C}_{\rx} )  \norm{\piti(\bv)}_{\cAs} + 2^s \sqrt{C_\delta} \hat C_2 \sum\limits_{j=1}^d \norm{\pi^{(\rt,j)}(\bv)}_{\cAs},
	\end{align}
	where $\check{C}_\rx = 2 (d-1)  \norm{\bC_2}$, $\hat C_1=  \frac{2^{3s+2}}{2^s-1}  \norm{\alpha(\bC_2)}^s_{\ell_1} (2 \norm{\bC_2})$ and $\hat C_2=2c^s2^{4s}(\norm{\beta^{(\rt)}}_{\ell_1} + \norm{\bC_\rt})$.
	Moreover, we have the rank estimate
	\begin{align}\label{eq:apply-eta-rank}
		\rank_\infty(\bw_\eta)\leq \norm{\hmcX(\eta;\bv)}_{\ell_\infty} \norm{\hmcY(\eta;\bv)}_{\ell_\infty} (cJ(\eta;\bv) + 2) \rank_\infty(\bv),
	\end{align}
	where
	\begin{alignat*}{2}
		\hmcXnu(\eta;\bv) &= 1 + n_{a_{\nut}}^+(\delta) + \mcXnu(\eta;\bv), \quad &&\nut \in \vee_\rt, \\
		\hmcYnu(\eta;\bv) &= 1 + n^+(\delta) + \mcYnu(\eta;\bv), \quad &&\nut \in \vee_\rt .
	\end{alignat*}
\end{lemma}
\begin{proof}
	The support is not influenced by the scaling matrices due to the diagonal structure. By Corollary \ref{cor:operator-apply-J-full} one has $J(\eta;\bv) \leq \bar{J}(\eta;\bv)$ with
	\begin{align*}
		\bar{J}(\eta;\bv) = \argmin \Big\{ J \in \N_0 : 2^s C_\delta \tilde{C} 2^{-sJ}  \sum\limits_{i=1}^d  \lrnorm{\piti(\bv)}_{\cAs} \leq \frac{\eta}{4}\Big\} ,
	\end{align*}
	which yields
	\begin{align}\label{eq:J-eta-As-estimate}
		2^{\bar{J}(\eta;\bv)} \leq 4^{1+\frac{1}{s}} C_\delta^{\frac{1}{s}} \eta^{-\frac{1}{s}} \tilde C^{\frac{1}{s}} \left( \sum\limits_{i=1}^d \lrnorm{\piti(\bv)}_{\cAs}\right)^{\frac{1}{s}} .
	\end{align}
	Inserting this in \eqref{eq:apply-approx-J-supp} gives \eqref{eq:supp-apply-eta}.
	
	The rank estimate \eqref{eq:apply-eta-rank} is obtained taking the product of the bounds on the hierarchical ranks of each of the factors $\mathbf{D}_{\cY,\sfn_2}$, $\bTtil_J[\bv]$ and $\mathbf{D}_{\cX,\sfn_1}$. The factor $cJ(\eta;\bv) + 2$ results from the representation of the operator as $\bT = \bT_\rt \otimes \bI_\rx + \bI_\rt \otimes \bT_\rx$. The summand 2 is based on the rank of the spatial operator, which here is the representation of the Laplacian, whereas the summand $cJ(\eta;\bv)$ is based on Remark \ref{rem:super-compressible} by the super-compressibility property of the representation of the time derivative. We obtain the factor $cJ(\eta;\bv)$, because at most $c$ entries are non-zero in each row of $(\bT_{\rt,j+1} - \bT_{\rt,j})$ for each $j \in \N_0$ and there are $J(\eta;\bv)$ summands. 
	
	In the proof of \eqref{eq:apply-eta-As-norm}, we again suppress the $\bv$-dependence of index sets and operator approximations, and in particular we write $\bTtil_J = \bTtil_J[\bv]$.
	We start with the observation
	\begin{align*}
		\piti_{\nu_\rt,\nu_i}(\bw_\eta) = \piti_{\nu_\rt,\nu_i}(\mathbf{D}_{\cY,\sfn_2} \bTtil_J \mathbf{D}_{\cX,\sfn_1} \bv) \leq \piti_{\nu_\rt,\nu_i}(\bDY \bTtil_J \bDX \mathbf{\tilde v})
	\end{align*}
	for $(\nu_\rt,\nu_i) \in \supp(\piti(\bw_\eta))$ with $\mathbf{\tilde v} = \bDX^{-1} \mathbf{D}_{\cX,\sfn_1} \bv$.
	
	We will look separately at the operators $\bT_\rt \otimes \bI_\rx$ and $\bI_\rt \otimes \bT_\rx$ and make use of Proposition \ref{prop:cAs-properties}\eqref{prop:cAs-triangle}. We give the proof for $i=1$, the further values of $i$ can be treated in a completely analogous manner. We first consider $\bI_\rt \otimes \bT_\rx$.
	Due to the structure of $\bTtil_{\rx,J}$, we have
	\begin{equation} \label{eq:contraction-scaling}
	\begin{aligned}
		\pi^{(\rt,1)}_{\nu_\rt,\nu_1}(\bDY  \bTtil_{\rx,J} \bDX \mathbf{\tilde v}) &= \pi^{(1)}_{\nu_1}(\bD \bTtil_{\rx,\nu_\rt,J} \bDXnut \mathbf{\tilde v}_{\nu_\rt}) \\
		&\leq \pi^{(1)}_{\nu_1}\Bigl(\bD \Bigl(\bI_1 \otimes \sum\limits_{\sfn \in \KdOneTwo} c_\sfn \bigotimes_{i=2}^d \bTtil^{(i)}_{\nut, n_i}\Bigr) \bDXnut  \mathbf{\tilde v}_{\nu_\rt}\Bigr)\\
		&\quad + \pi^{(1)}_{\nu_1}\Bigl(\bD (\bTtil^{(1)}_{\nut,2} \otimes \bI_2 \otimes \cdots \otimes \bI_d) \bDXnut  \mathbf{\tilde v}_{\nu_\rt} \Bigr)  \\
		&\eqqcolon D_{1, \nu_\rt,\nu_1} + D_{2, \nu_\rt, \nu_1} ,
	\end{aligned}
	\end{equation}
	where we have also used our simplifying assumption on $\bT_\rx$.
	To estimate the first expression, for $i=1,\ldots,d$ we introduce the notation
	\[ 
		\mathbf{\check D}_i = \biggl( \Bigl(\sum_{j \neq i} \norm{\psi_{\nu_i}}_{H^1_0(0,1)}^2 \Bigr)^{-1/2} \delta_{\nu,\nu'}  \biggr)_{\nu,\nu' \in \vee_\rx}.
	\] 

	Since we are considering the Laplacian, we know that $c_\sfn = 1$ if $\sfn$ is a permutation of $(2,1,\dots,1)$ and zero elsewhere. Using a slightly modified version of Lemma \ref{lem:scaling-to-one-dim} and \cite[Lemma 35]{bachmayr2014adaptive}, we obtain
	\begin{align*}
		D_{1,\nu_\rt,\nu_1} &\leq (1+\delta) \pi^{(1)}_{\nu_1}\Bigl( \Bigl( \bI_1 \otimes \mathbf{\check D}_1\Bigl( \sum\limits_{\sfn \in \KdOneTwo} c_\sfn \bigotimes_{i=2}^d \bTtil^{(i)}_{\nut,n_i}\Bigr) \mathbf{\check D}_1\Bigr) (\bI_1 \otimes \mathbf{\check D}_1^{-1}) \bDXnut  \mathbf{\tilde v}_{\nu_\rt} \Bigr) \\
		&\leq (1+\delta)\sum\limits_{\sfn \in \KdOneTwo} \abs{c_\sfn} \lrnorm{\mathbf{\check D}_1 \bigotimes_{i=2}^d \bTtil^{(i)}_{\nut,n_i} \mathbf{\check D}_1} \pi_{\nu_1}^{(1)}((\bI_1 \otimes \mathbf{\check D}_1^{-1}) \bDXnut  \mathbf{\tilde v}_{\nu_\rt}) \\
		&\leq (1+\delta) \sum\limits_{i=2}^d \norm{\mathbf{\hat D} \bTtil_{\nut,2}^{(i)} \mathbf{\hat D}} \pi_{\nu_1}^{(1)}((\bI_1 \otimes \mathbf{\check D}_1^{-1}) \bDXnut  \mathbf{\tilde v}_{\nu_\rt}) \\
		&\leq  \check{C}_\rx (1 + \delta)^2 \pi^{(\rt,1)}_{\nu_\rt,\nu_1} (\mathbf{\tilde v}) 
		\leq  \check{C}_\rx (1 + \delta)^2 \pi^{(\rt,1)}_{\nu_\rt,\nu_1} (\mathbf{v}) ,
	\end{align*}
	where we have used that all entries of the diagonal operators $\mathbf{\check D}_1^{-1}\mathbf{D}$ and $(\bI_1 \otimes \check{\bD}^{-1}_1) \bDXnut$ are bounded by $(1+\delta)$.

	We now examine the second summand on the right-hand side of \eqref{eq:contraction-scaling}. We obtain
	\begin{align*}
		D_{2,\nu_\rt,\nu_1} \leq (1 + \delta) \pi^{(1)}_{\nu_1}\left( \left( (\mathbf{\hat D} \bTtil_2^{(1)} \mathbf{\hat D}) \otimes \bI_2 \otimes \cdots \otimes \bI_d\right) \mathbf{\bar D}_1^{-1} \bDXnut \mathbf{\tilde v}_{\nu_\rt} \right),
	\end{align*}
	where $\mathbf{\bar D}_1 = \mathbf{\hat D} \otimes \bI_2 \otimes \cdots \otimes \bI_d$. We set $\hat{\bv} = (\bI_\rt \otimes \mathbf{\bar D}^{-1}_1) \bDX \mathbf{\tilde v}$, $\mathbf{\tilde C}_{\nut,2}^{(1)} = \mathbf{\hat D} \bTtil_{\nut,2}^{(1)} \mathbf{\hat D}$ and $\bC_{2,j} = \mathbf{\hat D} \bT_{2,j} \mathbf{\hat D}$ for each $j \in \N_0$ and for each $\lambda_\rt \in \vee_\rt$ define
	\begin{align*}
		\hat{\Lambda}_{\lambda_\rt,[0]} &= \supp \range \bC_{2,0} \Res_{\Lambda_{\lambda_\rt,[0]}^{(1)}}, \\
		\hat{\Lambda}_{\lambda_\rt,[q]} &= \bigl( \bigcup_{j+l=q}  \supp \range \bC_{2,j} \Res_{\Lambda_{\lambda_\rt,[l]}^{(1)}} \bigr) \setminus \bigl(\bigcup_{i < q} \hat{\Lambda}_{\lambda_\rt,[i]}\bigr), \ q > 0,
	\end{align*}
	where $\Lambda_{\lambda_\rt,[j]}^{(1)} = \Lambda_{\lambda_\rt,[j]}^{(1)}(\bv)$ and $\Lambda_{[j]}^{(\rt,1)} = \Lambda_{[j]}^{(\rt,1)}(\bv)$ are defined as in \eqref{eq:operator-space-restriction-sets} and \eqref{eq:lambda-all-time-indices}, respectively. Furthermore, we define the union over all temporal basis indices
	\begin{align*}
		\hat{\Lambda}_{[k]} = \bigcup_{\nu_\rt \in \vee_\rt} \{\nu_\rt\} \times \hat{\Lambda}_{\nu_\rt,[k]}
	\end{align*}
	for each $k \in \N_0$. Thus for 
	\[
	\hat{\bw} = \sum\limits_{\nut \in \veet} \mathbf{e}_{\nut} \otimes \Bigl( \mathbf{\tilde C}_{\nut,2}^{(1)} \otimes \bI_2 \otimes \cdots \otimes \bI_d\Bigr) \bv_{\nut}
	\]
	 we obtain
\begin{equation*}
	 \norm{\Res_{\hat{\Lambda}_{[q]}} \pi^{(\rt,1)}(\hat{\bw}) }^2
	= \sum\limits_{\nu_\rt \in \vee_\rt} \norm{\Res_{\hat{\Lambda}_{\nu_\rt,[q]}} \pi^{(1)}((\mathbf{\tilde C}_{2,\nu_\rt}^{(1)} \otimes \bI_2 \otimes \cdots \otimes \bI_d )\hat{\bv}_{\nu_\rt}) }^2 .
\end{equation*}
	For each summand on the right-hand side can be proceeded as in the proof of \cite[Theorem 8]{BachmayrNearOptimal} to obtain
	\begin{align*}
		\norm{\Res_{\hat{\Lambda}_{\nu_\rt,[q]}} \pi^{(1)}((\mathbf{\tilde C}_{2,\nu_\rt}^{(1)} \otimes \bI_2 \otimes \cdots \otimes \bI_d) \hat{\bv}_{\nu_\rt}) }^2 \leq  
		3 \tilde{C} \left(\sum\limits_{l=0}^{q-1} \gamma_{q,l} 2^{-2s(q-l-1)} \norm{\Res_{\Lambda_{\nu_\rt,[q]}} \pi^{(1)}(\hat{\bv}_{\nu_\rt}) }^2 \right. \\
		\left. + \sum\limits_{l=q}^J   \beta^{(\rx)}_{J-l} 2^{-2s(j-l)} \norm{\Res_{\Lambda_{\nu_\rt,[q]}} \pi^{(1)}(\hat{\bv}_{\nu_\rt}) }^2 + \norm{\bC_{2}} \norm{\Res_{\bigcup_{j \geq q} \Lambda_{\nu_\rt,[j]}} \pi^{(1)}(\hat{\bv}_{\nu_\rt})}^2 \right) ,
	\end{align*}
	where $\beta^{(\rx)} = \beta(\bC_2)$ and $\gamma_{q,l} = \beta^{(\rx)}_{J-l} + \beta^{(\rx)}_{q-l-1}$.
	Summing over each temporal index and using
	\begin{align*}
		\sum\limits_{\nu_\rt \in \vee_\rt} \norm{\Res_{\Lambda_{\nu_\rt,[q]}} \pi^{(1)}(\hat{\bv}_{\nu_\rt}) }^2 = \norm{\Res_{\Lambda_{[q]}} \pi^{(\rt,1)}(\hat{\bv}) }^2,
	\end{align*}
	we arrive at
	\begin{multline*}
		\norm{\Res_{\hat{\Lambda}_{[q]}} \pi^{(\rt,1)}( \hat{\bw}) }^2 
		\leq  3 \tilde{C} \left(\sum\limits_{l=0}^{q-1} \gamma_{q,l} 2^{-2s(q-l-1)} \norm{\Res_{\Lambda_{[q]}} \pi^{(\rt,1)}(\hat{\bv}) }^2 \right. \\
		\left. + \sum\limits_{l=q}^J   \beta^{(x)}_{J-l} 2^{-2s(j-l)} \norm{\Res_{\Lambda_{[q]}} \pi^{(\rt,1)}(\hat{\bv}) }^2 + \norm{\bC_2} \norm{\Res_{\bigcup_{j \geq q} \Lambda_{[j]}} \pi^{(\rt,1)}(\hat{\bv})}^2 \right) .
	\end{multline*}
	We then proceed exactly as in the proof of \cite[Theorem 8]{BachmayrNearOptimal}. 
	In the end we use that the entries of the diagonal operator $\mathbf{\bar D}_1^{-1} \bDXnut$ are bounded by $1+\delta$ and therefore $\pi^{(1)}_{\nu_1}(\mathbf{\bar D}_1^{-1} \bDXnut \mathbf{\tilde v}_{\nu_\rt}) \leq (1+\delta) \pi^{(1)}_{\nu_1}(\mathbf{\tilde v}_{\nu_\rt})$. In particular, for 
	\[
	\mathbf{\tilde w} = \sum\limits_{\nut \in \veet} \biggl(\mathbf{e}_{\nut} \otimes \Bigl( \mathbf{\tilde C}_{\nut,2}^{(1)} \otimes \bI_2 \otimes \cdots \otimes \bI_d  \Bigr) \mathbf{\bar D}_1^{-1} \bDXnut \mathbf{\tilde v}_{\nut} \biggr)
	\]
	we have
	\begin{align*}
		\norm{\pi^{(\rt,1)}(\mathbf{\tilde w})}_{\cAs} \leq \frac{2^{3s+2}}{2^s-1} (1+\delta)^2 \norm{\alpha(\bC_2)}^s_{\ell_1} (2 \norm{\bC_2}) \norm{\pi^{(\rt,1)}(\mathbf{\tilde v})}_{\cAs} .
	\end{align*}
	Combining these two estimates with Proposition \ref{prop:cAs-properties} results in
	\begin{equation*}
		\norm{\bDY  \bTtil_{\rx,J} \bDX \mathbf{\tilde v}}_{\cAs} \leq 2^s C_\delta (C+\check{C}_\rx) \norm{\piti(\mathbf{\tilde v})}_{\cAs}
		\leq  2^s C_\delta  (C+\check{C}_\rx) \norm{\piti(\bv)}_{\cAs},
	\end{equation*}
	where we used that each entry of the diagonal matrix $\bDX^{-1} \mathbf{D}_{\cX,\sfn_1}$ is bounded by one.
	
	We turn to the temporal operator $\mathbf{\tilde B}_{\rt,J}$ with $J = J(\eta;\bv)$. Here we have restrictions to subsets of indices in all variables simultaneously and thus need to proceed differently. The basic idea is to estimate the $\cAs$-norm of $\piti(\bDY \mathbf{\tilde T}_{\rt,J}[\bv] \bDX \mathbf{\tilde v})$ by the sum of the $\cAs$-norm of the contractions of all dimensions of $\bv$. Let $\mathbf{\tilde w}_q = \bDY \mathbf{\tilde T}_{\rt,q}[\bv] \bDX \mathbf{\tilde v}$. For each $q=0,\dots,J$, we define
	\begin{align*}
		\hat{\Lambda}^{(i)}_q = \supp\left(\piti(\mathbf{\tilde w}_{q})\right) , \quad q=0,\dots,J = J(\eta;\bv).
	\end{align*}
	By \eqref{eq:operator-apply-temporal-supp-J} from Lemma \ref{lem:operator-apply-J-time} we have
	\begin{align*}
		\# \hat{\Lambda}^{(i)}_q \leq 2c 2^q
	\end{align*}
	for each $q$. Let $N \in \N$ be arbitrary with $N \leq \# \hat{\Lambda}^{(i)}_{J}$. Then there exists a $q<N$ with $\# \hat{\Lambda}^{(i)}_q < N \leq \# \hat{\Lambda}^{(i)}_{q+1}$. We observe that
	\begin{align*}
		\lrnorm{\piti(\mathbf{\tilde w}_J) - \Res_{\hat{\Lambda}^{(i)}_q} \piti(\mathbf{\tilde w}_J)}^2 &= \sum\limits_{(\nu_\rt,\nu_i) \notin \hat{\Lambda}^{(i)}_q} \left(\piti_{\nu_\rt,\nu_i}(\mathbf{\tilde w}_J)\right)^2 \\
		&= \sum\limits_{(\nu_\rt,\nu_i) \notin \hat{\Lambda}^{(i)}_q} \left(\piti_{\nu_\rt,\nu_i}(\mathbf{\tilde w}_J - \mathbf{\tilde w}_q)\right)^2 \\
		&\leq \lrnorm{\piti(\mathbf{\tilde w}_J - \mathbf{\tilde w}_q)}^2 \\
		&= \lrnorm{\mathbf{\tilde w}_J - \mathbf{\tilde w}_q}^2 .
	\end{align*}
	As a consequence,
	\begin{equation*}
		\begin{aligned}
		(N+1)^s \inf\limits_{\# \Lambda = N} \norm{\piti(\mathbf{\tilde w}_J) - \Res_{\Lambda} \piti(\mathbf{\tilde w}_J)} &\leq (2c2^{q+1} +1)^s \norm{\piti(\mathbf{\tilde w}_J) - \Res_{\hat{\Lambda}^{(i)}_q} \piti(\mathbf{\tilde w}_J)} \\
		&\leq (2c2^{q+1} +1)^s \norm{\mathbf{\tilde w}_J - \mathbf{\tilde w}_q} \\
		&\leq 2^{3s} c^s 2^{sq} \left(\norm{\mathbf{\tilde w}_J - \bB_\rt \mathbf{\tilde v}} + \norm{\mathbf{\tilde w}_q - \bB_\rt \mathbf{\tilde v}}\right) 
		\end{aligned}
	\end{equation*}
	Arguing as in Lemma \ref{lem:operator-apply-J-time} and using 
	\[ \bignorm{\Res_{\Lambda_{[p]}^{(\rt,i)}(\bv) } \piti(\mathbf{\tilde v}) } \leq (1+\delta)  2^{-s(p-1)} \norm{\piti(\mathbf{\tilde v})}_{\cA^s}, \] 
	where $\Lambda_{[p]}^{(\rt,i)}(\bv)$ is defined in \eqref{eq:lambda-all-time-indices}, we obtain
	\[
		\norm{\mathbf{\tilde w}_q - \bB_\rt \mathbf{\tilde v}}   \leq (1+\delta)
		 2^{-s(q-1)} (\norm{\beta^{(\rt)}}_{\ell_1} + \norm{\bC_\rt}) \sum\limits_{i=1}^d \norm{\piti(\mathbf{\tilde v})}_{\cAs}, \quad q = 0,\ldots, J.
	\]
	Since $N$ was chosen arbitrarily, and using again that the entries of the diagonal matrix $\bDX^{-1} \mathbf{D}_{\cX,\sfn_1}$ are bounded by one, altogether it follows that
	\begin{align*}
		\norm{\piti(\mathbf{\tilde w}_J)}_{\cAs} \leq 2 \sqrt{C_\delta} c^s 2^{4s} (\norm{\beta^{(\rt)}}_{\ell_1} + \norm{\bC_\rt}) \sum\limits_{i=1}^d \norm{\piti(\bv)}_{\cAs} .
	\end{align*}
	Once again using Proposition \ref{prop:cAs-properties}\eqref{prop:cAs-triangle} yields \eqref{eq:apply-eta-As-norm}.
\end{proof}
In comparison to the similar bounds obtained for elliptic problems in \cite{bachmayr2014adaptive}, we have an additional factor $J(\eta;\bv)$ in the rank estimate \eqref{eq:apply-eta-rank}. As shown in Section \ref{sec:Solver}, this factor does not influence the asymptotic computational complexity of the method.
\begin{remark}\label{rem:apply-sum-contr}
	Our analysis of the computational complexity of an adaptive solver is based on estimates for sums of support sizes and $\cAs$-norms over the spatial dimensions. We thus reformulate the estimate \eqref{eq:apply-eta-As-norm} accordingly, where summation over $i=1,\ldots,d$ yields
	\begin{align}
		\sum\limits_{i=1}^d \norm{\piti(\bw_{\eta})}_{\cAs} \leq C(\delta, s,\bB_1) d \sum\limits_{i=1}^d \norm{\piti(\bv)}_{\cAs} .
	\end{align}
\end{remark}

We next obtain an estimate for the number of operations required for computing the approximation $\bw_{\eta}$. 

\begin{lemma}\label{lem:apply-num-ops}
	Under the assumptions of Theorem \ref{thm:error-rescaled-trunc}, the number $\text{\rm flops}(\bw_{\eta})$ of floating point operations to compute $\bw_{\eta}$ is bounded by
	\begin{align}\label{eq:apply-operations}
		\begin{split}
			\operatorname{flops}(\bw_{\eta}) \lesssim \ (8 + c^4 (J(\eta;\bv))^4) \hat M^3(\bv,\eta)  \rank^3_\infty(\bv)  \sum\limits_{i=1}^d \# \supp\left(\piti(\bv) \right) \\
			+ d \eta^{-\frac{1}{s}}  c^2 \hat{\alpha}(\bB_1)  \rank_\infty(\bv) J(\eta;\bv) \hat M(\bv,\eta) \tilde{C}^{\frac{1}{s}} \Biggl( \sum\limits_{i=1}^d \lrnorm{\pi^{(\rt,i)}(\bv)}_{\cAs}\Biggr)^{\frac{1}{s}}  ,
		\end{split}
	\end{align}
	where $\hat M(\bv, \eta) = \norm{\hmcX(\bv;\eta)}_{\ell_\infty} \norm{\hmcY(\bv;\eta)}_{\ell_\infty}$, and where the hidden constant is independent of $\eta, \bB_1, \bv$ and $d$.
\end{lemma}
\begin{proof}
	In the following we will analyze the complexity of the spatial and the temporal operator separately. For the spatial operator $\bB_\rx$, we can proceed similarly to the elliptic case \cite{BachmayrNearOptimal,bachmayr2014adaptive} using the fact that by construction $\sum_{\nu_\rt \in \vee_\rt} \Lambda^{(i)}_{\nu_\rt,[p]} \leq 2^p$ for each $p \in \N_0$. Let $\bv \in \ell_2(\vee) = \sum_{\nu_\rt \in \vee_\rt} \mathbf{e}_{\nu_\rt} \otimes \bv_{\nu_\rt}$ with ranks $\rank_{\alpha}( \bv_{\nu_\rt}) = r_{\nu_\rt,\alpha}$ for $\alpha \in \mathbb{T}_d \setminus \alpha^*$. Then the number of floating point operations to calculate the result of the spatial operator is bounded by a fixed multiple of
	\begin{multline*}
		 \sum\limits_{\nu_\rt \in \vee_\rt} \sum\limits_{\substack{\alpha \in \mathbb{T}_d \\ \# \alpha > 1}} 8 (\hmcX(\bv;\eta))_{\nu_\rt}^3 (\hmcY(\bv;\eta))_{\nu_\rt}^3  r_{\nu_\rt,\alpha} \prod\limits_{q=1}^2  r_{\nu_\rt, c_q(\alpha)} \\
		+ d \eta^{-\frac{1}{s}} \hat{\alpha}(\bB_1) \rank_\infty(\bv) \hat M(\bv,\eta) \tilde{C}^{\frac{1}{s}}  \Biggl( \sum\limits_{i=1}^d  \lrnorm{\pi^{(\rt,i)}(\bv)}_{\cAs}\Biggr)^{\frac{1}{s}} ,
	\end{multline*}
	where we used that the rank of the Laplace operator $R=2$.
	By using the fact that there are $d-1$ transfer tensors, the first summand can be further bounded by
	\begin{align*}
		\sum\limits_{\nu_\rt \in \vee_\rt} \sum\limits_{\substack{\alpha \in \mathbb{T}_d \\ \# \alpha > 1}} 8 (\hmcX(\bv;\eta))_{\nu_\rt}^3 (\hmcY(\bv;\eta))_{\nu_\rt}^3 & r_{\nu_\rt,\alpha} \prod\limits_{q=1}^2 r_{\nu_\rt, c_q(\alpha)}  \\
		\leq \ &8 \hat M^3(\bv,\eta)  \rank^3_\infty(\bv) (d-1) \# \supp\left(\pit(\bv) \right) \\
		\leq \  &8 \hat M^3(\bv,\eta)  \rank^3_\infty(\bv)  \sum\limits_{i=1}^d \# \supp\left(\piti(\bv) \right) .
	\end{align*}

	Next we examine the temporal operator $\bB_\rt$. First we start with the computation of $\bv_0,\dots,\bv_J$. By definition of $\bv_i$, we need $d 2^i \rank_\infty(\bv)$ operations for the mode frames of $\bv_i$ and at most $\# \supp(\pit(\bv)) d \rank_\infty^3(\bv)$ operations for the transfer tensors.
	 
	We divide the analysis of the costs of the application step into the  mode frames and transfer tensors. Since $\bB_\rt$ is diagonal with respect to the spatial variables, the corresponding low-rank approximations are rescaled, but remain otherwise unchanged. The number of operations required to compute the mode frames of $\bw_{\eta}$ can be estimated by the number of entries in the mode frames of $\bw_{\eta}$, because each entry is only multiplied by the corresponding diagonal entries of the scaling matrices which has constant costs. Therefore the number of operations for the mode frames is bounded by a fixed multiple of the sum of the one-dimensional support sizes times the maximum rank, which can be bounded further by
	\begin{multline*}
		 2 d c^2 2^{J(\eta;\bv)} \hat M(\bv,\eta) J(\eta;\bv) \rank_\infty(\bv) \\
		\lesssim d \eta^{-\frac{1}{s}} c^2 \hat{\alpha}(\bB_1) \rank_\infty(\bv) J(\eta;\bv) \hat M(\bv,\eta) \tilde{C}^{\frac{1}{s}} \Biggl( \sum\limits_{i=1}^d \lrnorm{\pi^{(\rt,i}(\bv)}_{\cAs}\Biggr)^{\frac{1}{s}} . 
	\end{multline*}

	It remains to consider the costs of the transfer tensors. Due to the super-compressibility of the temporal operator, we have $\# \supp(\pit(\bw_{\eta})) \leq cJ(\eta;\bv) \# \supp(\pit(\bv))$. The application of the operator is a multiplication of the low-rank representations by a certain value, which is done by the application of the scaling matrices as well. Additionally due to the super compressibility we have at most $c J(\eta;\bv)$ non zero entries in each row. Therefore the maximum rank grows at most up to the factor $c J(\eta;\bv) \hat M(\bv,\eta)$. Each low-rank approximation has $d-1$ transfer tensors, which results in a number of operations which can be bounded by a fixed multiple of
	\begin{multline*}
		 cJ(\eta;\bv) \# \supp(\pit(\bv)) (d-1)  c^3 (J(\eta;\bv))^3 \hat M^3(\bv,\eta) \rank^3_\infty(\bv) \\
		\leq  c^4 (J(\eta;\bv))^4 \hat M^3(\bv,\eta) \rank^3_\infty(\bv) \sum\limits_{i=1}^d  \# \supp(\piti(\bv)) .
	\end{multline*}
	Combining these estimates yields \eqref{eq:apply-operations}.
\end{proof}

We next consider the bounds on approximation ranks in further detail, especially concerning bounds on the quantities $\norm{\hmcX(\bv;\eta)}_{\ell_\infty}$ and $\norm{\hmcY(\bv;\eta)}_{\ell_\infty}$ in \eqref{eq:apply-eta-rank}. The basic idea is to express these quantities in terms of the tolerance $\eta$ and the maximum one-dimensional level. We will subsequently further estimate the latter under additional assumptions on the regularity. To bound the maximum level of the result of the adaptive operator application, we make use of the \emph{level decay} property as introduced in \cite{bachmayr2014adaptive}, which is defined as follows.
\begin{definition}\label{def:level-dec-sobolev-stab}
	Let $\hat{\vee}$ be a countable set and $\bM \colon \ell_2(\hat{\vee})\to\ell_2(\hat{\vee})$ be $s^*$-compressible or super-compressible with approximations $\bM_j$. We say that these approximations have \emph{level decay} if there exists a $\gamma > 0$ such that $\bM_{j,\nu\mu} = 0$ for $\abs{\abs{\nu}-\abs{\mu}} > \gamma j$.
\end{definition}

In what follows we denote by $L_\rx(\bv)$ the maximum active one-dimensional spatial wavelet level of $\bv$ and by $L_\rt(\bv)$ the maximum active temporal wavelet level.
In view of \eqref{eq:MX-bound} and \eqref{eq:MY-bound}, we need to find an upper bound for $\ln(\norm{\sfK}_{\ell_\infty})$ with $\sfK = \sfK(J;\bv)$ defined in \eqref{eq:K-def-LamdaK}. By \cite[Section 6.2]{bachmayr2014adaptive}, $\Lambda_{\hat{K}_{\nu_\rt}}$ as defined in \eqref{eq:LambdaKdef} contains $\supp (\bTtil_{J(\eta;\bv)}[\bv] \mathbf{\tilde v})_{\nu_\rt}$ if
\begin{align*}
	L_\rx\bigl((\bTtil_{J(\eta;\bv)}[\bv] \mathbf{\tilde v})_{\nu_\rt}\bigr) \leq \tfrac{1}{2} \log_2(\hat{K}_{\nu_\rt}) + \log_2(c^{-1} \hat{S}^{-1}_{\min}) 
\end{align*}
with $\hat{S}_{\min} = \bigl( \min\limits_{\nu \in \vee_1} \norm{\psi_\nu}_{H_0^1(0,1)} \bigr)^{-1}$. Due to the minimality of $\sfK$, one has
\begin{align*}
	\tfrac{1}{2} \log_2 (\sfK_{\nu_\rt}) &\leq L_\rx((\bTtil_{J(\eta;\bv)}[\bv] \mathbf{\tilde v})_{\nu_\rt}) + \log_2(c^{-1} \hat{S}^{-1}_{\min}) \\
	&\leq L_\rx(\bTtil_{J(\eta;\bv)}[\bv] \mathbf{\tilde v}) + \log_2(c^{-1} \hat{S}^{-1}_{\min}) \\
	&\leq L_\rx(\bv) + C_1(\bB_1,s) J(\eta;\bv) + \log_2(c^{-1}\hat{S}^{-1}_{\min}) ,
\end{align*}
where we used the level decay property in the last line.
By \cite[Section 6.2]{bachmayr2014adaptive} we have the estimate 
\begin{align*}
	\abs{\ln(\min\{\delta/2,c(\bv)\eta\})}\leq C_2(\bB_1,\delta) + \abs{\ln \eta} + \max\{0, \ln \norm{\bv}\} .
\end{align*}
Combining all the previous estimates with 
\begin{align}\label{eq:J-eta-bound}
	J(\eta;\bv) \leq \frac{1}{s} \Big( \abs{\ln \eta} + \ln \Big(C_3(\bB_1,\delta) \sum\limits_{i=1}^d \norm{\piti(\bv)}_{\cAs}  \Big)\Big)
\end{align}
from \eqref{eq:J-eta-As-estimate} yields
\begin{align}\label{eq:mcy-inf-bound}
	\norm{\hmcY(\eta;\bv)}_{\ell_\infty} \leq C_4(\delta,s,\bB_1) \Big[ 1 + L_\rx(\bv) + \abs{\ln \eta} + \ln\Big(\sum\limits_{i=1}^d \norm{\piti(\bv)}_{\cAs} \Big)  \Big] .
\end{align}
Now we turn to $\mcX$, which depends not only on the spatial supports of $\bv$ at the respective temporal basis indices, but also on the norm of each temporal basis index. We use that for $a_{\nu_\rt}$ defined in \eqref{eq:scaling-time-a}, we have $\ln(a_{\nu_\rt}) \eqsim \abs{\nu_\rt}$. Let $\nu_\rt$ be an active temporal index. Then
\begin{align*}
	m_{\cX,\nu_\rt}(\eta;\bv) \leq
		C_5(\delta,s,\bB_1)(1+\abs{\nu_\rt}) \Big[ 1 + L_\rx(\bv) + \abs{\ln \eta} + \ln\Big(\sum\limits_{i=1}^d \norm{\piti(\bv)}_{\cAs} \Big)  \Big]
\end{align*}
if $K_{\nu_\rt} > a_{\nu_\rt}^2$, where we use the same estimates as for $\mcY$, as well as $h_{a_{\nu_\rt}}^{-1} \eqsim (1+\abs{\nu_\rt})$.
In the case $K_{\nu_\rt} \leq a_{\nu_\rt}^2$, we have
\begin{align*}
	m_{\cX,\nu_\rt}(\eta;\bv) \leq C_6(\delta,s,\bB_1) (1 + \abs{\nu_\rt}) \Big[ 1+ \abs{\nu_\rt} + \abs{\ln \eta} + \ln\Big(\sum\limits_{i=1}^d \norm{\piti(\bv)}_{\cAs} \Big)  \Big] .
\end{align*}
In the next step, we take into account the maximum temporal level $L_\rt(\bTtil_{J(\eta;\bv)} \bv)$, because $m_{\cX,\nu_\rt}$ can be chosen zero for $\abs{\nu_\rt} > L_\rt(\bTtil_{J(\eta;\bv)} \bv)$. As in the other case we use the level decay property, which results in
\begin{align*}
	\norm{\mcX(\eta;\bv)}_{\ell_\infty} \leq C_7(\delta,s,\bB_1) \Big[ 1 + L(\bv) + \abs{\ln(\eta)} + \ln\Big(\sum\limits_{i=1}^d \norm{\piti(\bv)}_{\cAs} \Big) \Big]^2,
\end{align*}
where $L(\bv) = \max\{L_\rx(\bv), L_\rt(\bv)\}$ is the overall maximum level. Another difference to the case of $\mcY$ is that $n_{a_{\nut}}^+(\delta)$ depends on the temporal wavelet as well. By the definition of $n_{a_{\nut}}^+$ in \eqref{eq:nplus-eponential-sum} and $h_{a_{\nut}}$ in \eqref{eq:h-epxonential-sum}, we obtain that $n_{a_{\nut}}^+$ only depends quadratically on the maximum temporal level. As a consequence,
\begin{align}\label{eq:mcx-inf-bound}
	\norm{\hmcX(\eta;\bv)}_{\ell_\infty}  \leq C_8(\delta,s,\bB_1) \Big[ 1 + L(\bv) + \abs{\ln(\eta)} + \ln\Big(\sum\limits_{i=1}^d \norm{\piti(\bv)}_{\cAs} \Big) \Big]^2 .
\end{align}

\begin{remark}
	The approximations of the scaling matrix $\bDX$ are be applied to $\bv$ directly. Hence in this case, we do not actually need the relation to the maximum temporal level $L_\rt(\bTtil_{J(\eta;\bv)}[\bv] \bv)$. However, as considered in further detail in Section \ref{sec:Solver}, we do not only apply the operator $\bB_1$, but also its transpose $\bB_1^\intercal$. In this case the order of the scaling matrices is reversed. The above estimates thus cover the two extreme cases that arise.
\end{remark}

\subsection{Initial value operator}\label{sec:apply-initial-value}
In this section we consider the initial value operator $\bB_2 = \bT_0 \bDX$ and its transpose $\bB_2^\intercal = \bDX \bT_0^\intercal$. As for the temporal operator in Section \ref{sec:Apply-temporal}, the operator $\bT_0$ is in the form
\[
	\bT_0 = \bT_{0,\rt} \otimes \bI_{\rx},
\]
with the identity in the spatial variable, but where $\bT_{0,\rt} \colon \ell_2(\vee_\rt) \to \R$ and correspondingly $\bT_0 \colon \ell_2(\vee) \to \ell_2(\vee_\rx)$.

Based on the definitions of $\bT_0$ and $\mathbf{\bar D}_{\cX}$, as well as $\norm{\Psi_{\nu_{\rx}}}_{H_0^1(\Omega)} \eqsim 2^{\max_i \abs{\nu_{\rx,i}}}$ for $\nu_\rx \in \vee_\rx$, the entries of $\mathbf{\bar B}_2$ satisfy
\begin{equation*}
	\abs{(\mathbf{\bar B}_2)_{\nu_\rx,(\nut,\nu_\rx)} } = \lrabs{\frac{\theta_{\nut}(0) \norm{\Psi_{\nu_\rx}}_{H_0^1} }{\norm{\Psi_{\nu_\rx}}_{H_0^1}^2 + \norm{\theta_{\nut}}_{H^1} }}\eqsim \frac{\norm{\theta_{\nut}}_{H^1}^{\frac{1}{2}}\norm{\Psi_{\nu_\rx}}_{H_0^1}} {\norm{\Psi_{\nu_\rx}}_{H_0^1}^2 + \norm{\theta_{\nut}}_{H^1} }
	\lesssim 2^{-\lrabs{\frac{1}{2}\abs{\nut} - \max_i \abs{\nu_{\rx,i}}}}.
\end{equation*}
In contrast to the operators considered above, the decay of the entries does not depend on a single dimension, but on all dimensions. Therefore this case poses some additional difficulties concerning low-rank approximability. Note the initial value operators $\mathbf{\bar B}_2$ and $\bB_2$ have only one nonvanishing entry in each column.

We now define a sequence of approximations for $\mathbf{\bar B}_{2}$. Similarly to the construction in \cite{SpaceTimeAdaptive}, we define $\mathbf{\bar T}_{0,j}$ by
\[
	(\mathbf{\bar T}_{0,j})_{\nu_\rx,(\nut,\nu_\rx)} = 
	\begin{cases}
		(\bT_0)_{\nu_\rx,(\nut,\nu_\rx)}, \quad &\lrabs{\abs{\nut} - 2\max_i \abs{\nu_{\rx,i}}} \leq j, \\
		0, \quad  \quad &\lrabs{\abs{\nut} - 2\max_i \abs{\nu_{\rx,i}}} > j ,
	\end{cases}
\]
that is, by dropping all entries where $\abs{\abs{\nut} - 2\max_i \abs{\nu_{\rx,i}}} > j$. For each level, only a constant number of wavelet functions does not vanish at time $t=0$. Therefore, by definition, we have at most $2j \bar{c}$ non-zero entries in each row of $\mathbf{\bar T}_{0,j}$. By the Schur lemma (see for example \cite[(3.15)]{CDD1}), we additionally have $\norm{(\bT_0 - \mathbf{\bar T}_{0,j})\mathbf{\bar D}_{\cX}} \lesssim 2^{-j/2}$. Hence the operator $\bT_0$ is super-compressible, where the compressibility is again to be understood in combination with the scaling matrix. 

In the next step, we define an adaptive approximation of the action of the initial value operator. Based on the definition of the approximation sequence, which uses the maximum level, we define
\[
	\hat{\Lambda}^{(\rt)}_l =  \veet \times \bigtimes_{i=1}^d \hat{\Lambda}_l^{(1)}  \quad \text{with} \quad  \hat{\Lambda}_l^{(1)} = \{\nu \in \vee_1 : \abs{\nu} \leq l\} .
\]
Using this index set, we can divide a given input $\bv \in \ell_2(\vee)$ in parts with the same maximum level
\[
	\bv = \Res_{\hat{\Lambda}^{(\rt)}_0} \bv + \sum\limits_{l=1}^{L_\rx(\bv)} \Res_{\hat{\Lambda}^{(\rt)}_l} \bv - \Res_{\hat{\Lambda}^{(\rt)}_{l-1}} \bv .
\]

For a given $s > 0$, let $(\bT_{0,j})_{j \in \N}$ be an approximation sequence of the super-compressible operator $\bT_0$, which means $\norm{(\bT_0 - \bT_{0,j}) \bDX} \leq \beta^{(0)}_J 2^{-sj}$  and the operators $\bT_{0,j+1} -\bT_{0,j}$ have only $c$ non-zero entries in each row and column. Additionally we assume $\bT_{0,0} = 0$, so that we only have $cj$ non-zero entries in each row and column of $\bT_{0,j}$. In view of Remark \ref{rem:scaling-DX-delta}, such approximations can be obtained from $(\mathbf{\bar T}_{0,j})$.

For a given $\bv \in \ell_2(\vee)$ and $J \in \N_0$, we define the approximation
\begin{equation}\label{eq:initial-value-application-approx}
	\begin{aligned}
		\bBtil_{2,J}(\bv) 
		&= \sum\limits_{l=0}^{L_\rx(\bv)} \sum\limits_{j=0}^{J-1} (\bT_{0,J-j} - \bT_{0,J-j-1}) \bDX  \Bigl(\Res_{\hat{\Lambda}^{(\rt)}_l} \bv_j  - \Res_{\hat{\Lambda}^{(\rt)}_{l-1}} \bv_j  \Bigr) ,
	\end{aligned}
\end{equation}
where $\bv_j$ is defined by \eqref{eq:timeindexsets}. We thus apply the differences of the operator approximations to restrictions to index sets with fixed maximum spatial level.
For later reference, we also note that
\[
	\bBtil_{2,J}(\bv) = \bTtil_{0,J}[\bv] \bDX \bv, 
\]
where the operator $\bTtil_{0,J}[\bv]$ is defined by
\[
	\bTtil_{0,J}[\bv] =  \sum\limits_{l=0}^{L_\rx(\bv)} \sum\limits_{j=0}^{J-1} (\bT_{0,J-j} - \bT_{0,J-j-1}) \bDX  \Bigl(\Res_{\hat{\Lambda}^{(\rt)}_l}  - \Res_{\hat{\Lambda}^{(\rt)}_{l-1}} \Bigr) \Res_{\bar{\Lambda}_j^{(\rt)}(\bv)} .
\]
\begin{lemma}\label{lem:initial-operator-J}
	Let $\bB_2$ be defined by \eqref{eq:Bpartsdef} and assume that $\bT_0$ is super-compressible. Given any $\bv, \bw \in \ell_2(\vee)$ and $J \in \N_0$, let $\bBtil_{2,J}(\bv)$ be defined by \eqref{eq:initial-value-application-approx}. Then whenever $\piti(\bv) \in \cAs$ for some $0 < s < s^*$ and $i=1,\dots,d$, one has the a posteriori error bound
	\begin{equation}\label{eq:initial-value-a-posteriori}
		\norm{\bB_2 \bw - \bTtil_{0,J}[\bv] \bDX \bw} \leq e_{0,J}[\bv](\bw),
	\end{equation}
	where
	\begin{equation*}
		e_{0,J}[\bv](\bw)  = \sum\limits_{j=0}^J 2^{-s(J-j)} \beta^{(0)}_{J-j} \sum\limits_{i=1}^d \Bignorm{\Res_{\Lambda^{(\rt,i)}_{[j]}(\bv)} \piti(\bw)} + \norm{\bB_2} \Bignorm{\Res_{\Lambda^{(\rt,i)}_{[J+1]}(\bv)} \piti(\bw)} 
	\end{equation*}
	with $\Lambda^{(\rt,i)}_{[j]}(\bv)$ from \eqref{eq:lambda-all-time-indices}, 
	as well as the a priori error bound
	\begin{equation}\label{eq:initial-value-a-priori}
		\norm{\bB_2 \bv - \bBtil_{2,J}(\bv)} \leq 2^s 2^{-sJ} (\norm{\beta^{(0)}}_{\ell_1} + \norm{\bB_2}) \sum\limits_{i=1}^d \norm{\piti(\bv)}_{\cAs} .
	\end{equation}
	Moreover, one has the support estimate
	\begin{equation}\label{eq:initial-value-J-support}
		\# \supp\left(\pi^{(i)} (\bBtil_{2,J}(\bv))\right) \leq 2c 2^J (L_\rx (\bv)+1), \quad i=1,\dots,d .
	\end{equation}
\end{lemma}
\begin{proof}
	For the error bounds, we can proceed similar to the Lemma \ref{lem:operator-apply-J-time}. Once again we set $\Lambda^{(\rt)}_{-1}(\bv)= \emptyset$. Using $\Res_{\hat{\Lambda}^{(\rt)}_l} \Res_{\bar{\Lambda}^{(\rt)}_j(\bv)} \bw = 0$ for $l > L_\rx(\bv)$, we have
	\begin{multline*}
		\norm{\bB_2 \bw - \bTtil_{0,J}[\bv] \bDX \bw}  \\ \leq \sum\limits_{j=0}^{J} \norm{(\bT_0 - \bT_{0,J-j}) \bDX} \Bignorm{\Res_{\bar{\Lambda}^{(\rt)}_{j}(\bv)}\bw - \Res_{\bar{\Lambda}^{(\rt)}_{j-1}(\bv)} \bw} + \norm{\bB_2} \Bignorm{\bw - {\Res_{\bar{\Lambda}^{(\rt)}_{J}(\bv)}}\bw } .
	\end{multline*}
	Using the super-compressibility of $\bT_0$ and \eqref{eq:time-operator-vj-diff} yields \eqref{eq:initial-value-a-posteriori}. The a priori bound \eqref{eq:initial-value-a-priori} can be shown as for the spatial operator using the definition of the index sets.
	
	As mentioned above, for each summand in \eqref{eq:initial-value-application-approx} the maximum spatial level is fixed and based on differences of operator approximations we only have to consider a constant number of temporal basis functions.
	Nevertheless, for each maximum spatial level, based on the definition of the approximations $\bT_{0,j}$, we obtain a different set of temporal basis function.
	Hence
	\[
		\#  \supp\bigg(\pi^{(i)} \bigg(\sum\limits_{l=0}^{L_\rx(\bv)} (\bT_{0,J-j} - \bT_{0,J-j-1})\bDX ( \Res_{\hat{\Lambda}_l} \bv_j - \Res_{\hat{\Lambda}_{l-1}} \bv_j)\bigg)\bigg) \leq  c 2^j (L_\rx(\bv)+1),
	\]
	which yields \eqref{eq:initial-value-J-support}.
\end{proof}
For notational simplicity we again define $e_{0,J}(\bv) = e_{0,J}[\bv](\bv)$.
Note that in contrast to the previous operators, the support size bound depends additionally on the maximum active spatial level $L_\rx(\bv)$. This results from the coupling between the spatial and temporal basis functions in the decay of the entries of $\bB_2$.

\begin{remark}
	Instead of building the approximations $\bv_j$ based on the best approximation of the contractions $\piti(\bv)$, one could build the best approximation of the subset associated to temporal basis functions that do not vanish at $t=0$. Although this is not relevant for the above estimates and the subsequent complexity bounds, this may lead to improved efficiency in practical implementations.
\end{remark}

As in Section \ref{sec:adaptive-operator-comb}, we additionally need to replace the scaling matrix $\bDX$ by a low-rank approximation. We obtain the following simplified variant of Lemma \ref{lem:scaling-to-truncated}. 
The proof can be done by a simplified adaptation of \cite[Lemma 32]{bachmayr2014adaptive} using our knowledge on the scaling matrix $\bDX$.

\begin{lemma}\label{lem:scaling-to-truncated-only-DX}
	Let $\mathbf{T}_0$ be defined as in \eqref{eq:opertor-T1T2-def}, let $\mathbf{\tilde T}\in \R^{\vee_\rx \times \vee}$, $\bv \in \ell_2(\vee)$, $\sfK \in \N_0^{\vee_\rt}$ be such that $\supp(\bv) \in \Lambda_{\sfK}$ and let $\bG =  (\bTtil - \bT_0) \bDX$.
	Then whenever $\sfn \geq \MX(\eta,\sfK)$, we have
	\begin{align*}
		\lrnorm{( \bT_0 \bDX - \bTtil \mathbf{D}_{\cX,\sfn})\bv}  \leq &\norm{{\bG}\bv} + \norm{{\bG} (\bI - \bDX^{-1}\mathbf{D}_{\cX,\sfn}) \bv}  + \frac{\eta}{1-\delta} \norm{\bB_2} \norm{\bv} .
	\end{align*}
\end{lemma} 

For a given $\bv \in \ell_2(\vee)$ and tolerance $\eta > 0$, we choose the parameters
\[
	J_0(\eta;\bv) = \min\{ J \in \N_0 : e_{0,J}(\bv) \leq \tfrac{\eta}{4}  \}, \quad c_0(\bv)\eta = \tfrac{\eta (1-\delta)}{2 \norm{\bB_2} \norm{\bv}}
\]
and 
\[
	\sfm_{\cX,0}(\eta;\bv) = \MX(c_0(\bv)\eta;K_0(\bv)) \quad \text{with} \quad \sfK_0(\bv) = \min \{ \sfK : \supp(\bv) \subseteq \Lambda_{\sfK} \} .
\]

\begin{theorem}\label{thm:error-bound-initial-value-truncated}
	Let $\bv \in \ell_2(\vee)$ and $0 < \eta < 2 \norm{\bB_2} \norm{\bv}$. We fix
	$J = J_0(\eta;\bv)$  and $\sfn = \sfm_{\cX,0}(\eta;\bv)$.
	Then with $\bw_{\eta} = \mathbf{\tilde T}_{0,J}[\bv] \bD_{\cX,\sfn} \bv$, we have
	\begin{equation*}
		\norm{\bB_2 \bv - \bw_{\eta}} \leq \eta .
	\end{equation*}
\end{theorem}
The theorem is proved analogously to Theorem \ref{thm:error-rescaled-trunc} using Lemma \ref{lem:scaling-to-truncated-only-DX}. Based on this error bound, we derive estimates for the support sizes, hierarchical ranks and $\cAs$-norms of approximations.

\begin{lemma}\label{lem:initial-operator-trunc-estimates}
	Under the assumptions of Theorem \ref{thm:error-bound-initial-value-truncated}, we have
	\begin{equation}\label{eq:supp-initial-value-operator}
		\#\supp(\pi^{(i)}(\bw_{\eta})) \leq 2c 4^{1+\frac{1}{s}} \tilde{C}_0^{\frac{1}{s}} (L_\rx(\bv)+1) \eta^{-\frac{1}{s}} \bigg(\sum\limits_{i=1}^d \norm{\piti(\bv)}_{\cAs}\bigg)^{\frac{1}{s}} ,
	\end{equation}
	with $\tilde{C}_0 = (\norm{\beta^{(0)}}_{\ell_1} + \norm{\bB_2})$, as well as
	\begin{equation*}
		\norm{\pi^{(i)}(\bw_\eta)}_{\cAs} \leq 2 c^s 2^{4s} \tilde{C}_0  (L_\rx(\bv)+1)^s \sum\limits_{i=1}^d \norm{\piti(\bv)}_{\cAs} .
	\end{equation*}
	Moreover, we have the rank estimate
	\begin{equation}\label{eq:ranks-initial-value-operator}
		\abs{\rank(\bw_{\eta})}_{\infty} \leq \norm{\hat{\sfm}_{\cX,0}(\eta;\bv)}_{\ell_\infty} c J_0(\eta;\bv) (L_\rx(\bv)+1) \rank_\infty(\bv),
	\end{equation}
	where
	\[
		\hat{\sfm}_{\cX,0,\nut}(\eta;\bv) = 1+ n_{a_{\nut}}^+(\delta) + \sfm_{\cX,0,\nut}(\eta;\bv), \quad \nut \in \veet .
	\]
\end{lemma}
\begin{proof}
	We first show the rank estimate. By the representation \eqref{eq:initial-value-application-approx} with fixed maximum spatial wavelet level in each summand, we have to consider a constant number of corresponding temporal basis functions, because only a constant number of wavelet functions do not vanish at time $t=0$ on each level. We thus have at most $c J_0(\eta;\bv)(L_\rx(\bv)+1)$ summands. Combining this with the truncated scaling matrix yields \eqref{eq:ranks-initial-value-operator}.
	
	The support size estimate \eqref{eq:supp-initial-value-operator} is a direct consequence of the definition of $J_0$ as well as  \eqref{eq:initial-value-J-support} from Lemma \ref{lem:initial-operator-J} and can be proven similar to Theorem \ref{thm:error-rescaled-trunc}.
	
	The $\cAs$-norms of contractions can be estimated in the same way as for the temporal part in Theorem \ref{thm:error-rescaled-trunc}. The additional factor $(L_\rx(\bv)+1)^s$ depending on the maximum spatial level is based on the support estimate \eqref{eq:supp-initial-value-operator}.
\end{proof}

\begin{lemma}\label{lem:apply-initial-value-num-ops}
	Under the assumptions of Theorem \ref{thm:error-bound-initial-value-truncated}, the number $\text{\rm flops}(\bw_{\eta})$ of floating point operations required to compute $\bw_{\eta}$ is bounded by
		\begin{multline*}
			\operatorname{flops}(\bw_{\eta}) \lesssim \  d c^3 (J_0(\eta;\bv))^3 (L_\rx(\bv)+1)^3 \norm{\sfm_{\cX,0}(\eta;\bv)}_{\ell_\infty}^3 \rank^3_\infty(\bv) \\
			+ 8c4^{\frac{1}{s}}\tilde{C}_0^{\frac{1}{s}} d \eta^{-\frac{1}{s}}   \rank_\infty(\bv) J_0(\eta;\bv) \norm{\sfm_{\cX,0}(\eta;\bv)}_{\ell_\infty} (L_\rx(\bv)+1)^2 \Biggl( \sum\limits_{i=1}^d \lrnorm{\piti(\bv)}_{\cAs}\Biggr)^{\frac{1}{s}}  ,
		\end{multline*}
	 where the hidden constant is independent of $\eta, \bv, \bB_2$ and $d$.
\end{lemma}
The above lemma can be proved analogously to Lemma \ref{lem:apply-num-ops} for the temporal operator, and we thus omit the proof.
As mentioned above, we also need to apply the transposed operator $\bB^\intercal$, and thus also need an approximate application of the transposed operator $\bB_2^\intercal = \bDX \bT_0^\intercal$. The procedure differs from the one for $\bB_2$ because we start from the space $\ell_2(\vee_\rx)$ with only spatial variables and map to $\ell_2(\vee)$. The operator $\bT_0^\intercal$ duplicates the input to different time indices where the input is scaled depending on the value of the time basis function at time $t=0$. Nevertheless, we can adapt the basic ideas used for the approximation of $\bB_2$.

Before we define the approximation, we need to adapt some notation for the case without temporal variable. We set
\[
\hat{\Lambda}_l =  \bigtimes_{i=1}^d \hat{\Lambda}_l^{(1)}  \quad \text{with} \quad  \hat{\Lambda}_l^{(1)} = \{\nu \in \vee_1 : \abs{\nu} \leq l\} .
\]
Let now $\bar{\Lambda}_j^{(i)}(\bv)$ be the best $2^j$-term approximation of the contractions $\pi^{(i)}(\bv)$ and
\begin{align}\label{eq:operator-space-only-restriction-sets}
	\Lambda^{(i)}_{[p]}(\bv) = \begin{cases}
		\bar{\Lambda}^{(i)}_{p}(\bv) \setminus \bar{\Lambda}^{(i)}_{p-1}(\bv), \quad &p = 0,\dots,J, \\
		\vee_{1} \setminus \bar{\Lambda}^{(i)}_{J}(\bv), \quad &p = J+1, \\
		\emptyset, \quad &p > J+1.
	\end{cases}
\end{align}
We set
\begin{equation}
		\bar{\Lambda}_j(\bv) = \bar{\Lambda}_j^{(1)}(\bv) \times \cdots \times \bar{\Lambda}_j^{(d)}(\bv)  \quad\text{and}\quad  
		\bv_j = \Res_{\bar{\Lambda}_j(\bv)} \bv .
\end{equation}
Analogously to \eqref{eq:time-operator-vj-diff}, we have
\begin{equation}
	\norm{\bv_{j+1} - \bv_j} \leq \sum\limits_{i=1}^d \norm{\Res_{\Lambda_{[j+1]}^{(i)}(\bv)} \pi^{(i)}(\bv)} .
\end{equation}

To define the approximation, we reuse the approximation sequence $(\bT_{0,j})_{j \in \N_0}$. One can easily see that $\bT_0^\intercal$ is super-compressible with the approximation sequence $(\bT^\intercal_{0,j})_{j \in \N_0}$. We set
\begin{equation}\label{eq:initial-value-trans-application-approx}
	\begin{aligned}
		\bTtil^\ad_{0,J}(\bv) 
		&= \sum\limits_{l=0}^{L_\rx(\bv)} \sum\limits_{j=0}^{J-1} (\bT^\intercal_{0,J-j} - \bT^\intercal_{0,J-j-1}) ( \Res_{\hat{\Lambda}_l} \bv_j - \Res_{\hat{\Lambda}_{l-1}} \bv_j)  , \\
		\bBtil^\ad_{2,J}(\bv) &= \bDX \bTtil^\ad_{0,J}(\bv) .
	\end{aligned}
\end{equation}
Here we again apply the differences of the operator approximations to parts with fixed maximum level, and therefore only a constant number of temporal basis functions are affected for each summand.

\begin{lemma}
	Let $\bB_2$ be defined by \eqref{eq:Bpartsdef} and assume that $\bB_2$ is super-compressible. Given any $\bv \in \ell_2(\vee_\rx)$ and $J \in \N_0$, let $\bBtil^\ad_{2,J}(\bv)$ be defined by \eqref{eq:initial-value-trans-application-approx}. Then whenever $\pi^{(i)}(\bv) \in \cAs$ for some $0 < s < s^*$ and $i=1,\dots,d$, one has the a posteriori error bound
	\begin{equation}\label{eq:initial-trans-value-a-posteriori}
		\norm{\bB^\intercal_2 \bv - \bBtil^\ad_{2,J}(\bv)} \leq e^\ad_{0,J}(\bv),
	\end{equation}
	where
	\begin{equation*}
		e^\ad_{0,J}(\bv)  = \sum\limits_{j=0}^J 2^{-s(J-j)} \beta^{(0)}_{J-j} \sum\limits_{i=1}^d \Bignorm{\Res_{\Lambda^{(i)}_{[j]}} \pi^{(i)}(\bv)} + \norm{\bB_2} \Bignorm{\Res_{\Lambda^{(i)}_{[J+1]}} \pi^{(i)}(\bv)} 
	\end{equation*}
	with $\Lambda^{(i)}_{[j]}$ from \eqref{eq:operator-space-only-restriction-sets}, 
	as well as the a priori error bound
	\begin{equation}\label{eq:initial-trans-value-a-priori}
		\norm{\bB^\intercal_2 \bv - \bBtil^\ad_{2,J}(\bv)} \leq 2^s 2^{-sJ} \tilde{C}_0 \sum\limits_{i=1}^d \norm{\pi^{(i)}(\bv)}_{\cAs} .
	\end{equation}
	Moreover, one has the support estimate
	\begin{equation}\label{eq:initial-trans-value-J-support}
		\# \supp\left(\pi^{(\rt,i)} (\bBtil^\ad_{2,J}(\bv))\right) \leq 2c 2^J (L(\bv)+1), \quad i=1,\dots,d ,
	\end{equation}
	where $L(\bv)$ is the maximum one dimensional spatial level of $\bv$.
\end{lemma}
The proof is analogous to the one of Lemma \ref{lem:initial-operator-J}.
Once again we need to replace the scaling matrix $\bDX$ by a low-rank approximation. In the present case, the scaling matrix is applied from the left, which leads to a modified error estimate, which can again be shown by a simplification of the proof of \cite[Lemma 32]{bachmayr2014adaptive}.

\begin{lemma}\label{lem:scaling-to-truncated-only-DX-left}
	Let $\mathbf{T}_0$ be defined as in \eqref{eq:opertor-T1T2-def}, let $\mathbf{\tilde T}\in \R^{\vee \times \vee_\rx}$, $\bv \in \ell_2(\vee_\rx)$, $\sfK \in \N_0^{\vee_\rt}$ be such that $\supp(\bDX \mathbf{\tilde T} \bv) \in \Lambda_{\sfK}$ and let $\bG =  \bDX (\bTtil - \bT^\intercal_0) $.
	Then whenever $\sfn \geq \MX(\eta,\sfK)$ we have
	\begin{align*}
		\lrnorm{( \bDX \bT_0^\intercal  - \mathbf{D}_{\cX,\sfn} \bTtil)\bv}  \leq \Big(1+ \frac{\eta}{1-\delta}\Big) \norm{{\bG}\bv} +  \frac{\eta}{1-\delta} \norm{\bB_2} \norm{\bv} .
	\end{align*}
\end{lemma} 
In view of the above lemma, we choose the approximation parameters
\[
J_{0}^\mathrm{ad}(\eta;\bv) = \min\{ J \in \N_0 : e^\mathrm{ad}_{0,J}(\bv) \leq \tfrac{\eta}{4}  \}, \quad c_0(\bv)\eta = \tfrac{\eta (1-\delta)}{2 \norm{\bB_2} \norm{\bv}}
\]
and 
\[
\sfm_{\cX,0}^\mathrm{ad}(\eta;\bv) = \MX(c_0(\bv)\eta;K^{\mathrm{ad}}_{0}(J_{0}^{\mathrm{ad}}(\eta;\bv);\bv)),
\]
where
\[	
	\sfK^\ad_{0}(J;\bv) = \min \{ \sfK : \supp(\mathbf{\tilde T}^\ad_{0,J}(\bv)) \subseteq \Lambda_{\sfK} \} .
\]
Following the lines of Theorem \ref{thm:error-rescaled-trunc} using Lemma \ref{lem:scaling-to-truncated-only-DX-left}, we obtain the following error bound.

\begin{theorem}\label{thm:error-bound-initial-value-trans-truncated}
	Let $\bv \in \ell_2(\vee_\rx)$ and $0 < \eta < 2 \norm{\bB_2} \norm{\bv}$. We fix
	$J = J^\ad_{0}(\eta;\bv)$ and $\sfn = \sfm^\ad_{\cX,0}(\eta;\bv)$ .
	Then with $\bw_{\eta} =  \bD_{\cX,\sfn} \mathbf{\tilde T}^\ad_{0,J} (\bv)$, one has
	\begin{equation*}
		\norm{\bB^\intercal_2 \bv - \bw_{\eta}} \leq \eta .
	\end{equation*}
\end{theorem}

\begin{lemma}\label{lem:initial-operator-trans-trunc-estimates}
	Under the assumptions of Theorem \ref{thm:error-bound-initial-value-trans-truncated}, we have
	\begin{equation}\label{eq:supp-initial-value-trans-operator}
		\#\supp(\piti(\bw_{\eta})) \leq 2c 4^{1+\frac{1}{s}} \tilde{C}_0^{\frac{1}{s}}(L(\bv)+1) \eta^{-\frac{1}{s}} \bigg(\sum\limits_{i=1}^d \norm{\pi^{(i)}(\bv)}_{\cAs}\bigg)^{\frac{1}{s}} ,
	\end{equation}
	as well as
	\begin{equation}\label{eq:Asnorm-initial-value-trans-operator}
		\norm{\piti(\bw_\eta)}_{\cAs} \leq 2 c^s 2^{4s} \tilde{C}_0  (L(\bv)+1)^s \sum\limits_{i=1}^d \norm{\pi^{(i)}(\bv)}_{\cAs} .
	\end{equation}
	Moreover, we have the rank estimate
	\begin{equation}\label{eq:ranks-initial-value-trans-operator}
		\rank_\infty(\bw_{\eta}) \leq \norm{\hat{\sfm}^\ad_{\cX,0}(\eta;\bv)}_{\ell_\infty} J^\ad_{0}(\eta;\bv) (L(\bv)+1) \rank_\infty(\bv),
	\end{equation}
	where
	\[
	\hat{\sfm}^\ad_{\cX,0,\nut}(\eta;\bv) = 1+ n_{a_{\nut}}^+(\delta) + \sfm^\ad_{\cX,0,\nut}(\eta;\bv), \quad \nut \in \veet .
	\]
\end{lemma}
\begin{proof}
	The estimates \eqref{eq:supp-initial-value-trans-operator} and \eqref{eq:Asnorm-initial-value-trans-operator} can be shown analogously to Lemma \ref{lem:initial-operator-trunc-estimates}. By the representation of the approximation and the structure of $\bB_2^\intercal$ and its approximations, we have at most $\norm{\hat{\sfm}^\ad_{\cX,0}(\eta;\bv)}_{\ell_\infty} J^\ad_{0}(\eta;\bv) (L(\bv)+1)$ summands per time index. The rank of each summand is bounded by $\rank_\infty(\bv)$, which yields \eqref{eq:ranks-initial-value-trans-operator}.	
\end{proof}

\begin{remark}
	Due to the structure of the operator $\bT^\intercal_0$, the maximum rank does not increase by its application, as the input is scaled and distributed over different time indices. However, in the bound \eqref{eq:ranks-initial-value-trans-operator}, we have a potential increase of the rank by the approximate application of $\bB_2^\intercal$. Although it would be possible to construct approximations where rank do not increase, this leads to difficulties in controlling the support sizes and $\cAs$-quasinorms.
\end{remark}

\begin{lemma}\label{lem:apply-initial-value-trans-num-ops}
	Under the assumptions of Theorem \ref{thm:error-bound-initial-value-trans-truncated}, for the number $\text{\rm flops}(\bw_{\eta})$ of floating point operations required to compute $\bw_{\eta}$ we have the bound
	\begin{equation}\label{eq:apply-initial-value-trans-num-ops}
		\begin{aligned}
			\operatorname{flops}(\bw_{\eta}) &\lesssim \\  
			&d (J^\ad_{0}(\eta;\bv))^3 (L(\bv)+1)^3 (C(s,\bB_2)J^\ad_{0}(\eta;\bv) + L(\bv)) \norm{\sfm^\ad_{\cX,0}(\eta;\bv)}_{\ell_\infty}^3 \abs{\rank(\bv)}_{\infty}^3 \\
			+ &8c 4^{\frac{1}{s}} \tilde{C}_0^{\frac{1}{s}} d \eta^{-\frac{1}{s}}   \abs{\rank(\bv)}_{\infty} J^\ad_{0}(\eta;\bv) \norm{\sfm^\ad_{\cX,0}(\eta;\bv)}_{\ell_\infty} (L(\bv)+1)^2 \Biggl( \sum\limits_{i=1}^d \lrnorm{\pi^{(i)}(\bv)}_{\cAs}\Biggr)^{\frac{1}{s}}  ,
		\end{aligned}
	\end{equation}
	where the hidden constant is independent of $\eta, \bv$ and $d$.
\end{lemma}
\begin{proof}
	The lemma can be proved in the same manner as Lemma \ref{lem:apply-num-ops} for the temporal operator. The only difference lies in the size of the temporal support of $\bw_{\eta}$. By construction of the approximation, we obtain
	\[
	L_\rt(\bw_{\eta}) \leq C_1(s,\bB_2) J^\ad_{0}(\eta;\bv) + L(\bv) .
	\]
	In addition, we can make use of the fact that we only have a constant number of wavelet functions per level that does not vanish at time $t=0$. 
	Then, following the lines of Lemma \ref{lem:apply-num-ops} we obtain \eqref{eq:apply-initial-value-trans-num-ops}.
\end{proof}

We conclude the section by estimating $\norm{\hat{\sfm}_{\cX,0}(\eta;\bv)}_{\ell_\infty}$ for $\bv \in \ell_2(\vee)$ and $\norm{\hat{\sfm}^\ad_{\cX,0}(\eta;\bv_\rx)}_{\ell_\infty}$ for $\bv_\rx \in \ell_2(\vee_\rx)$. For the initial value operator $\bB_2$, the scaling matrix is directly applied to the input. Concerning the transposed operator $\bB_2^\intercal$, for a given $\eta>0$ and $\bv_\rx \in \ell_2(\vee_\rx)$ let $J = J^\ad_{0}(\eta;\bv_\rx)$. Then by definition of the operator with the identity in the spatial variables, we have $L_\rx(\mathbf{\tilde T}^\ad_{0,J} \bv_\rx) \leq L(\bv_\rx)$. 

Proceeding in the same way as in Section \ref{sec:adaptive-operator-comb}, we obtain
\begin{align}
	J_0(\eta;\bv) \leq \frac{1}{s} \Big( \abs{\ln \eta} + \ln \Big(C_2(\bB_2) \sum\limits_{i=1}^d \norm{\piti(\bv)}_{\cAs}  \Big)\Big), \label{eq:J0-estimate} \\
	J^\ad_{0}(\eta;\bv_\rx) \leq \frac{1}{s} \Big( \abs{\ln \eta} + \ln \Big(C_2(\bB_2) \sum\limits_{i=1}^d \norm{\pi^{(i)}(\bv_\rx)}_{\cAs}  \Big)\Big), \label{eq:J0-trans-estimate}
\end{align}
and
\begin{align}
	\norm{\hat{\sfm}_{\cX,0}(\eta;\bv)}_{\ell_\infty} &\leq C_3(s,\bB_2) \Big[ 1 + L(\bv) + \abs{\ln(\eta)} + \ln\Big(\sum\limits_{i=1}^d \norm{\piti(\bv)}_{\cAs} \Big) \Big]^2, \label{eq:m-hat-initial-ell-inf} \\
	\norm{\hat{\sfm}^\ad_{\cX,0}(\eta;\bv_\rx)}_{\ell_\infty} &\leq C_4(s,\bB_2) \Big[ 1 + L(\bv_\rx) + \abs{\ln(\eta)} + \ln\Big(\sum\limits_{i=1}^d \norm{\pi^{(i)}(\bv_\rx)}_{\cAs} \Big) \Big]^2 . \label{eq:m-hat-initial-trans-ell-inf} 
\end{align}

\section{Space-Time Adaptive Method for Parabolic Problems}\label{sec:Solver}

\subsection{A convergent solver}\label{sec:convergent-solver}

In the construction of our adaptive low-rank scheme, we follow the strategy of \cite{BachmayrNearOptimal}, which is based on a perturbed Richardson iteration. Following \cite[Algorithm 1]{bachmayr2014adaptive}, the iterative scheme is stated in Algorithm \ref{alg:lowrankrichardson}. 
The main differences in this work to the previous instances of this algorithmic template are the modified realizations of the subroutines $\recompress$, $\coarsen$ and $\apply$, which are the numerical realizations of the operators $\hat{\Pro}_{\eta},\hat{\Cor}_{\eta}$ introduced in Section \ref{sec:recompress-coarsen} and of the adaptive operator approximation described in Section \ref{sec:Apply}, respectively. 

As in \cite{SpaceTimeAdaptive}, from the problem \eqref{eq:Bfdef} with non-symmetric operator $\bB$ of different domain and codomain, we pass on to the equivalent least squares formulation
\begin{align}\label{eq:normal-eq}
	\bB^{\intercal} \bB \bu = \bB^{\intercal} \bfs ,
\end{align}
which is again an elliptic problem: setting $\bA = \bB^\intercal \bB$, the operator $\bA$ is symmetric, bounded and elliptic on $\ell_2(\vee)$. By definition, the adjoint of an $s^*$-compressible operator is also $s^*$-compressible, and the same applies to super-compressible operators. One can define as in \cite{SpaceTimeAdaptive,Cohen00adaptivewavelet} the routine $\apply$ for approximately applying $\bA$ by
\begin{equation}\label{eq:apply-normal-eq}
	\begin{aligned}
		\apply_{\bA} (\bw, \eta) &= \apply_{\bB^{\intercal}}\big(\apply_{\bB}\big(\bw, \tfrac{\eta}{2 \norm{\bB}}\big), \tfrac{\eta}{2} \big) \\
		&= \apply_{\bB_1^{\intercal}}\big(\apply_{\bB_1}\big(\bw, \tfrac{\eta}{4 \norm{\bB}}\big), \tfrac{\eta}{4} \big) \\
		&\quad + \apply_{\bB_2^{\intercal}}\big(\apply_{\bB_2}\big(\bw, \tfrac{\eta}{4 \norm{\bB}}\big), \tfrac{\eta}{4} \big) ,
	\end{aligned}
\end{equation}
where we have used the structure of the operator $\bB$ in the second line.
The routine for approximating the right-hand side is given by
\begin{equation}\label{eq:rhs-normal-eq}
	\begin{aligned}
		\rhs_{\bB^{\intercal} \bfs} (\eta) &= \apply_{\bB^{\intercal}}\big( \rhs_{\bfs}\big(\tfrac{\eta}{2 \norm{\bB}}\big), \tfrac{\eta}{2} \big) \\
		&=\apply_{\bB_1^{\intercal}}\big( \rhs_{\bfs_1}\big(\tfrac{\eta}{4 \norm{\bB}}\big), \tfrac{\eta}{4} \big) + \apply_{\bB_2^{\intercal}}\big( \rhs_{\bfs_2}\big(\tfrac{\eta}{4 \norm{\bB}}\big), \tfrac{\eta}{4} \big).
	\end{aligned}
\end{equation}

In what follows, we write $\apply$ for $\apply_{\bA}$ and $\rhs$ for $\rhs_{\bB^{\intercal} \bfs}$.
Now we have routines which fulfil
\begin{align*}
	&\norm{\bB^{\intercal} \bB \bv - \apply(\bv,\eta)} \leq \eta, \quad &&\norm{\bB^{\intercal} \bfs - \rhs(\eta)} \leq \eta, \\
	&\norm{\bv - \recompress(\bv, \eta)} \leq \eta, \quad &&\norm{\bv - \coarsen(\bv,\eta)} \leq \eta.
\end{align*}
The basic convergence properties of Algorithm \ref{alg:lowrankrichardson} are summarized in the following proposition; see \cite[Proposition 5]{BachmayrNearOptimal} for a proof.

\begin{algorithm}
	\SetAlgoLined
	\KwIn{$\omega > 0$ and $\rho \in (0,1)$ s.t. $\norm{\bI - \omega \bA} \leq \rho$,  \newline $c_\bA \geq \norm{\bA^{-1}}, \varepsilon_0 \geq c_\bA \norm{\bfs}$, \newline $\kappa_1, \kappa_2, \kappa_3 \in (0,1)$ with $\kappa_1 + \kappa_2 + \kappa_3 \leq 1$ und $\beta_1 \geq 0, \beta_2 > 0$.}
	\KwOut{$\bu_\varepsilon$ with  $\norm{\bu- \bu_\varepsilon} \leq \varepsilon$.}
	\Begin{
		$\bu_0 = 0$ \\
		$k = 0$ \\
		$I = \min\left\{ j : \rho^j (1 + (\omega + \beta_1 + \beta_2)j) \leq \frac{1}{2} \kappa_1 \right\}$  \label{alg:jchoice} \\
		\While{$2^{-k} \varepsilon_0 > \varepsilon$}{
			$\mathbf{w}_{k,0} = \bu_k$ \\
			$j = 0$ \\
			\Repeat{$j \geq I \  \lor \  c_\bA \rho \norm{\mathbf{r}_{k,j-1}} + (c_\bA \rho + \omega + \beta_1 + \beta_2) \eta_{k,j-1} \leq \kappa_1 2^{-(k+1)}\varepsilon_0$}{
				$\eta_{k,j} = \rho^{j+1}2^{-k} \varepsilon_0$ \\
				$\mathbf{r}_{k,j} = \apply\left(\mathbf{w}_{k,j}, \frac{1}{2} \eta_{k,j} \right) - \rhs\left(\frac{1}{2} \eta_{k,j} \right)$	\\
				$\mathbf{w}_{k,j+1} = \coarsen\left( \recompress\left(\mathbf{w}_{k,j} - \omega \mathbf{r}_{k,j}, \beta_1 \eta_{k,j} \right), \beta_2 \eta_{k,j} \right)$ \label{alg:line-cor-rec-add}\\
				$j \pluseq 1$
			}
			$\bu_{k+1} = \coarsen\left(\recompress\left(\mathbf{w}_{k,j}, \kappa_2 2^{-(k+1)} \varepsilon_0\right), \kappa_3 2^{-(k+1)} \varepsilon_0 \right)$ \\
			$k \pluseq 1$
		}
		$\bu_\varepsilon = \bu_k$
	}
	\caption{solve$(\bA,\bfs,\varepsilon) = \bu_\varepsilon$}
	\label{alg:lowrankrichardson}
\end{algorithm}

\begin{prop}
	Let the conditions of Algorithm \ref{alg:lowrankrichardson} be fulfilled, in particular let $\omega > 0$ be such that $\norm{\bI - \omega  \bA} \leq \rho < 1$. Then the intermediate results $\bu_k$ satisfy $\norm{\bu_k - \bu} \leq 2^{-k} \varepsilon_0$, and in particular $\norm{\bu_{\varepsilon} - \bu} \leq \varepsilon$.
\end{prop}

\begin{remark}
	The realization of the $\apply$ routine of the transposed operator $\bB_1^\intercal$ can be done in an analogous way to the one for $\bB_1$. Due to symmetry of the Laplacian, the spatial part of the transpose is identical to the spatial part of $\bT$. Thus we have
	\begin{align*}
		\bT^\intercal = \bI_\rt \otimes \bT_{\rx} + \bT^\intercal_{\rt} \otimes \bI_{\rx} .
	\end{align*}
	As mentioned above, the transpose of a super-compressible operator is also super-com\-pres\-si\-ble. Additionally one can show the same estimates for the transposed temporal operator as in Section \ref{sec:Apply-temporal} and Section \ref{sec:adaptive-operator-comb}. Here, the order of the scaling matrices is reversed. However, as shown in Section \ref{sec:low-rank-preconditioning}, we have the same estimates for both scaling matrices. Altogether, we thus obtain exactly the same results for $\apply_{\bB_1^\intercal}$ as for $\apply_{\bB_1}$.
\end{remark}

\subsection{Computational complexity}
In this section, we analyze the computational complexity of Algorithm \ref{alg:lowrankrichardson}. Such estimates require certain assumptions on the approximability of problem data and solutions, to which we then relate the computational costs of the method. Note, however, that the feasibility and convergence of the method is independent of these assumptions.

\begin{ass}\label{ass:A-f}Concerning the scaled matrix representation $\bB$ and the right-hand side $\bbf$ we require the following properties for some fixed $s^*, \tau > 0$.
\begin{enumerate}[{\rm(i)}]
 \item\label{it:B2-compress-lvldecay-sobolev} The one-dimensional operator $\bC_{2}=\mathbf{\hat D} \bT_2 \mathbf{\hat D}$ is $s^*$-compressible with the level decay property and approximations $\bC_{2,j} = \mathbf{\hat D} \bT_{2,j} \mathbf{\hat D}$. Additionally the operator has Sobolev stability of order $\tau$, which means $\norm{\mathbf{\hat D}^{-\tau}(\bC_2 - \bC_{2,j}) \mathbf{\hat D}^{\tau}} < C_{\tau,2} \beta_j(\bC_2)$ for some constant $C_{\tau,2}$.
 \item The one-dimensional operator $\bC_\rt = \bB_\rt \mathbf{\hat D}_\rt $ is super-compressible with the level decay property and  approximations $\bC_{\rt,j} = \bT_{\rt,j} \mathbf{\hat D}$. Additionally the operator has Sobolev stability of order $\tau$, that is, $\norm{\mathbf{\hat D}_{\rt}^{-\tau}(\bC_\rt - \bC_{\rt,j}) \mathbf{\hat D}_{\rt}^{\tau}} < C_{\tau,\rt} \beta_j(\bC_\rt)$ for some constant $C_{\tau,\rt} > 0$.
 \item\label{it:sobolev-initial-value} The operator $\bB_2 = \bT_0 \bDX$ is super-compressible with approximations $\bB_{2,j} = \bT_{0,j} \bDX$. Additionally the operator has Sobolev stability of order $\tau$, that is, there exists $C_{\tau,0} > 0$ with $\norm{\mathbf{\bar D}^{-2\tau}(\bB_2 - \bB_{2,j}) (\mathbf{\hat D}_\rt^{\tau} \otimes \bI_\rx)} \leq C_{\tau,0} \beta_j(\bB_2)$ and $\norm{(\mathbf{\hat D}_\rt^{-\tau} \otimes \bI_\rx)(\bB^\intercal_2 - \bB^\intercal_{2,j})\mathbf{\bar D}^{2\tau} } \leq C_{\tau,0} \beta_j(\bB_2)$.
 \item The number of operations required for evaluating each entry in the approximations $\bT_{2,j}$ and $\bT_{\rt,j}$ is uniformly bounded.
\item We have an estimate $c_\bA \geq \norm{\bA^{-1}}$ with $c_\bA \lesssim \norm{\bA^{-1}}$, and the initial error estimate $\varepsilon_0$ satisfies $\varepsilon_0  \eqsim  \norm{\bA^{-1}}\norm{\bB^\intercal \bbf}$.
\item
The contractions of $\bbf_1$ and $\bbf_2$ are compressible, that is, $\pi^{(\rt,i)}(\bbf_1)\in \cAs(\vee_\rt \times \vee_1)$ and $\pi^{(i)}(\bbf_2) \in \cAs(\vee_1)$, $i=1,\ldots,d$, for any $s$ with $0 < s <s^*$.
\item \label{ass:higher-regularity-A}
For $\bT_2, \bT_\rt$ and $\bbf = [ \mathbf{D}_\cY \mathbf{g}_1, \mathbf{g}_2]$, we assume
\begin{alignat*}{3}
	 \norm{ \mathbf{\hat D}^{1-\tau} \bT_2 \mathbf{\hat D}^{1+\tau} } &< \infty, \quad  \norm{\mathbf{D}_\cY^{1-\tau} &&\mathbf{g}_1 } + \norm{\bD^{-2\tau} \mathbf{g}_2} &&< \infty, \\
	 \norm{\mathbf{\hat D}_\rt^{-\tau} \bT_\rt \mathbf{\hat D}_\rt^{1+\tau}} &< \infty, \qquad &&\norm{\mathbf{\hat D}_\rt^{1-\tau} \bT_\rt^\intercal \mathbf{\hat D}_\rt^{\tau}} &&< \infty,
\end{alignat*}
to which we refer to as \emph{excess regularity of order $\tau$}.
\end{enumerate}
\end{ass}

\begin{remark}
Assumptions \ref{ass:A-f} are satisfied if the wavelets $\Psi_\nu$ are sufficiently regular to be a Riesz basis also for $H^{1+s}(\Omega)$ for some $s>\tau>0$ with appropriate renormalization, and if $g \in L_2(0,T; H^{-1+\tau}(\Omega))$ and $h \in H^{2 \tau}(\Omega)$. 
\end{remark}

\begin{lemma}\label{lem:higher-reg-initial-value}
	For all $0 < \tau < \frac{1}{2}$ we have
	\[
		\lrnorm{\Big(\mathbf{\hat D}_\rt^{-\tau} \otimes \bI_{\rx}\Big) \bB^\intercal_2 \mathbf{\bar D}^{2\tau}} \leq (1+\delta), \quad \lrnorm{\mathbf{\bar D}^{-2\tau} \bB_2\Big(\mathbf{\hat D}_\rt^{\tau} \otimes \bI_{\rx}\Big)} \leq (1+\delta) .
	\]
\end{lemma}
\begin{proof}
	By the definition of the matrices and Young's inequality we have
	\[
		\lrabs{\Big((\mathbf{\hat D}_\rt^{-\tau} \otimes \bI_{\rx}) \mathbf{\bar D}_{\cX}\bT^\intercal_0 \mathbf{\bar D}^{2\tau}\Big)_{(\nut,\nu_\rx),\nu_\rx}} \leq 1 .
	\]
	Combining this with Remark \ref{rem:scaling-DX-delta}, we obtain the first statement. The second statement follows analogously.
\end{proof}

We state next the assumptions concerning the procedure $\rhs$ for approximating the right-hand side $\bbf$ that will be used
in the subsequent complexity analysis. 

\begin{ass}
\label{ass:rhs}
The procedure $\rhs_{\bfs}(\eta) = [\rhs_{\bfs_1}(\tfrac{\eta}{2}) \  \rhs_{\bfs_2}(\tfrac{\eta}{2})]^\intercal$ is assumed to have  the following properties.
\begin{enumerate}[{\rm(i)}]
\setcounter{enumi}{7}
 \item \label{ass:rhsapprox}There exists an approximation $\bbf_{1,\eta} = \rhs_{\bfs_1}(\eta)$ such that $\norm{\bbf_1 - \bbf_{1,\eta}} \leq \eta$ and
  \begin{align*}
  \norm{\piti(\bbf_{1,\eta})}_{\cA^s} &\leq C^\text{{\rm sparse}} \norm{\piti(\bbf_1)}_{\cA^s}    ,\\
   \sum_i \#\supp (\piti(\bbf_{1,\eta})) &\leq C^{\text{{\rm supp}}} \,d \,\eta^{-\frac1s}\, \Bigl(\sum_i \norm{\piti(\bbf_1)}_{\cA^s}\Bigr)^{\frac1s}, \\   
   {\rank_\infty(\bbf_{1,\eta})} &\leq   C_\bbf^{\text{{\rm rank}}}\, ( 1 +  \abs{\ln \eta})^{b_\bbf} \,,   
  \end{align*}
as well as an approximation $\bbf_{2,\eta} = \rhs_{\bfs_2}(\eta)$ such that $\norm{\bbf_2 - \bbf_{2,\eta}} \leq \eta$ and
\begin{align}
	\norm{\pi^{(i)}(\bbf_{2,\eta})}_{\cA^s} &\leq C^\text{{\rm sparse}} \norm{\pi^{(i)}(\bbf_2)}_{\cA^s}    ,\nonumber\\
	\sum_i \#\supp (\pi^{(i)}(\bbf_{2,\eta})) &\leq C^{\text{{\rm supp}}} \,d \,\eta^{-\frac1s}\, \Bigl(\sum_i \norm{\pi^{(i)}(\bbf_2)}_{\cA^s}\Bigr)^{\frac1s}, \nonumber \\   
	\abs{\rank(\bbf_{2,\eta})}_{\infty} &\leq   C_\bbf^{\text{{\rm rank}}}\, ( 1 +  \abs{\ln \eta})^{b_\bbf} \,,   \nonumber \\
	L(\bbf_{2,\eta}) &\leq C^{\text{\rm lvl}} ( \abs{\ln \eta} + \ln d), \label{eq:rhs-f2-ass-lvl}
\end{align}
where $C^\text{{\rm sparse}},C^{\text{{\rm supp}}},C_\bbf^{\text{{\rm rank}}} > 0$,  $b_\bbf \geq 1$ are  independent of $\eta$, and $C^\text{{\rm sparse}}$, $C^{\text{{\rm supp}}}$ are independent of $\bbf$.
 \item \label{ass:rhsops}The number of operations required for evaluating $\rhs_{\bfs}(\eta)$ is bounded, with a constant $C^\text{{\rm ops}}_\bbf(d)$, by
  $\operatorname{flops}(\bbf_\eta) \leq C^\text{{\rm ops}}_\bbf(d) \bigl[ (1+\abs{\ln\eta})^{3b_\bbf} 
    + (1+\abs{\ln \eta})^{b_\bbf}  \eta^{-\frac1s} \bigr]  $. 
 \item \label{ass:higher-regularity-f} $\rhs$ preserves the excess regularity of the problem, that is, there exists $C^\text{{\rm reg}}_\bbf  >0$ independent of $\eta$ such that
 \begin{align}
  \label{eq:rhsreg}
  \bignorm{\mathbf{D}_\cY^{-\tau} \bbf_{1,\eta}} +  \bignorm{\mathbf{D}^{-2\tau} \bbf_{2,\eta}} \leq C^\text{{\rm reg}}_\bbf \Big(\bignorm{\mathbf{D}_\cY^{-\tau} \bbf_{1}} +  \bignorm{\mathbf{D}^{-2\tau} \bbf_{2}}  \Big) .
 \end{align}
\end{enumerate}
\end{ass}

\begin{remark}\label{rem:f2-lvl-estimate-realization}
	In the normal equation \eqref{eq:normal-eq}, $\bB_2^\intercal$ is applied to $\bbf_2$. Due to the estimates in Section \ref{sec:apply-initial-value}, based on the maximum level of the input, we need an estimate for the maximum level of $\bbf_{2,\eta} = \rhs_{\bfs_2}(\eta)$ as well. Such an estimate is obtained by first applying the routine with a lower tolerance and combining this with a coarsening step for a single low-rank representation. Then applying \cite[Lemma 37]{bachmayr2014adaptive} and combining it with the excess regularity assumptions yields estimate \eqref{eq:rhs-f2-ass-lvl}.
\end{remark}

Under the above conditions on the data and their processing, we are now primarily interested in whether the adaptive algorithm produces
low-rank sparse approximate solutions if the exact solution permits such approximations.
We state now our precise {\em benchmark assumptions} on the solution $\bu$.

\begin{ass}\label{ass:approximability}
Concerning the approximability of the solution $\bu$, we   assume:
\begin{enumerate}[{\rm(i)}]
\setcounter{enumi}{10}
\item \label{ass:uapprox}$\bu \in \cA(\gamma_\bu)$ with $\gamma_\bu(n) = e^{d_\bu n^{1/b_\bu}}$
for some $d_\bu>0$, $b_\bu \geq 1$.
\item $\piti(\bu) \in \cA^s$ for $i=1,\ldots,d$, for any $s$ with $0 < s
 <s^*$.
\end{enumerate}
\end{ass}

In order to analyze the dimension-dependence of the complexity of our algorithm, we would ideally need  a {\em reference family} of problems 
exhibiting the same level of difficulty for each $d$. 
Although this is not quite possible, there are elements of problems that are comparable for different values of $d$, for example
the structure of the Laplacian. It is therefore important to state next exactly how the relevant quantities relate to the spatial dimension $d$. 

\begin{ass}\label{ass:dim}
In our comparison of problems for different values of $d$, we assume the following.
\begin{enumerate}[{\rm(i)}]
\setcounter{enumi}{12}
 \item The following are independent of $d$: the constants $d_\bu$, $b_\bu$, $C^\text{{\rm sparse}}$, $C^{\text{{\rm supp}}}$, $C_\bbf^{\text{{\rm rank}}}$, $C_\bbf^{\text{{\rm lvl}}}$; the excess regularity index $\tau$, and $C^\text{{\rm reg}}_\bbf$ in \eqref{eq:rhsreg}.
  \item The following quantities remain bounded independently of $d$: $\norm{\bB}$, $\norm{\bB^{-1}}$, as well as $\norm{\mathbf{D}_\cY^{1-\tau} \mathbf{g}_1 } + \norm{\bD^{-\tau} \mathbf{g}_2} $;
  the quantities $\norm{\bu}_{\cA(\gamma_\bu)}$ and $\norm{\piti(\bu) }_{\cA^s}$ in Assumptions \ref{ass:approximability} and $\norm{\piti(\bbf_1)}_{\cA^s}$, $\norm{\pi^{(i)}(\bbf_2)}_{\cA^s}$ in Assumptions \ref{ass:rhs}\eqref{ass:rhsapprox}, each for $i=1,\ldots,d$.
\item  In addition, we assume that  $C^\text{{\rm ops}}_\bbf(d)$ as in Assumptions \ref{ass:approximability}\eqref{ass:rhsops} grows at most polynomially as $d\to \infty$.
 \end{enumerate}
 \end{ass}

\begin{remark}\label{rem:indpence-rho-omega}
	By the existence of $d$-independent bounds on $\norm{\bB}$ and $\norm{\bB^{-1}}$, we have that $\norm{\bA}$ and $\norm{\bA^{-1}}$ are bounded independently of $d$ as well, and thus the reduction parameter $\rho$ satisfies $\rho < \hat\rho < 1$ with some $\hat \rho$ independent of $d$.
\end{remark}
 
 Under the above benchmark assumptions, we obtain the following bound on the complexity of the computed approximations and on the total computational costs of the method.

\begin{theorem}
\label{thm:complexity}
Suppose that  Assumptions \ref{ass:A-f}, \ref{ass:rhs} hold and that Assumptions \ref{ass:approximability} are valid for the solution $\bu$  of $\mathbf{A}\bu = \mathbf{f}$.
Let $\alpha > 0$ and let $\constsvd = \sqrt{2d-3}, \constcrs = \sqrt{d}$.
Let the constants $\kappa_1,\kappa_2,\kappa_3$ in
Algorithm \ref{alg:lowrankrichardson} be chosen as
\begin{gather*}
  \kappa_1 = \bigl(1 + (1+\alpha)(\constsvd + \constcrs +
  \constsvd\constcrs)\bigr)^{-1}\,, \\
  \kappa_2 = (1+\alpha)\constsvd \kappa_1\,,\qquad 
  \kappa_3 = \constcrs(\constsvd + 1)(1+\alpha)\kappa_1 \,,
\end{gather*}
and let $\beta_1 \geq 0$, $\beta_2 >0$ be arbitrary but fixed.
Then the approximate solution $\bu_\varepsilon$ produced by Algorithm \ref{alg:lowrankrichardson} for $\varepsilon < \varepsilon_0$
satisfies
\begin{gather}
 \label{eq:complexity_rank} 
  \rank_\infty(\bu_\varepsilon)
   \leq  \gamma^{-1}_\bu \bigl( 2(\alpha \kappa_1)^{-1} \rho_{\gamma_\bu}    \,\norm{\bu}_{\cA(\gamma_\mathbf{\bu})}\,\varepsilon^{-1} \bigr)
   \\
 \label{eq:complexity_supp} \sum_{i=1}^d \#\supp(\piti(\bu_\varepsilon)) \lesssim
     d^{1 + s^{-1}} \, \Bigl(\sum_{i=1}^d \norm{ \piti(\bu)}_{\cA^s} \Bigr)^{\frac{1}{s}}
           \varepsilon^{-\frac{1}{s}} \,,
\end{gather}
as well as
\begin{gather}
  \label{eq:complexity_ranknorm} 
  \norm{\bu_\varepsilon}_{\cA(\gamma_\mathbf{\bu})}
  \lesssim \sqrt{d}\,
      \norm{\bu}_{\cA(\gamma_\mathbf{\bu})}    \,,   \\
  \label{eq:complexity_sparsitynorm} \sum_{i=1}^d \norm{
  \piti(\bu_\varepsilon)}_{\cA^s} \lesssim d^{1 + \max\{1,s\}} 
      \sum_{i=1}^d \norm{ \piti(\bu)}_{\cA^s}  \,.
\end{gather}
The multiplicative constant in \eqref{eq:complexity_ranknorm} depends only on $\alpha$,
those in \eqref{eq:complexity_supp} and \eqref{eq:complexity_sparsitynorm} depend only on
$\alpha$ and $s$.

If in addition, Assumptions \ref{ass:dim} hold, then for the number of required operations $\operatorname{flops}(\bu_\varepsilon)$ we have the estimate
\begin{equation} 
\label{eq:complexity_totalops}
\operatorname{flops}(\bu_\varepsilon) \leq Cd^a \,d^{2c s^{-1} \ln d} (\ln d)^{102 c  \ln d} (\ln d)^b
  (1+ \abs{\ln \varepsilon})^{16+34 c  \ln d + 4 \max\{b_\bu, b_\mathbf{f}\}}\,  
       \varepsilon^{-\frac{1}{s}} \,
\end{equation}
where $C,a,b$ are constants independent of $\varepsilon$ and $d$, and $c$ is the smallest $d$-independent value such that $I \leq c \ln d$ for $I$ as in line \ref{alg:jchoice} of Algorithm \ref{alg:lowrankrichardson}. In particular,  $c$   does not depend on $\varepsilon$ and $s$.
\end{theorem}

The proof of this theorem requires some preparations.
Recall that we aim to solve the normal equation \eqref{eq:normal-eq}. For notational simplification, we write
\begin{align*}
	\bzkj^{(1)} = \apply_{\bB_1}\left(\bwkj, \tfrac{\eta_{k,j}}{8 \norm{\bB}}\right), \quad \bzkj^{(2)} = \apply_{\bB_2}\left(\bwkj, \tfrac{\eta_{k,j}}{8 \norm{\bB}}\right)
\end{align*}
for the intermediate results arising in the approximate application of $\bB^\intercal \bB$. In contrast to the case of a single low-rank representation as in \cite{bachmayr2014adaptive}, due to the approximation of the initial value operator considered in Section \ref{sec:apply-initial-value}, the support and the $\cAs$-norm of the results of the routine $\apply$ depend additionally on the maximum spatial level of the input. 

\subsubsection{Maximum level and support of the iterates}
The following lemma estimates the maximum current spatial and temporal wavelet levels in the output of $\coarsen(\bv,\eta)$. It depends on the excess regularity assumptions as well as on the mixed lower dimensional support $\supp \piti(\bv)$ of the temporal variable combined with the $i$th spatial variable.

\begin{lemma}\label{lem:max-level}
	For given $\bv \in \ell_2(\vee)$, we consider $\bp = (\piti_{\nu_\rt,\nu_i}(\bv))_{i,\nu_\rt,\nu_i}$ as a vector on  $\{1,\dots,d\} \times \vee_\rt \times \vee_1$. Assume that
	\begin{align*}
		\#\supp \bp = \sum\limits_{i=1}^d \# \supp \piti(\bv) < \infty
	\end{align*}
	and that for some $\tau > 0$, one has $\norm{(\mathbf{\hat D}_\rt^{-\tau} \otimes \bI_i) \piti(\bv)} < \infty$ and $\norm{(\bI_\rt \otimes \mathbf{\hat D}^{-\tau}) \piti(\bv)} < \infty$ for all $i=1,\dots,d$. Let $\eta > 0$, let $\bp_\eta$ be the vector with minimal support in $\{1,\dots,d\} \times \vee_\rt \times \vee_1$ such that $\norm{\bp - \bp_\eta} \leq \eta$, and let \[  C_{\vee_\rt} = \sup_{\nu \in \vee_\rt} \norm{\theta_{\vee_\rt}}_{H^1(0,T)}^{-\tau} 2^{\tau \abs{\nu}}, \quad C_{\vee_1} = \sup_{\nu \in \vee_1} \norm{\psi_\nu}_{H^1_0(0,1)}^{-\tau} 2^{\tau \abs{\nu}} \,. \] 
	Then for all $(i,\nu_\rt,\nu_i) \in \supp \ \bp_\eta$, we have
	\begin{align*}
		\abs{\nu_\rt} &\leq  \tau^{-1} \log_2 \left( \eta^{-1} C_{\vee_\rt} \norm{(\mathbf{\hat D}_\rt^{-\tau} \otimes \bI_i) \piti(\bv)} \sqrt{\# \supp \bp} \right) , \\
		\abs{\nu_i} &\leq  \tau^{-1} \log_2 \left( \eta^{-1} C_{\vee_1} \norm{(\bI_\rt \otimes \mathbf{\hat D}^{-\tau}) \piti(\bv)} \sqrt{\# \supp \bp} \right) .
	\end{align*}
\end{lemma}
The proof can be done exactly as in \cite[Lemma 37]{bachmayr2014adaptive} and is therefore not repeated here.
This lemma will be applied to line \ref{alg:line-cor-rec-add} of Algorithm \ref{alg:lowrankrichardson} with $\eta = \beta_2 \eta_{k,j}$ and $\bp_\eta$ the result of $\coarsen$. By definition of the routine via the index set $\Lambda(\bu;N)$, which is based on the total ordering of all contractions $\piti(\bv)$, and  since $\beta_2 > 0$, the assumptions of Lemma \ref{lem:max-level} are satisfied.

As a consequence of Lemma \ref{lem:max-level}, it suffices to estimate $\norm{(\mathbf{\hat D}_\rt^{-\tau} \otimes \bI_i) \piti(\bw_{k,j})}$ and $ \norm{(\bI_\rt \otimes \mathbf{\hat D}^{-\tau}) \piti(\bw_{k,j})}$ to arrive at suitable bounds on the maximum wavelet levels of the iterates $\bwkj$.
To this end, we adapt \cite[Lemma 38]{bachmayr2014adaptive} to obtain the following result on the stability of the contractions under basis coarsening and rank truncation.
\begin{lemma}\label{lem:coarse-trunc-stability}
	For any $\tau,\eta > 0$ and $i=1,\dots,d$, we have
	\begin{align*}
		\norm{(\mathbf{\hat D}_\rt^{-\tau} \otimes \bI_i) \piti(\hat{C}_\eta \bv)} &\leq \norm{(\mathbf{\hat D}_\rt^{-\tau} \otimes \bI_i) \piti(\bv)} , \\
		\norm{(\mathbf{\hat D}_\rt^{-\tau} \otimes \bI_i) \piti(\hat{P}_\eta \bv)} &\leq \norm{(\mathbf{\hat D}_\rt^{-\tau} \otimes \bI_i) \piti(\bv)} ,\\
		\norm{(\bI_\rt \otimes \mathbf{\hat D}^{-\tau}) \piti(\hat{C}_\eta \bv)} &\leq \norm{(\bI_\rt \otimes \mathbf{\hat D}^{-\tau}) \piti(\bv)} ,\\
		\norm{(\bI_\rt \otimes \mathbf{\hat D}^{-\tau}) \piti(\hat{P}_\eta \bv)} &\leq \norm{(\bI_\rt \otimes \mathbf{\hat D}^{-\tau}) \piti(\bv)} , 
	\end{align*}
	for any $\bv \in \ell_2(\vee)$.
\end{lemma}

Next, we consider the evolution of $\norm{(\mathbf{\hat D}_\rt^{-\tau} \otimes \bI_i) \piti(\bw_{k,j})}$ and $ \norm{(\bI_\rt \otimes \mathbf{\hat D}^{-\tau}) \piti(\bw_{k,j})}$ over the course of the iteration. We mention here that due to the excess regularity Assumptions~\ref{ass:A-f}\eqref{ass:higher-regularity-A} on $\bB$ and Assumptions \ref{ass:rhs}\eqref{ass:higher-regularity-f} on $\bfs$, one has
\begin{align}
	\zeta &=  \max\left\{\bignorm{(\mathbf{\hat D}_\rt^{-\tau} \otimes \bI_\rx) \mathbf{g}_1},  \bignorm{(\bI_\rt \otimes \bD^{-\tau})\mathbf{g}_1}, \bignorm{ \mathbf{D}^{-\tau} \mathbf{g}_2 } \right\} < \infty, \nonumber \\
	\xi &= \max\left\{\lrnorm{\mathbf{\hat D}_\rt^{-\tau} \bC_\rt \mathbf{\hat D}_\rt^{\tau}}, \lrnorm{\mathbf{\hat D}_\rt^{-\tau} \bC_\rt^\intercal \mathbf{\hat D}_\rt^{\tau}}, \lrnorm{\mathbf{\hat D}^{-\tau} \bC_2 \mathbf{\hat D}^{\tau}} \right\} < \infty \label{eq:higher-reg-max}.
\end{align}
Since the spatial operator $\bC_2$ is symmetric, we have $\norm{\mathbf{\hat D}^{-\tau} \bC_2^\intercal \mathbf{\hat D}^{\tau}} \leq \xi$.

\begin{prop}\label{prop:higher-reg}
	Under the assumptions of Theorem \ref{thm:complexity}, the iterates $\bw_{k,j}$ of Algorithm~\ref{alg:lowrankrichardson} satisfy
	\begin{align*}
		\max \Bigl\{ \norm{(\mathbf{\hat D}_\rt^{-\tau} \otimes \bI_i) \piti(\bw_{k,j})},\norm{(\bI_\rt \otimes \mathbf{\hat D}^{-\tau}) \piti(\bw_{k,j})}  \Bigr\} \leq \frac{\gamma^{kI+j+1} -1}{\gamma - 1} \bar{C}_{\bfs} ,
	\end{align*}
	where
	\begin{align*}
		\gamma = 1+ \omega \max \left\{ \gamma_{\rt,1}^2+ \gamma_{\rt,2}^2, \gamma_{\rx,1}^2 + \gamma_{\rx,2}^2 \right\}	, \quad \bar{C}_{\bbf} = \omega C_{\bbf}^\text{{\rm reg}} \zeta \max \left\{ \gamma_{\rt,1} + \gamma_{\rt,2},  \gamma_{\rx,1} + \gamma_{\rx,2} \right\}
	\end{align*}
	with
	\begin{align*}
		\gamma_{\rx,1} &=  (1+\delta)^2(\check{C}_\rx + 2 \norm{\bC_\rt} + \xi + C_{\tau,2} \norm{\beta(\bC_2)}_{\ell_1}), \quad  \gamma_{\rx,2} = 2 \norm{\bB_2} \\
		\gamma_{\rt,1} &=(1+\delta)^2(\check{C}_\rx + 2 \norm{\bC_2} + \xi + C_{\tau,\rt} \norm{\beta(\bC_\rt)}_{\ell_1}) , \quad \gamma_{\rt,2} = C_{\tau,0} \norm{\beta(\bB_2)}_{\ell_1} + (1+\delta).
	\end{align*}
\end{prop}
\begin{proof}
	The proof follows the idea of \cite[Proposition 39]{bachmayr2014adaptive}, with adaptations to the additional time component and to the structure of the normal equation. We give the proof for $\norm{(\bI_\rt \otimes \mathbf{\hat D}^{-\tau}) \piti(\bw_{k,j})}$; for $\norm{(\mathbf{\hat D}_\rt^{-\tau} \otimes \bI_i) \piti(\bw_{k,j})}$ we will only cover the part with the application of $\bB_2$, for the rest one can proceed analogously.
	In the outer loop $k$ each inner loop step $j$ is of the form
	\begin{align*}
		\bw_{k,j+1} = \hat{C}_{\beta_2 \eta_{k,j}} \left( \hat{P}_{\beta_1 \eta_{k,j}} \big( (\bI - \omega \bAtil_{k,j}) \bw_{k,j} + \omega \bfs_{k,j} \big) \right) .
	\end{align*}
	By definition of $\bu_0$ in Algorithm \ref{alg:lowrankrichardson}, we have $\norm{(\bI_\rt \otimes \mathbf{\hat D}^{-\tau}) \piti(\bu_0)} = 0$ for $i=1,\dots,d$.
	Using Lemma \ref{lem:coarse-trunc-stability} yields
	\begin{align*}
		\norm{(\bI_\rt \otimes \mathbf{\hat D}^{-\tau}) \piti(\bw_{k,j+1})} &= \norm{(\bI_\rt \otimes \mathbf{\hat D}^{-\tau}) \piti(\hat{C}_{\beta_2 \eta_{k,j}} ( \hat{P}_{\beta_1 \eta_{k,j}} ( (\bI - \omega \bAtil_{k,j}) \bw_{k,j} + \omega \bfs_{k,j} ) ))} \\
		&\leq \norm{(\bI_\rt \otimes \mathbf{\hat D}^{-\tau}) \piti(\bw_{k,j})} + \omega \norm{(\bI_\rt \otimes \mathbf{\hat D}^{-\tau}) \piti(\bAtil_{k,j} \bw_{k,j})} \\ &\quad + \omega \norm{(\bI_\rt \otimes \mathbf{\hat D}^{-\tau}) \piti(\bbf_{k,j})} .
	\end{align*}
	Taking into account the definition of $\bAtil_{k,j}$, we need to estimate
	\[
		\norm{(\bI_\rt \otimes \mathbf{\hat D}^{-\tau}) \piti(\bAtil_{k,j} \bw_{k,j})} \leq \norm{(\bI_\rt \otimes \mathbf{\hat D}^{-\tau}) \piti(\bAtil^{(1)}_{k,j} \bw_{k,j})} + \norm{(\bI_\rt \otimes \mathbf{\hat D}^{-\tau}) \piti(\bAtil^{(2)}_{k,j} \bw_{k,j})},
	\]
	where $\bAtil^{(l)}_{l,j}$ for $l=1,2$ is defined by
	\begin{align*}
		\norm{(\bI_\rt \otimes \mathbf{\hat D}^{-\tau}) \piti(\bAtil^{(l)}_{k,j} \bw_{k,j})} &= \lrnorm{(\bI_\rt \otimes \mathbf{\hat D}^{-\tau}) \piti\big(\apply_{\bB_l^\intercal}\big(\bzkj^{(l)}, \tfrac{\eta_{k,j}}{8}\big) \big) } \\
		&= \lrnorm{(\bI_\rt \otimes \mathbf{\hat D}^{-\tau}) \piti\big(\apply_{\bB_l^\intercal}\big(\apply_{\bB_l} \big(\bwkj, \tfrac{\eta_{k,j}}{8 \norm{\bB}}\big), \tfrac{\eta_{k,j}}{8}\big) \big) } .
	\end{align*}
	We start with $l=1$. We proceed by the stepwise estimates
	\begin{align*}
		\lrnorm{(\bI_\rt \otimes \mathbf{\hat D}^{-\tau}) \piti\left(\apply_{\bB_1^\intercal}\left(\bzkj^{(1)}, \tfrac{\eta_{k,j}}{8}\right) \right) } &\leq \gamma_{\rx,1}  \lrnorm{(\bI_\rt \otimes \mathbf{\hat D}^{-\tau}) \piti(\bzkj^{(1)}) }, \\
		\lrnorm{(\bI_\rt \otimes \mathbf{\hat D}^{-\tau}) \piti(\bzkj^{(1)}) } &\leq \gamma_{\rx,1} \lrnorm{(\bI_\rt \otimes \mathbf{\hat D}^{-\tau}) \piti(\bwkj) }.
	\end{align*}
	The two estimates can be proved in the same manner, and we thus state the proof only for the second estimate.
	
	In what follows, we write $\mathbf{\bar D}_i = \bI_1\otimes\cdots \otimes \bI_{i-1}\otimes \mathbf{\hat D} \otimes \bI_{i+1}\otimes\cdots\otimes \bI_d$.
	By definition, we have $ \bzkj^{(1)} =  \bD_{\cY,\sfn_2} \bTtil_{k,j} \bD_{\cX,\sfn_1} \bwkj$ with $\bTtil_{k,j} = \bTtil_{J(\eta;\bwkj)}[\bwkj]$, $\eta = \tfrac{\eta_{k,j}}{8\norm{\bB}} , \sfn_1 = \hmcX(\eta;\bw_{k,j})$ and $\sfn_2 = \hmcY(\eta; \bw_{k,j})$.
	We define $\mathbf{\tilde w}_{k,j} = \bDX^{-1} \bDXtildeNN \bw_{k,j}$. Using that each entry of the diagonal matrix $\bD_{\cY,\sfn_2} \bD_{\cY}^{-1}$ is bounded by one and that diagonal operators commute, we get
	\begin{align*}
		\lrnorm{\piti \left( (\bI_\rt \otimes \mathbf{\bar D}_i^{-\tau}) \mathbf{D}_{\cY,\sfn_2} \bTtil_{k,j} \bD_{\cX,\sfn_1} \bw_{k,j} \right)} \leq  \lrnorm{\piti\left( (\bI_\rt \otimes \mathbf{\bar D}_i^{-\tau}) \bDY \bTtil_{k,j} \bDX \mathbf{\tilde w}_{k,j} \right)} .
	\end{align*}
	In the next step we make use of the structure of $\bT$ as in \eqref{eq:contraction-scaling}, and additionally take into account the temporal operator. For notational simplicity we restrict ourselves to $i=1$ from here on. For $\bTtil_\rt = \bTtil_{\rt,J(\eta;\bwkj)}[\bwkj]$, we have
	\begin{equation}\label{eq:higher-regularity-split-space-time}
		\lrnorm{\pitiOne\left( (\bI_\rt \otimes \mathbf{\bar D}_1^{-\tau}) \bDY \bTtil_{k,j} \bDX \mathbf{\tilde w}_{k,j} \right)} \leq \lrnorm{\mathbf{d}_{\rx}} + \lrnorm{\pitiOne\left( (\bI_\rt \otimes \mathbf{\bar D}_1^{-\tau})  \bTtil_\rt \bDX   \mathbf{\tilde w}_{k,j} \right)}
	\end{equation}
	with $\mathbf{d}_{\rx} = \bigl( D_{1,\nu_\rt,\nu_1} + D_{2,\nu_\rt,\nu_1}\bigr)_{\nu_\rt \in \vee_\rt, \nu_1\in\vee_1}$, where
	\begin{align*}
		D_{1,\nu_\rt,\nu_1} &= \pi^{(1)}_{\nu_1}\Big(  \mathbf{\bar D}_1^{-\tau} \bD (\bI_1 \otimes \sum\limits_{\sfn \in \KdOneTwo} c_\sfn \bigotimes_{i=2}^d \bTtil^{(i)}_{\nut,n_i}[\bwkj]) \bDXnut (\mathbf{\tilde w}_{k,j})_{\nut} \Big) , \\
		D_{2,\nu_\rt,\nu_1} &= \pi^{(1)}_{\nu_1}\left( \mathbf{\bar D}_1^{-\tau} \bD (\bTtil^{(1)}_{\nut,2}[\bwkj] \otimes \bI_2 \otimes \cdots \otimes \bI_d) \bDXnut (\mathbf{\tilde w}_{k,j})_{\nut} \right) .
	\end{align*}
	To estimate $D_{1,\nu_\rt,\nu_1}$, we can use that diagonal matrices commute and that also $(\mathbf{\bar D}_1^{-\tau})$ and $\bI_1 \otimes \bM$ commute for each $\bM$ by definition of $\mathbf{\bar D}_1^{-\tau}$. We thus arrive at
	\begin{align*}
		D_{1,\nu_\rt,\nu_1} \leq \check{C}_\rx (1+\delta)^2 \pitiOne_{\nu_\rt,\nu_1}((\bI_\rt \otimes \mathbf{\bar D}_1^{-\tau}) \bw_{k,j}) .
	\end{align*}
	
	To estimate $D_{2,\nu_\rt,\nu_1}$, we can proceed similarly to \cite[Proposition 39]{bachmayr2014adaptive} and Lemma \ref{lem:apply-quantity-estimates}. We define $\mathbf{\tilde C}^{(1)}_{\nut,2} = \mathbf{\hat D} \mathbf{\tilde T}^{(1)}_{\nut,2}[\bwkj] \mathbf{\hat D}$. Using that the entries of the diagonal matrices $\bD \mathbf{\bar D}_1^{-1}$ and $\mathbf{\bar D}_1^{-1}\bDXnut$ are bounded by $(1+\delta)$, we obtain
	\begin{align*}
		D_{2,\nu_\rt,\nu_1} &\leq (1+\delta) \pi^{(1)}_{\nu_1}\left(\mathbf{\bar D}_1^{-\tau} \mathbf{\bar D}_1  (\bTtil^{(1)}_{\nut,2}[\bwkj] \otimes \bI_2 \otimes \cdots \otimes \bI_d)  \mathbf{\bar D}_1 \mathbf{\bar D}_1^{-1} \bDXnut (\mathbf{\tilde w}_{k,j})_{\nut} \right) \\
		&= (1+\delta) \pi^{(1)}_{\nu_1}\left(\mathbf{\bar D}_1^{-\tau}  (\mathbf{\tilde C}^{(1)}_{\nut,2} \otimes \bI_2 \otimes \cdots \otimes \bI_d)  \mathbf{\bar D}_1^{-1} \bDXnut (\mathbf{\tilde w}_{k,j})_{\nut} \right) \\
		&\leq (1+ \delta) \pi^{(1)}_{\nu_1}\left(\mathbf{\bar D}_1^{-\tau}  (\bC_2 \otimes \bI_2 \otimes \cdots \otimes \bI_d)  \mathbf{\bar D}_1^{-1} \bDXnut (\mathbf{\tilde w}_{k,j})_{\nut} \right) \\
		&\ + (1+\delta) \pi^{(1)}_{\nu_1}\left(\mathbf{\bar D}_1^{-\tau}  ((\mathbf{\tilde C}^{(1)}_{\nut,2} - \bC_2) \otimes \bI_2 \otimes \cdots \otimes \bI_d)  \mathbf{\bar D}_1^{-1} \bDXnut (\mathbf{\tilde w}_{k,j})_{\nut} \right) \\
		&\leq (1+ \delta) \pi^{(1)}_{\nu_1}\left(\mathbf{\bar D}_1^{-\tau}  (\bC_2 \otimes \bI_2 \otimes \cdots \otimes \bI_d)  \mathbf{\bar D}_1^{-1} \bDXnut (\mathbf{\tilde w}_{k,j})_{\nut} \right) \\
		&\ + (1+\delta)^2 C_{\tau,2} \norm{\beta(\bC_2)}_{\ell_1} \pi^{(1)}_{\nu_1}\left(\mathbf{\bar D}_1^{-\tau}   (\mathbf{\tilde w}_{k,j})_{\nut} \right) ,
	\end{align*}
	where we used Assumption \ref{ass:A-f}\eqref{it:B2-compress-lvldecay-sobolev} in the last line.
	
	For the second summand in \eqref{eq:higher-regularity-split-space-time} we can use that diagonal matrices commute, and due to the special structure with restrictions only in the spatial part of the operator $\bTtil_\rt$, we obtain $(\bI_\rt \otimes \mathbf{\bar D}_1^{-\tau})\bTtil_\rt \bDX = \bTtil_\rt \bDX (\bI_\rt \otimes \mathbf{\bar D}_1^{-\tau})$. Hence, we arrive at
	\begin{align*}
	\lrnorm{\pitiOne\left( (\bI_\rt \otimes \mathbf{\bar D}_1^{-\tau})  \bTtil_\rt \bDX   \mathbf{\tilde w}_{k,j} \right)} \leq 2 (1+\delta)^2 \norm{\bC_\rt} \lrnorm{\pitiOne((\bI_\rt \otimes \mathbf{\bar D}_1^{-\tau})\bw_{k,j})} .
	\end{align*}

	To obtain a bound on $\norm{\mathbf{d}_{\rx}}$, we sum the estimates for $D_{1,\nu_\rt,\nu_1}$ and $D_{2,\nu_\rt,\nu_1}$ over all $\nu_\rt,\nu_1$ and use \eqref{eq:higher-reg-max}, which yields
	\begin{align*}
		\norm{(\bI_\rt \otimes \mathbf{\hat D}^{-\tau}) \pitiOne( \bz^{(1)}_{k,j})} \leq \gamma_{\rx,1} \norm{(\bI_\rt \otimes \mathbf{\hat D}^{-\tau}) \pitiOne(\bw_{k,j})} .
	\end{align*}
	For $\norm{(\bI_\rt \otimes \mathbf{\hat D}^{-\tau}) \piti(\bAtil^{(2)}_{k,j} \bw_{k,j})}$ we can proceed in the same way as for the temporal operator using the fact that we only have restrictions in the spatial variables. Hence,
	\[
		\norm{(\bI_\rt \otimes \mathbf{\hat D}^{-\tau}) \piti(\bAtil^{(2)}_{k,j} \bw_{k,j})} \leq \gamma_{\rx,2}^2 \norm{(\bI_\rt \otimes \mathbf{\hat D}^{-\tau}) \pitiOne(\bw_{k,j})} .
	\]

	For $\bfs_{k,j} = \rhs(\tfrac{1}{2}\eta_{k,j})$ with $\rhs$ defined in \eqref{eq:rhs-normal-eq}, we can proceed similarly and combine the result with 
	\eqref{eq:rhsreg} from Assumption \ref{ass:rhs}\eqref{ass:higher-regularity-f}.
	Thus we arrive at
	\begin{align*}
		\norm{(\bI_\rt \otimes \mathbf{\hat D}^{-\tau}) \pitiOne(\bw_{k,j+1})} &\leq (1 + \omega (\gamma_{\rx,1}^2+\gamma_{\rx,2}^2))\norm{(\bI_\rt \otimes \mathbf{\hat D}^{-\tau}) \pitiOne(\bw_{k,j})} + \bar{C}_{\bbf} \\
		&\leq \gamma \norm{(\bI_\rt \otimes \mathbf{\hat D}^{-\tau}) \pitiOne(\bw_{k,j})} + \bar{C}_{\bbf}  .
	\end{align*}
	
	This estimate is preserved by the recompression and coarsening steps at the end of the outer loop. By taking $j \leq I$ and $\norm{(\bI_\rt \otimes \mathbf{\hat D}^{-\tau}) \pitiOne(\bu_0)} = 0$ into account, we thus obtain
	\begin{align*}
		\norm{(\bI_\rt \otimes \mathbf{\hat D}^{-\tau}) \pitiOne(\bw_{k,j})} \leq \frac{\gamma^{kI+j+1} -1}{\gamma - 1} \bar{C}_{\bfs} .
	\end{align*}
	
	For the quantity $\norm{(\mathbf{\hat D}_\rt^{-\tau} \otimes \bI_i) \pitiOne(\bw_{k,j})}$, we can proceed  analogously for the operator $\bB_1$, using the fact that in this case $(\mathbf{\hat D}_\rt^{-\tau} \otimes \bI_i)$ and the spatial operator commute. 
	However, for the operator $\bB_2$ we need to proceed differently. We again use stepwise estimates, here in the form
	\begin{equation} \label{eq:higher-regularity-time-gammat}
		\begin{aligned}
		\lrnorm{(\mathbf{\hat D}_\rt^{-\tau} \otimes \bI_i) \piti\big(\apply_{\bB_1^\intercal}\big(\bzkj^{(2)}, \tfrac{\eta_{k,j}}{8}\big) \big) } &\leq \gamma_{\rt,2}  \lrnorm{\mathbf{\bar D}^{-2\tau}  \bzkj^{(2)} }, \\
		\lrnorm{\mathbf{\bar D}^{-2\tau} \bzkj^{(2)} } &\leq \gamma_{\rt,2} \lrnorm{(\mathbf{\hat D}_\rt^{-\tau} \otimes \bI_i) \piti(\bwkj) }. 
		\end{aligned}
	\end{equation}

	By definition, we have $ \bzkj^{(2)} =  \bTtil_{0,J} \bD_{\cX,\sfn} \bwkj$ with $\bTtil_{0,J} = \bTtil_{0,J_0(\eta;\bwkj)}[\bwkj]$, $\eta = \tfrac{\eta_{k,j}}{8\norm{\bB}} $ and $\sfn = \hat{\sfm}_{\cX,0}(\eta;\bw_{k,j})$. Let $\mathbf{\tilde w}_{k,j} = \bDX^{-1} \bD_{\cX,\sfn} \bw_{k,j}$. Then
	\begin{equation*}
		\begin{aligned}
			\lrnorm{\mathbf{\bar D}^{-2\tau} \bzkj^{(2)} } &= \lrnorm{\mathbf{\bar D}^{-2\tau}  \bTtil_{0,J} \bD_{\cX} (\mathbf{\tilde w}_{k,j}) } \\
			&\leq \lrnorm{\mathbf{\bar D}^{-2\tau}  (\bT_0 - \mathbf{\tilde T}_{0,J}) \bDX  (\mathbf{\tilde w}_{k,j}) } + \lrnorm{\mathbf{\bar D}^{-2\tau}  \bB_2  \mathbf{\tilde w}_{k,j} } \\
			&\leq (C_{\tau,0} \norm{\beta(\bB_2)}_{\ell_1} + (1+\delta)) \norm{(\mathbf{\hat D}_{\rt}^{-\tau} \otimes \bI_i) \piti(\mathbf{\tilde w}_{k,j})},
		\end{aligned}
	\end{equation*}
	where we used Assumption \ref{ass:A-f}\eqref{it:sobolev-initial-value} and Lemma \ref{lem:higher-reg-initial-value} in the last line. For \eqref{eq:higher-regularity-time-gammat} we can proceed in the same manner.
	We thus arrive at
	\begin{align*}
		\norm{(\mathbf{\hat D}_\rt^{-\tau} \otimes \bI_i) \piti(\bw_{k,j+1})} &\leq (1 + \omega (\gamma_{\rt,1}^2 + \gamma_{\rt,2}^2)) \norm{(\mathbf{\hat D}_\rt^{-\tau} \otimes \bI_i) \piti(\bw_{k,j})} + \bar{C}_\bbf \\
		&\leq \gamma \norm{(\mathbf{\hat D}_\rt^{-\tau} \otimes \bI_i) \piti(\bw_{k,j})} + \bar{C}_\bbf .
	\end{align*}
	The remaining statements follow as above.
\end{proof}
From Lemma \ref{lem:max-level}, we obtain a bound on the maximum level that depends additionally on the mode-wise support sizes. However, according to Section \ref{sec:apply-initial-value}, the support of an iterate depend on the maximum level of the last iterate, as well as on the $\cAs$-norm, which depends on the maximum level as well. Therefore we estimate the three quantities together step by step. An estimate for the support of the result of $\apply$ is given by Lemma \ref{lem:apply-quantity-estimates} for $\bB_1$ and by Lemma \ref{lem:initial-operator-trunc-estimates} and Lemma \ref{lem:initial-operator-trans-trunc-estimates} for $\bB_2$.
We now aim to use these estimates as well as Lemma \ref{lem:max-level} to derive a bound for the support and the maximum level of the iterates $\bwkj$. In particular, we are interested in the dependence on $d$.
In the statement of the corresponding lemma, we use the notation
\begin{align*}
	C_{\bu,\bbf} = \max \Big\{ \Big(\sum\limits_{i=1}^d \norm{\piti(\bu)}_{\cAs}\Big)^{\frac{1}{s}}, \Big(\sum\limits_{i=1}^d \norm{\piti(\bbf_1)}_{\cAs}\Big)^{\frac{1}{s}} , \Big(\sum\limits_{i=1}^d \norm{\pi^{(i)}(\bbf_2)}_{\cAs}\Big)^{\frac{1}{s}}   \Big\} .
\end{align*}
\begin{lemma}\label{lem:iteration-supp-As-lvl}
	Let $\bw_{k,j}$ be defined by Algorithm \ref{alg:lowrankrichardson}. Under the assumptions of Theorem \ref{thm:complexity}, one has
	\begin{align}
		\sum\limits_{i=1}^d \norm{\piti(\bwkj)}_{\cAs} &\leq \bar{C}_1 (\tilde{C} d)^{2j} d^{p_0} \big((\ln d)^2 \abs{\ln \eta_{k,j}} + (\ln d)^3\big)^{2js} C_{\bu,\bfs}^s, \label{eq:iteration-As-estimate} \\
		\sum\limits_{i=1}^d \# \supp(\piti(\bwkj)) &\leq \bar{C}_2 (\tilde{C}d)^{\frac{2j}{s}} d^{p_1} \big((\ln d)^2 \abs{\ln \eta_{k,j}} + (\ln d)^3\big)^{2j} \eta_{k,j}^{-\frac{1}{s}} C_{\bu,\bfs}, \label{eq:iteration-supp-estimate} \\
		L(\bw_{k,j}) &\leq \bar{C}_3 ((\ln d)^2 \abs{\ln \eta_{k,j}} + (\ln d)^3) . \label{eq:iteration-rank-estimate}
	\end{align}
	where $p_0 = 1 + \max\{1,s\}$, $p_1 = \max\{2+2s^{-1},1+3s^{-1}\}$ and $\bar{C}_1,\bar{C}_2,\bar{C}_3,\tilde{C}$ are $d$-independent. 
\end{lemma}
\begin{proof}
	Let $\varepsilon_k = 2^{-k} \varepsilon_0$. By definition of the algorithm and \eqref{eq:recompress-coarsen-supp}, we have
	\begin{align*}
		\sum\limits_{i=1}^d \norm{\piti(\bw_{k,0})}_{\cAs} &\leq C_1 d^{p_0} \sum\limits_{i=1}^d \norm{\piti(\bu)}_{\cAs}, \\
		\sum\limits_{i=1}^d \# \supp(\piti(\bw_{k,0})) &\leq C_2 d^{1+s^{-1}} (\eta_{k,0})^{-\frac{1}{s}} \Big(\sum\limits_{i=1}^d \norm{\piti(\bu)}_{\cAs}\Big)^{\frac{1}{s}} .
	\end{align*}
	where we take into account the definition of $\kappa_{\Cor}$ and $\kappa_{\Pro}$ as well as the $d$-independence of $\rho$, so the constants $C_1,C_2$ are independent of $d$. Additionally by definition of the algorithm, we know that $\bw_{0,0} = 0$. Hence, \eqref{eq:iteration-As-estimate} and \eqref{eq:iteration-supp-estimate} hold for $\bw_{k,0}$ for each $k \in \N_0$.
	
	The proof is structured as follows. In the first step we show that the statement is true for $\bw_{k,j+1}$ if it is true for $\bw_{k,j}$. We then show that estimate \eqref{eq:iteration-rank-estimate} is true for $\bw_{k+1,0}$ if it is true for $\bw_{k,I}$.
	
	Let the statements \eqref{eq:iteration-As-estimate}, \eqref{eq:iteration-supp-estimate} and \eqref{eq:iteration-rank-estimate} be true for $\bwkj$. We refer here once to Assumption \ref{ass:dim} to recall which parameters are independent of the dimension $d$. By Remark \ref{rem:apply-sum-contr}, we have
	\begin{align*}
		\sum\limits_{i=1}^d \norm{\piti(\apply_{\bB_1^\intercal}(\bzkj^{(1)};\tfrac{1}{8}\eta_{k,j}))}_{\cAs} &\leq C_3 d \sum\limits_{i=1}^d \norm{\piti(\apply_{\bB_1}(\bwkj;\tfrac{1}{8\norm{\bB}}\eta_{k,j}))}_{\cAs} \\
		&\leq (C_3 d)^2 \sum\limits_{i=1}^d \norm{\piti(\bwkj)}_{\cAs} .
	\end{align*}
	Based on the structure of $\bB_2$ with the identity in the spatial variable, we have $L(\bzkj^{(2)}) \leq L_\rx(\bwkj)$.
	Lemma \ref{lem:initial-operator-trunc-estimates} and Lemma \ref{lem:initial-operator-trans-trunc-estimates} yield
	\[
	\sum\limits_{i=1}^d \norm{\piti(\apply_{\bB_2^\intercal}(\bzkj^{(2)};\tfrac{1}{8}\eta_{k,j}))}_{\cAs} \leq (C_4 d)^2 (L_\rx(\bwkj)+1)^{2s} \sum\limits_{i=1}^d \norm{\piti(\bwkj)}_{\cAs} .
	\]
	Combining these two results with Proposition \ref{prop:cAs-properties}\eqref{prop:cAs-triangle} we obtain
	\begin{align*}
		\sum\limits_{i=1}^d \norm{\piti(\apply(\bwkj;\tfrac{1}{2}\eta_{k,j}))}_{\cAs} \leq 2^{1+s} (C_5 d)^2 (L_\rx(\bwkj)+1)^{2s} \sum\limits_{i=1}^d \norm{\piti(\bwkj)}_{\cAs}.
	\end{align*}
	Moreover, by definition of the right-hand side \eqref{eq:rhs-normal-eq} as well as Proposition \ref{prop:cAs-properties}\eqref{prop:cAs-triangle}, Remark \ref{rem:apply-sum-contr}, Lemma \ref{lem:initial-operator-trunc-estimates}, Lemma \ref{lem:initial-operator-trans-trunc-estimates} and Assumption \ref{ass:rhs}\eqref{ass:rhsapprox}, we have
	\begin{align*}
		\sum\limits_{i=1}^d \norm{\piti(\rhs(\tfrac{\eta_{k,j}}{2}))} &\leq  C_6 d \sum\limits_{i=1}^d \norm{\piti(\bbf_1)}_{\cAs} + C_7 d ((\ln d)^2 \abs{\ln \eta_{k,j}} + (\ln d)^3)^s  \sum\limits_{i=1}^d \norm{\pi^{(i)}(\bbf_2)}_{\cAs} \\
		&\leq C_8 d ((\ln d)^2 \abs{\ln \eta_{k,j}} + (\ln d)^3)^s C_{\bu,\bbf}^s .
	\end{align*}
	By the definition of $\bw_{k,j+1}$, the previous estimates and the assumptions on $\bwkj$, we arrive at
	\begin{align*}
		\sum\limits_{i=1}^d \norm{\piti(\bw_{k,j+1})}_{\cAs} &\leq \bar{C}_1 C_{9}  \bar{C}_3^{2s} (C_5 d)^2  \big((\ln d)^2 \abs{\ln \eta_{k,j}} + (\ln d)^3\big)^{2s(j+1)} (\tilde{C} d)^{2j} d^{p_0}  C_{\bu,\bfs}^s \\
		&\leq \bar{C}_1  \big((\ln d)^2 \abs{\ln \eta_{k,j+1}} + (\ln d)^3\big)^{2s(j+1)} (\tilde{C} d)^{2(j+1)} d^{p_0}  C_{\bu,\bfs}^s,
	\end{align*}
	where we assume that $\tilde{C}$ is chosen such that $C_{9} \bar{C}_{3}^{2s} C_5 \leq \tilde{C}$ and that without loss of generality $\eta_{k,j} < 1$ for each $k,j$. If this is not the case, we can instead solve a scaled problem. The parameters $\bar{C}_1$ and $\tilde{C}$ are independent of $d$.
	
	We estimate the support sizes of iterates. Here we can proceed in a similar manner as for the $\cAs$-norm. We set $\mathbf{\bar w}_{k,j+1} = \bwkj - \omega \br_{k,j}$. In the recompression and coarsening steps, the support sizes of $\bw_{k,j+1}$ cannot increase. Hence it is sufficient to show the support bound for $\mathbf{\bar w}_{k,j+1}$. Using the definition of $\apply$ from \eqref{eq:apply-normal-eq}, as well as Lemma \ref{lem:apply-quantity-estimates}, Remark \ref{rem:apply-sum-contr}, Lemma \ref{lem:initial-operator-trunc-estimates} and Lemma \ref{lem:initial-operator-trans-trunc-estimates}, we obtain
	\begin{align*}
		\sum\limits_{i=1}^d \# \supp(\piti(\apply(\bwkj;\tfrac{\eta_{k,j}}{2}))) \leq  C_{10} d^{1 +\frac{1}{s}} (L(\bwkj)+1)^2 \eta_{k,j}^{-\frac{1}{s}} \Big(\sum\limits_{i=1}^d \norm{\piti(\bwkj)}_{\cAs}\Big)^{\frac{1}{s}}
	\end{align*}
	By Assumption \ref{ass:rhs}\eqref{ass:rhsapprox}, we have
	\begin{align*}
		\sum\limits_{i=1}^d \# \supp(\piti(\rhs(\tfrac{1}{2}\eta_{k,j}))) \leq C_{11} d^{1+s^{-1}} ((\ln d)^2 \abs{\ln \eta_{k,j}} + (\ln d)^3) \eta_{k,j}^{-\frac{1}{s}} C_{\bu,\bbf}
	\end{align*}
	Using our knowledge on $\bwkj$ yields
	\begin{align*}
		\sum\limits_{i=1}^d \# \supp(\piti(\mathbf{\bar w}_{k,j+1})) &\leq  \bar{C}_2 (\tilde{C}d)^{\frac{2j}{s}} d^{p_1} \big((\ln d)^2 \abs{\ln \eta_{k,j}} 
		+ (\ln d)^3\big)^{2j} \eta_{k,j}^{-\frac{1}{s}} C_{\bu,\bfs} \\ 
		&\quad +  C_{12} \bar{C}_3^2 \bar{C}_1^{\frac{1}{s}} (\tilde{C}d)^{\frac{2j}{s}} d^{p_1} \big((\ln d)^2 \abs{\ln \eta_{k,j}} + (\ln d)^3\big)^{2(j+1)} \eta_{k,j}^{-\frac{1}{s}} C_{\bu,\bbf} \\
		&\leq \bar{C}_2 (\tilde{C}d)^{\frac{2(j+1)}{s}} d^{p_1} \big((\ln d)^2 \abs{\ln \eta_{k,j+1}} + (\ln d)^3\big)^{2(j+1)} \eta_{k,j+1}^{-\frac{1}{s}} C_{\bu,\bfs} ,
	\end{align*}
	where we take $\tilde{C}$ sufficiently large depending on the other constants and again use $\eta_{k,j} < 1$.
	
	We now estimate the maximum ranks of iterates. By applying Lemma \ref{lem:max-level} at $\mathbf{\bar w}_{k,j}$ as well as using Proposition \ref{prop:higher-reg} and the requirement $\beta_2 > 0$ in Algorithm \ref{alg:lowrankrichardson}, one has the estimate
	\begin{align*}
		L_{k,j+1} \leq  \tau^{-1} \log_2 \Big(C_{13} d^{\frac{p_1}{2}} \eta_{k,j}^{-1} \gamma^{kI+j} \bar{C}_{\bbf} \eta_{k,j+1}^{-\frac{1}{2s}} (\tilde{C}d)^{\frac{(j+1)}{s}} \big((\ln d)^2 \abs{\ln \eta_{k,j+1}} + (\ln d)^3\big)^{(j+1)} C_{\bu,\bbf}^{\frac{1}{2}} \Big) ,
	\end{align*}
	where $L_{k,j+1} = L(\bw_{k,j+1})$. At this point we note that $C_{\bu,\bbf}$ may depend on $d$, but $C_{\bu,\bbf} \leq d^{\frac{1}{s}} \hat{C}_{\bu,\bbf}$ with $\hat{C}_{\bu,\bbf} = \max_i\{\norm{\piti(\bu)}_{\cA}^{\frac{1}{s}}, \norm{\piti(\bbf_1)}_{\cAs}^{\frac{1}{s}},\norm{\pi^{(i)}(\bbf_2)}_{\cAs}^{\frac{1}{s}}\}$, which is independent of $d$ by Assumption \ref{ass:dim}. Additionally by the definition of $\eta_{k,j}$, one has
	\begin{align*}
		\log_2 \gamma^{k} \leq (\abs{\log_2 \eta_{k,j+1}} + j \abs{\log_2 \rho}  + \abs{\log_2 (\rho^2 \varepsilon_0)}) \log_2(\gamma) .
	\end{align*}
	 We know that $\gamma$ only depends on $d$ by $\check{C}_\rx$, which grows at most linearly in $d$. Hence, there exist constants $c, C_{\gamma}$ with
	\begin{align}\label{eq:j-itter-upper-bound}
		 j+1 \leq I \leq c \ln d, \quad \gamma \leq C_{\gamma} d^2 .
	\end{align}
	Combining this with $\log_2(x) \leq x$ yields
	\begin{align*}
		\log_2\Big(\big((\ln d)^2 \abs{\ln \eta_{k,j+1}} + (\ln d)^3\big)^{(j+1)}\Big) \leq c \big((\ln d)^2 \abs{\ln \eta_{k,j+1}} + (\ln d)^3\big) .
	\end{align*}
	Defining $C_{14}(d) = C_{13} \bar{C}_{\bbf} C^{\frac{1}{2}}_{\bu,\bbf} \rho$ and combining the previous estimates, we obtain
	\begin{align*}
		\tau L_{k,j+1} &\leq \log_2( C_{14}(d)) + \tfrac{j+1}{s} \log_2 (\tilde{C}d) + \tfrac{p_1}{2} \log_2 d + (1+\tfrac{j}{2s}) \abs{\log_2 \eta_{k,j+1}} + I \log_2 \gamma^k \\ 
		&\qquad + j \log_2 \gamma + \log_2\Big(\big((\ln d)^2 \abs{\ln \eta_{k,j+1}} + (\ln d)^3\big)^{(j+1)}\Big) \\
		&\leq  \log_2 (C_{14}(d)) + \tfrac{j+1}{s} \log_2 (\tilde{C}d) + \tfrac{p_1}{2} \log_2 d + (1+\tfrac{j}{2s} + I \log_2 \gamma) \abs{\log_2 \eta_{k,j+1}} \\
		&\qquad + (\log_2 \gamma )(j + jI \abs{\log_2(\rho)} + I \abs{\log_2(\rho^2 \varepsilon_0)}) + c \big((\ln d)^2 \abs{\ln \eta_{k,j+1}} + (\ln d)^3\big)\\
		&\leq \tau \bar{C}_3 \big((\ln d)^2 \abs{\ln \eta_{k,j+1}} + (\ln d)^3\big)
	\end{align*}
	for sufficiently large $\bar{C}_3$, where we used that without loss of generality $\ln d > 1$.
	Hence, we have shown that the statements are true for $\bw_{k,j+1}$ if they are true for $\bwkj$.
	
	In the last step we have to show that \eqref{eq:iteration-rank-estimate} is true for $\bw_{k+1,0}$ if it is true for $\bw_{k,I}$. By Lemma \ref{lem:max-level}, Proposition \ref{prop:higher-reg} and $\bw_{k+1,0} = \bu_{k+1}$, we have
	\begin{align*}
		L_{k+1,0} &\leq  \tau^{-1} \log_2 \Big(C_{15} d^{\frac{p_1}{2}} (\kappa_3 2^{-k+1} \varepsilon_0)^{-1} \gamma^{kI+I} \bar{C}_{\bbf} \eta_{k,I}^{-\frac{1}{2s}} (\tilde{C}d)^{\frac{I}{s}} \big((\ln d)^2 \abs{\ln \eta_{k,I}} + (\ln d)^3\big)^{I} C_{\bu,\bbf}^{\frac{1}{2}} \Big)  \\
		&=  \tau^{-1} \log_2 \Big( C_{16}(d) d^{\frac{p_1}{2}} \eta_{k+1,0}^{-1-\frac{1}{2s}} \rho^{\frac{I}{2s}} \gamma^{kI+I} (\tilde{C}d)^{\frac{I}{s}} \big((\ln d)^2 \abs{\ln \eta_{k,I}} + (\ln d)^3\big)^{I} \Big)
	\end{align*}
	with $C_{16}(d) = C_{15}  \kappa_3^{-1} 2^{-\frac{1}{2s}}\rho \bar{C}_{\bbf} C_{\bu,\bbf}^{\frac{1}{2}}$. Additionally we have by $I \leq c \ln d$ the estimate
	\begin{align*}
		\log_2\Big(\big((\ln d)^2 \abs{\ln \eta_{k,I}} + (\ln d)^3\big)^{I} \Big) \leq C_{17} \big((\ln d)^2 \abs{\ln \eta_{k+1,0}} + (\ln d)^3\big) .
	\end{align*}
	Proceeding as before, we find
	\begin{align*}
		L(\bw_{k+1,0}) &\leq \bar{C}_3 ((\ln d)^2 \abs{\ln \eta_{k+1,0}} + (\ln d)^3),
	\end{align*}
	where we again assume that $\bar{C}_3$ is chosen sufficiently large in dependence on the other constants.
\end{proof}
In this lemma we have derived bounds on the maximum temporal and spatial wavelet levels, whereas for the estimates on support sizes and $\cAs$-quasinorms we only need a bound on the maximum \emph{spatial} levels. However, the maximum overall level bound is also required for rank estimates.

\subsubsection{Hierarchical ranks of iterates}
For making use of the rank bounds provided by Lemma \ref{lem:apply-quantity-estimates}, Lemma \ref{lem:initial-operator-trunc-estimates} and Lemma \ref{lem:initial-operator-trans-trunc-estimates}, as a next step we estimate the quantities appearing in these bounds, which are 
\[
	J(\tfrac{1}{8\norm{\bB}}\eta_{k,j};\bwkj), \quad J_0(\tfrac{1}{8\norm{\bB}}\eta_{k,j};\bwkj), \quad  J(\tfrac{1}{8}\eta_{k,j},\bzkj^{(1)}), \quad J^\ad_{0}(\tfrac{1}{8}\eta_{k,j},\bzkj^{(2)})
\]
and the quantities related to scaling matrices
\[
\norm{\hmcX(\tfrac{1}{8\norm{\bB}}\eta_{k,j};\bw_{k,j})}_{\ell_\infty}, \quad \norm{\hat{\sfm}_{\cX,0}(\tfrac{1}{8\norm{\bB}}\eta_{k,j};\bw_{k,j})}_{\ell_\infty}, \quad
\norm{\hmcY(\tfrac{1}{8\norm{\bB}}\eta_{k,j};\bw_{k,j})}_{\ell_\infty}, \\
\]
as well as
\[
\norm{\hmcX(\tfrac{1}{8}\eta_{k,j};\bzkj^{(1)})}_{\ell_\infty}, \quad
\norm{\hat{\sfm}^\ad_{\cX,0}(\tfrac{1}{8}\eta_{k,j};\bzkj^{(2)})}_{\ell_\infty}, \quad
\norm{\hmcY(\tfrac{1}{8}\eta_{k,j};\bzkj^{(1)})}_{\ell_\infty} .
\]
The rank estimates depend on the maximum wavelet levels of activated basis indices, where a bound is given by Lemma \ref{lem:iteration-supp-As-lvl}.
Moreover, as shown in \eqref{eq:mcy-inf-bound} and \eqref{eq:mcx-inf-bound} in Section \ref{sec:adaptive-operator-comb} and in \eqref{eq:m-hat-initial-ell-inf} and \eqref{eq:m-hat-initial-trans-ell-inf} in Section \ref{sec:apply-initial-value}, the scaling matrix based quantities can be bounded using the maximum wavelet level among the basis indices that are active in the iterates.

\begin{lemma}\label{lem:iteration-ranks}
	Let $\bwkj$ be defined by Algorithm \ref{alg:lowrankrichardson}. Under the assumptions of Theorem \ref{thm:complexity}, one has
	\begin{align*}
		\rank_\infty(\bwkj)\leq C d^{p_2} (\ln d)^{b+4} \left((\ln d)^2 \abs{\ln \eta_{k,0}} + (\ln d)^3\right)^{8j} (1 + \abs{\ln \eta_{k,0}})^{b+4} 
	\end{align*}
	with $b = \max\{b_\bu, b_\bbf\}, p_2 = (\ln \bar{C} + 8 \ln(1 + c \abs{\ln \rho}))c$, and where $C, \bar{C}>0$ are $d$-independent constants.
\end{lemma}
\begin{proof}
	The basic approach of this proof is similar to \cite[Section 6.5]{bachmayr2014adaptive}. The differences is based on the new scaling matrix $\bDX$, the additional factor $J$ when applying the temporal operator and the wavelet level dependence when applying the initial value operator.
	
	First of all, we derive estimates for the scaling matrix based quantities of the non-transposed operator.
	Both of them depend on the maximum active level of $\bwkj$ as well as on the $\cAs$-norm of the contractions.
	Combining \eqref{eq:iteration-As-estimate} and \eqref{eq:j-itter-upper-bound} yields
	\begin{align*}
		\ln\Big(\sum\limits_{i=1}^d\norm{\piti(\bwkj)}_{\cAs}\Big) \leq C_1 ((\ln d)^2 \abs{\ln \eta_{k,j}} + (\ln d)^3) .
	\end{align*}
	Therefore, by \eqref{eq:mcy-inf-bound}, \eqref{eq:mcx-inf-bound} and \eqref{eq:m-hat-initial-ell-inf}, we have
	\begin{align*}
		\norm{\hmcY\big(\tfrac{1}{8\norm{\bB}}\eta_{k,j};\bwkj\big)}_{\ell_\infty} &\leq C_2 ((\ln d)^2 \abs{\ln \eta_{k,j}} + (\ln d)^3), \\
		\norm{\hmcX\big(\tfrac{1}{8\norm{\bB}}\eta_{k,j};\bwkj\big)}_{\ell_\infty} &\leq C_3 \left((\ln d)^2 \abs{\ln \eta_{k,j}} + (\ln d)^3 \right)^2 \\
		\norm{\hat{\sfm}_{\cX,0}\big(\tfrac{1}{8\norm{\bB}}\eta_{k,j};\bwkj\big)}_{\ell_\infty} &\leq C_4 \left((\ln d)^2 \abs{\ln \eta_{k,j}} + (\ln d)^3 \right)^2 
	\end{align*}
	and by \eqref{eq:J-eta-bound} and \eqref{eq:J0-estimate} the estimates
	\begin{align*}
		J\big(\tfrac{1}{8\norm{\bB}}\eta_{k,j};\bwkj \big) \leq C_5 ((\ln d)^2 \abs{\ln \eta_{k,j}} + (\ln d)^3), \\
		J_0\big(\tfrac{1}{8\norm{\bB}}\eta_{k,j};\bwkj \big) \leq C_6 ((\ln d)^2 \abs{\ln \eta_{k,j}} + (\ln d)^3)
	\end{align*}
	Additionally we have to estimate the quantities with respect to the intermediate results $\bzkj^{(1)}$ and $\bzkj^{(2)}$. We obtain
	\begin{align*}
		\norm{\hmcY(\tfrac{1}{8} \eta_{k,j};\bzkj^{(1)})}_{\ell_\infty} &\leq C_7 \Big[ 1 + L_\rx(\bzkj^{(1)}) + \abs{\ln \tfrac{\eta_{k,j}}{8}} + \ln\Big(\sum\limits_{i=1}^d \norm{\piti(\bzkj)}_{\cAs} \Big)  \Big] \\
		&\leq C_{8} \Big[ 1 + L_\rx(\bwkj) + \abs{\ln \eta_{k,j}} + \ln d  +\ln\Big(\sum\limits_{i=1}^d \norm{\piti(\bwkj)}_{\cAs} \Big)  \Big] \\
		&\leq C_{9} ((\ln d)^2 \abs{\ln \eta_{k,j}} + (\ln d)^3),
	\end{align*}
	where we again used the estimation \eqref{eq:mcy-inf-bound} as well as Remark \ref{rem:apply-sum-contr} and the level decay property for $L_{\rx}(\bzkj^{(1)})$ in the second line. Analogously we have
	\begin{align*}
		\norm{\hmcX(\tfrac{1}{8} \eta_{k,j};\bzkj^{(1)})}_{\ell_\infty} &\leq C_{10} \left((\ln d)^2 \abs{\ln \eta_{k,j}} + (\ln d)^3 \right)^2, \\
		\norm{\hat{\sfm}^\ad_{\cX,0}(\tfrac{1}{8} \eta_{k,j};\bzkj^{(2)})}_{\ell_\infty} &\leq C_{11} \left((\ln d)^2 \abs{\ln \eta_{k,j}} + (\ln d)^3 \right)^2, \\
		J\left(\tfrac{1}{8} \eta_{k,j}; \bzkj^{(1)}\right) &\leq C_{12} ((\ln d)^2 \abs{\ln \eta_{k,j}} + (\ln d)^3) , \\
		J^\ad_{0}\left(\tfrac{1}{8} \eta_{k,j}; \bzkj^{(2)}\right) &\leq C_{13} ((\ln d)^2 \abs{\ln \eta_{k,j}} + (\ln d)^3) ,
	\end{align*}
	where we use $L(\bzkj^{(2)}) \leq L_{\rx}(\bwkj)$.
	Furthermore we have to estimate the rank of $\bfs_{k,j} = \bfs_{k,j}^{(1)} + \bfs_{k,j}^{(2)}$. We have $\mathbf{\tilde f}_{k,j}^{(1)} = \rhs_{\bfs_1}(\tfrac{\eta_{k,j}}{8\norm{\bB}})$ and $\bfs_{k,j}^{(1)} = \apply_{\bB_1^{\intercal}}(\mathbf{\tilde f}^{(1)}_{k,j},\tfrac{\eta_{k,j}}{8})$. 
	Without loss of generality, we can apply Lemma \ref{lem:max-level} on $\mathbf{\tilde f}^{(1)}_{k,j}$, because the $\rhs_{\bbf_1}$ routine can be realized with a coarsening step in the end similar to Remark \ref{rem:f2-lvl-estimate-realization}. Combining Lemma \ref{lem:max-level} with Assumption \ref{ass:rhs}\eqref{ass:rhsapprox} and Assumption \ref{ass:rhs}\eqref{ass:higher-regularity-f}, yields
	\begin{align*}
		L(\mathbf{\tilde f}^{(1)}_{k,j}) \leq C_{14} (\ln d + \abs{\ln \eta_{k,j}}) .
	\end{align*}
	Using again Assumption \ref{ass:rhs}\eqref{ass:rhsapprox}, we obtain
	\begin{align*}
		\norm{\hmcY(\tfrac{1}{8} \eta_{k,j};\mathbf{\tilde f}_{k,j}^{(1)})}_{\ell_\infty},J\left(\tfrac{1}{8} \eta_{k,j}; \mathbf{\tilde f}_{k,j}^{(1)}\right) &\leq C_{16} (\ln d + \abs{\ln \eta_{k,j}}) , \\
		\norm{\hmcX(\tfrac{1}{8} \eta_{k,j};\mathbf{\tilde f}_{k,j}^{(1)})}_{\ell_\infty} &\leq C_{17} (\ln d + \abs{\ln \eta_{k,j}})^2 .
	\end{align*}
	For $\mathbf{\tilde f}_{k,j}^{(2)} = \rhs_{\bfs_2}(\tfrac{\eta_{k,j}}{8\norm{\bB}})$and $\bfs_{k,j}^{(2)} = \apply_{\bB_2^{\intercal}}(\mathbf{\tilde f}^{(2)}_{k,j},\tfrac{\eta_{k,j}}{8})$, we get in the same way
	\begin{align*}
			\norm{\hat{\sfm}^\ad_{\cX,0}(\tfrac{1}{8} \eta_{k,j};\mathbf{\tilde f}_{k,j}^{(2)})}_{\ell_\infty}  &\leq C_{18} (\ln d + \abs{\ln \eta_{k,j}})^2 , \\
			J^\ad_{0}\left(\tfrac{1}{8} \eta_{k,j}; \mathbf{\tilde f}_{k,j}^{(2)}\right) &\leq C_{19} (\ln d + \abs{\ln \eta_{k,j}}) .
	\end{align*}
	Taking Assumption \ref{ass:rhs}\eqref{ass:rhsapprox} into account as well as Lemma \ref{lem:apply-quantity-estimates} and Lemma \ref{lem:initial-operator-trans-trunc-estimates}, we arrive at
	\begin{align*}
		\rank_\infty(\bfs_{k,j})  &\leq C_{20} (\ln d + \abs{\ln \eta_{k,j}})^4   \big(\rank_\infty(\mathbf{\tilde f}^{(1)}_{k,j}) + \abs{\rank(\mathbf{\tilde f}^{(2)}_{k,j})}_{\infty} \big) \\
		&\leq  C_{21}(1 + \abs{\ln \eta_{k,j}})^{b_{\bbf}}  (\ln d + \abs{\ln \eta_{k,j}})^4 .
	\end{align*}
	Hence, by applying Lemma \ref{lem:apply-quantity-estimates} two times as well as Lemma \ref{lem:initial-operator-trunc-estimates} and Lemma \ref{lem:initial-operator-trans-trunc-estimates}, we arrive at
	\begin{equation}
		\rank_\infty(\bw_{k,j+1}) \leq C_{22} ((\ln d)^2 \abs{\ln \eta_{k,j}} + (\ln d)^3)^8 \rank_\infty(\bwkj) + C_{17} (\ln d + \abs{\ln \eta_{k,j}})^{b_{\bbf}+4} . \label{eq:rank-recursion}
	\end{equation}
	 Additionally by \eqref{eq:complexity_ranknorm} in Theorem \ref{thm:complexity} and Assumption \ref{ass:approximability}\eqref{ass:uapprox} we have
	\begin{align*}
		\rank_\infty(\bw_{k,0}) \leq (d_{\bu}^{-1} \ln[ 2 (\alpha \kappa_1)^{-1} \rho_{\gamma_\bu} \norm{\bu}_{\cA(\gamma_\bu)} \eta_{k,0}^{-1} \rho ] )^{b_\bu} \leq C(\bu) (\abs{\ln \eta_{k,0}} + \ln d)^{b_\bu},
	\end{align*}
	where we used $\kappa_1^{-1} \lesssim d$. We recall $j \leq I \leq c \ln d$, which yields
	\begin{equation}\label{eq:etakj-etak0-estimate}
		\abs{\ln \eta_{k,j}} \leq \abs{\ln \eta_{k,0}} + c \ln d \abs{\ln \rho}, \quad j=0,\dots,I.
	\end{equation}
	 By repeated insertion of \eqref{eq:rank-recursion}, we obtain
	 \begin{align*}
		\rank_\infty(\bwkj) &\leq C_{23} d^{c\ln C_{18}} ((\ln d)^2 (\abs{\ln \eta_{k,0}} + c \ln d \abs{\ln \rho}) + (\ln d)^3)^{8j} \\
		&\quad \times (\abs{\ln \eta_{k,0}} + c \ln d \abs{\ln \rho} + \ln d)^{b+4} \\
		&\leq C_{24} d^{p_2} (\ln d)^{b+4} \left((\ln d)^2 \abs{\ln \eta_{k,0}} + (\ln d)^3\right)^{8j} (1 + \abs{\ln \eta_{k,0}})^{b+4} ,
	 \end{align*}
 	where $b =  \max\{b_\bu, b_\mathbf{f}\}$.
\end{proof}

\subsubsection{Complexity of Algorithm \ref{alg:lowrankrichardson}}
With the above preparations, we are in a position to prove the complexity bound in Theorem \ref{thm:complexity}.

\begin{proof}[Proof of Theorem \ref{thm:complexity}]
	The estimates \eqref{eq:complexity_rank} and \eqref{eq:complexity_ranknorm} follow directly from \eqref{eq:recompress-coarsen-rank}, whereas \eqref{eq:complexity_supp} and \eqref{eq:complexity_sparsitynorm} are an intermediate consequence of \eqref{eq:recompress-coarsen-supp}. 
	
	We now turn to the proof of the estimate \eqref{eq:complexity_totalops} for the number of required operations.
	By Assumption \ref{ass:rhs}\eqref{ass:rhsops}, Lemma \ref{lem:apply-num-ops}, Lemma \ref{lem:apply-initial-value-num-ops}, Lemma \ref{lem:apply-initial-value-trans-num-ops} and Remark \ref{rem:hsvd-complexity} the complexity of each inner iteration step is dominated by the hierarchical singular value decomposition which is required for the $\recompress$ and $\coarsen$ routines. Hence, by Remark \ref{rem:hsvd-complexity} and \eqref{eq:iteration-supp-estimate} from Lemma \ref{lem:iteration-supp-As-lvl} for each $k$ and $j$ the complexity is bounded by
	\begin{align*}
		C_1 d^q \rank^4_\infty(\bwkj) d^{p_1} (\tilde{C}d)^{\frac{2j}{s}} C_{\bu,\bbf} \eta_{k,j}^{-\frac{1}{s}} \big((\ln d)^2 \abs{\ln \eta_{k,j}} + (\ln d)^3\big)^{2j} ,
	\end{align*}
	where $d^q$ corresponds to the growth of $C^\text{{\rm ops}}_\bbf(d)$, which by Assumption \ref{ass:dim} is at most polynomial. 
	Using \eqref{eq:etakj-etak0-estimate}, we obtain
	\begin{equation*}
		\big((\ln d)^2 \abs{\ln \eta_{k,j}} + (\ln d)^3\big)^{2j} \leq d^{p_3} \big((\ln d)^2 \abs{\ln \eta_{k,0}} + (\ln d)^3\big)^{2j} 
	\end{equation*}
	with $p_3 = 2c \ln(1+c \abs{\ln \rho})$.
	Consequently, the complexity of the outer loop $k$ is bounded by
	\begin{align*}
		C_2 d^{q+p_1+p_3} \rank_\infty^4(\bw_{k,I})  (\tilde{C}d)^{\frac{2(I+1)}{s}} C_{\bu,\bbf} \eta_{k,I}^{-\frac{1}{s}}   \big((\ln d)^2 \abs{\ln \eta_{k,0}} + (\ln d)^3\big)^{2I}  .
	\end{align*} 
	Furthermore, the total work to arrive at $\bu_k$ is bounded by
	\begin{align}\label{eq:total-work-uk}
		C_3 d^{q+p_1+p_3} \rank_\infty^4(\bw_{k-1,I}) (\tilde{C}d)^{\frac{2c \ln d+2}{s}} d^{\frac{1}{s}} \hat{C}_{\bu,\bbf} \eta_{k-1,I}^{-\frac{1}{s}} \big((\ln d)^2 \abs{\ln \eta_{k-1,0}} + (\ln d)^3\big)^{2I} .
	\end{align}
	In the next step we express the bound in terms of the tolerance $\varepsilon_k$. It holds $\eta_{k-1,0} = 2 \rho \varepsilon_k$ and $\eta_{k-1,I} = 2\rho^{I+1} \varepsilon_k$. Using $\rho \in (0,1)$ and $I \leq c \ln d$, yields $\eta_{k-1,I}^{-\frac{1}{s}} \leq (2\rho)^{-\frac{1}{s}} d^{cs^{-1} \abs{\ln \rho}} \varepsilon_k^{-\frac{1}{s}}$. 
	Therefore, we obtain
	\begin{equation}
		\big((\ln d)^2 \abs{\ln \eta_{k-1,0}} + (\ln d)^3\big)^{I} \leq (\ln d)^{c \ln(1+\abs{\ln(2\rho)}) + 3c\ln d} (1 + \abs{\ln \varepsilon_k})^{c \ln d} .
	\end{equation}
	Combining this with Lemma \ref{lem:iteration-ranks} yields
	\begin{align*}
		\rank_\infty(\bw_{k-1,I}) &\leq C_4 d^{p_2} (\ln d)^{b+4} \left((\ln d)^2 \abs{\ln \eta_{k-1,0}} + (\ln d)^3\right)^{8I} (1 + \abs{\ln \eta_{k-1,0}})^{b+4} \\
		&\leq C_5 d^{p_2} (\ln d)^{b + 4 + 8c \ln(1+ \abs{\ln 2\rho}) + 24 c \ln d} (1 + \abs{\ln \varepsilon_k})^{b + 4 + 8c \ln d} .
	\end{align*}
	By inserting this bound in \eqref{eq:total-work-uk} for the ranks, we arrive at
	\begin{align*}
		\operatorname{flops}(\bu_\varepsilon) \leq C_6 d^{p_4} (\ln d)^{4b + 16 + 34c \ln(1+ \abs{\ln 2\rho})} (\ln d)^{102 c \ln d} d^{2c s^{-1} \ln d} \varepsilon_k^{-\frac{1}{s}} (1 + \abs{\ln \varepsilon_k})^{4b +16+ 34c \ln d}
	\end{align*}
	with $p_4 = q + p_1 + 4p_2+p_3 + \frac{3}{s} + cs^{-1} \abs{\ln \rho} + 2cs^{-1} \ln \tilde{C}$, which shows \eqref{eq:complexity_totalops}.
\end{proof}

\section{Numerical Experiments}\label{sec:Numexp}

\subsection{Basic considerations}
In our implementation of Algorithm \ref{alg:lowrankrichardson}, all hierarchical tensor representations use the linear dimension tree
\[
 \mathbb{T}_d = \alpha^* \cup \bigcup_{i=1}^{d-1} \bigl\{ \{i\}, \{i+1,\ldots,d\}  \bigr\}	\,.
\]
As wavelet bases, we use $L_2$-orthonormal, continuously differentiable, piecewise polynomial Donovan-Geronimo-Hardin multiwavelets \cite{multiwavelets} of polynomial degree 6 and $L_2$-approximation order 7.
These wavelets satisfy the assumptions stated in \ref{sec:scaling-matrix} and 
used in Sections \ref{sec:Apply-spatial} and \ref{sec:Apply-temporal}, especially the requirement of $L_2(\Omega)$-orthonormality of the resulting spatial product basis that is crucial in view of Proposition \ref{prop:riesz}. In particular, with appropriate rescaling, we obtain $s^*$-compressibility of the one-dimensional operator $\bT_{2}$ and super-compressibility of $\bT_{\rt}$, as discussed in Remark \ref{rem:supercompressible-realization}. 

Additionally we use the technique described in \cite[Section 7.2]{bachmayr2014adaptive} for improving the practical efficiency of $\apply$. The basic idea is to systematically apply the $\recompress$ routine to intermediate results generated in $\apply$. This strategy leads to a substantial reduction of computational costs in practice.

\subsection{High-dimensional heat equation}
As a test case, we consider two versions of the heat equation
\begin{equation}\label{eq:heat-equation}
	\partial_t u - \Delta u = g\quad\text{in $\Omega = (0,1)^d$}, \qquad 
	u|_{t=0} = h,
\end{equation}
with Dirichlet boundary conditions in the spatial variables, one with $\norm{g}_{L_2} = 1$ and vanishing initial values $h=0$, the other with $\norm{h}_{L_2}=1$ and vanishing source term $g=0$. With data normalized in this manner, we consider absolute residual norms in what follows. As noted in Remark \ref{rem:laplaceT}, the hierarchical tensor representation of the operator $\bT_\rx$ in this case has a simple structure with ranks two. 

\subsubsection{Vanishing initial condition}\label{sec:results-source}
We first consider the case with vanishing initial conditions.
Note that such cases cannot directly be treated by alternative approaches based on dynamical low-rank approximation as in \cite{BEKU:21}.
For simplicity we use a function $g$ that can be written as a tensor product of single wavelet basis elements in the temporal and each spatial dimension.
The parameters in Algorithm \ref{alg:lowrankrichardson} are chosen as in our convergence analysis.

\begin{figure}
	\minipage{0.33\textwidth}
	\includegraphics{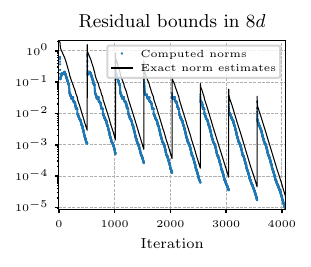}
	\endminipage\hfill
	\minipage{0.33\textwidth}
	\includegraphics{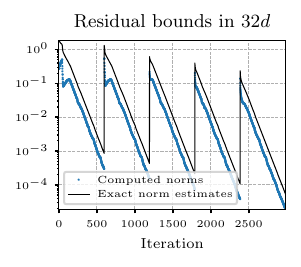}
	\endminipage\hfill
	\minipage{0.33\textwidth}
	\includegraphics{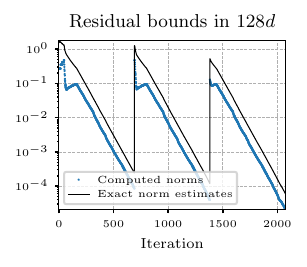}
	\endminipage
	\caption{Norms of computed error estimates and error bounds in dependence on the iteration number for the heat equation with source term, for $d=8,32,128$.}
	\label{fig:residual-sourceTerm}
\end{figure}
Figure \ref{fig:residual-sourceTerm} shows the residuals and the corresponding estimates of the error in $\cX$-norm in dependence on the iteration number for spatial dimensions $d \in \{ 8,32, 128\}$. The method behaves as expected with an increase of both estimates after each outer loop step due to the recompression and coarsening routines. Because of the $d$-dependence of the parameter $\kappa_1$, the number of inner steps required for the inner loop increases with increasing $d$.
Figure \ref{fig:supp-res-source-term} shows the sum of the one-dimensional supports as well as the maximum ranks with respect to the computed residual norm estimates, which are proportional to $\norm{\bu - \bw_{k,j}}$, over the course of the iteration. There is a pronounced preasymptotic range, and the expected rate of $s = \frac{1}{6}$ according to Theorem \ref{thm:complexity} is observed only for small residual norms, which are reached in the additional tests for $d=4$. In addition, the maximum ranks $\rank_\infty(\bwkj)$, which are shown in a semi-logarithmic plot, exhibit a logarithmic dependence on the residual norm estimate.

\begin{figure}
	\minipage{0.5\textwidth}
	\includegraphics{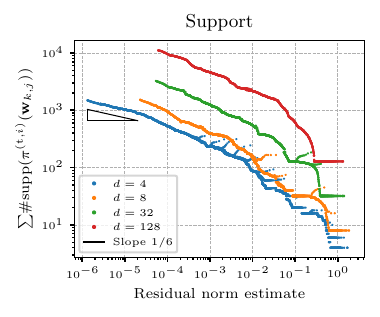}
	\endminipage
	\minipage{0.45\textwidth}
\includegraphics{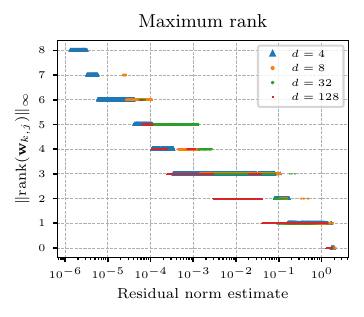}
	\endminipage\hfill
	\caption{Sum of the support and maximum rank per iteration versus the current residual norm estimate for the heat equation with source term, for $d=4,8,32,128$.}
	\label{fig:supp-res-source-term}
\end{figure}

\begin{figure}
	\minipage{0.33\textwidth}
	\includegraphics{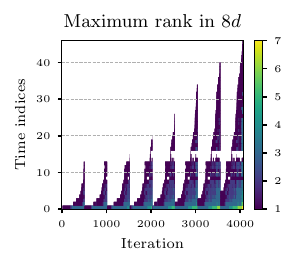}
	\endminipage\hfill
	\minipage{0.33\textwidth}
	\includegraphics{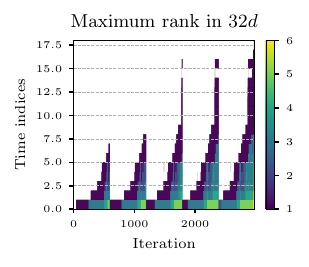}
	\endminipage\hfill
	\minipage{0.33\textwidth}
	\includegraphics{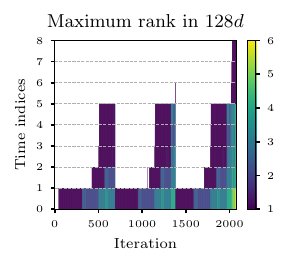}
	\endminipage
	\caption{Maximum ranks for each time index in dependence on the iteration number for the heat equation with source term, for $d=8, 32, 128$.}
	\label{fig:ranks-it-source-term}
\end{figure}

\begin{figure}
	\minipage{0.32\textwidth}
	\includegraphics[height=4.2cm]{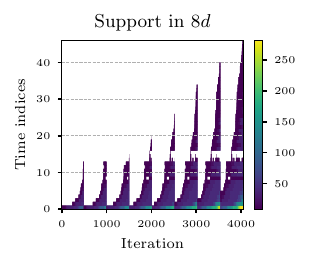}
	\endminipage\hfill
	\minipage{0.34\textwidth}
        \includegraphics[height=4.2cm]{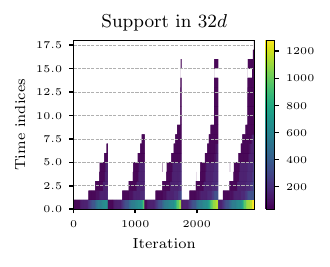}
	\endminipage\hfill
	\minipage{0.33\textwidth}
	\includegraphics[height=4.2cm]{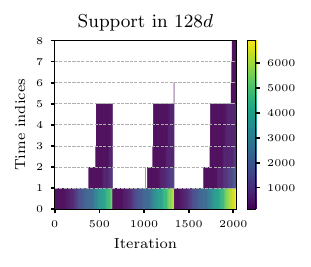}
	\endminipage
	\caption{Sum of one-dimensional support for each time index in dependence on the iteration number for the heat equation with source term, for $d=8,32,128$.}
	\label{fig:supp-it-source-term}
\end{figure}
For the same values of $d$, Figure \ref{fig:ranks-it-source-term} shows the maximum ranks of the separate low-rank approximations in each time index in dependence on the iteration count. One can discern an earlier increase of ranks for larger $d$ at comparable outer iteration numbers. 
In an analogous manner, the sum of the one-dimensional supports for each time index is shown in Figure \ref{fig:supp-it-source-term}. Altogether, we observe that only a small number of active spatial basis functions and a comparably low maximum rank is required for most of the activated temporal basis indices. 

\begin{figure}
	\includegraphics{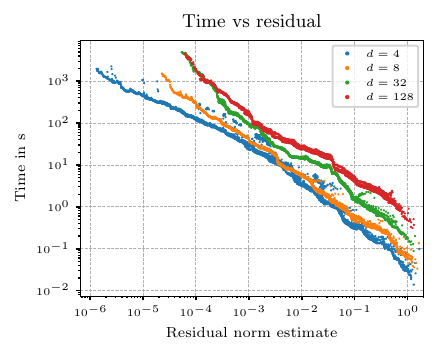}
	\caption{Time required for each inner iteration step for the heat equation with source term for $d=4,8,32, 128$ in dependence on the achieved residual estimate.}
	\label{fig:time-source-term}
\end{figure}
Figure \ref{fig:time-source-term} shows the time for each inner iteration step in dependence of the achieved residual norm estimate after this step, where the time in seconds is to be taken as a measure for the number of operations. While we observe algebraic behaviour with respect to the residual norm estimate, comparison to Figure \ref{fig:supp-res-source-term} shows that this is still in the pre-asymptotic regime. The $d$-dependence of the costs is subject to the same effects, but the results are consistent with a polynomial dependence on $d$ within the observable range.

\subsubsection{Vanishing source term}
As second test case we consider the one with vanishing source term $g = 0$ in \eqref{eq:heat-equation}. Similarly to the previous test case, for simplicity we use as initial condition a function $h$ that can be written as a tensor product of a single wavelet basis element in each spatial dimension.

This test case has higher computational costs due to a stronger increase of the temporal supports of approximations as the iteration progresses. Therefore we restrict ourselves to spatial dimensions $d=4,8,16$ in this case. In addition, we use error tolerances in $\apply$, $\coarsen$ and $\recompress$ that are larger than the ones used in the convergence analysis, which turn out to be stricer than necessary in practice. Specifically, in line 10 of Algorithm \ref{alg:lowrankrichardson} we replace $\eta_{k,j}$ by $10\eta_{k,j}$ and in line 11 by $2\eta_{k,j}$ without observing an impact on the convergence of the method.

\begin{figure}
	\minipage{0.33\textwidth}
	\includegraphics{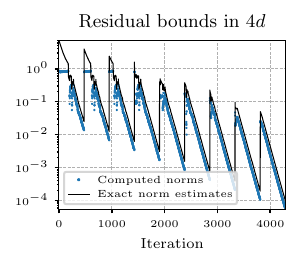}
	\endminipage\hfill
	\minipage{0.33\textwidth}
	\includegraphics{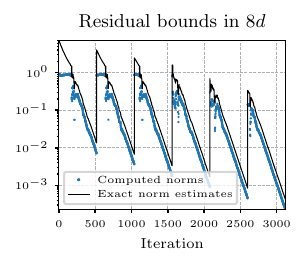}
	\endminipage\hfill
	\minipage{0.33\textwidth}
	\includegraphics{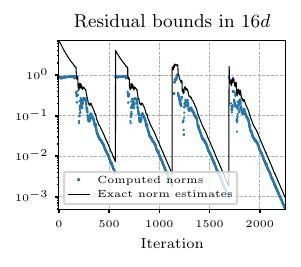}
	\endminipage
	\caption{Norms of computed error estimates and error bounds in dependence on the iteration number for the heat equation with initial value, for $d=4, 8, 16$.}
	\label{fig:residual-initialvalue}
\end{figure}

\begin{figure}
	\minipage{0.5\textwidth}
	\includegraphics{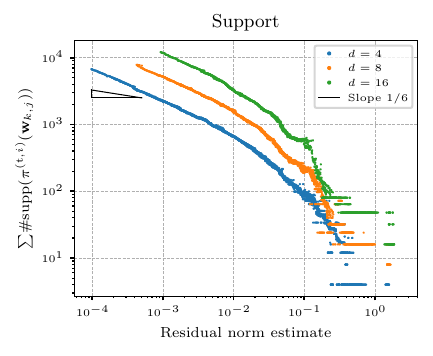}
	\endminipage
	\minipage{0.45\textwidth}
	\includegraphics{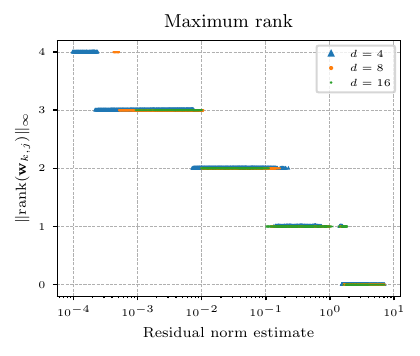}
	\endminipage\hfill
	\caption{Sum of the support and maximum rank per iteration versus the current residual norm estimate for the heat equation with initial value, for $d=4,8,16$.}
	\label{fig:supp-res-initial-value}
\end{figure}

\begin{figure}
	\minipage{0.33\textwidth}
	\includegraphics{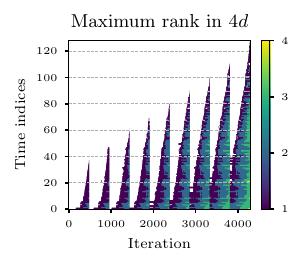}
	\endminipage\hfill
	\minipage{0.33\textwidth}
	\includegraphics{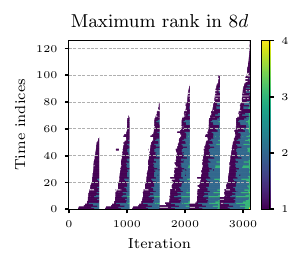}
	\endminipage\hfill
	\minipage{0.33\textwidth}
	\includegraphics{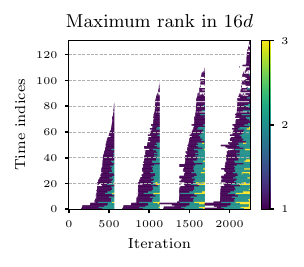}
	\endminipage
	\caption{Maximum ranks for each time index in dependence on the iteration number for the heat equation with initial value, for $d=4, 8, 16$.}
	\label{fig:ranks-it-initial-value}
\end{figure}

\begin{figure}
	\minipage{0.3333\textwidth}
	\includegraphics{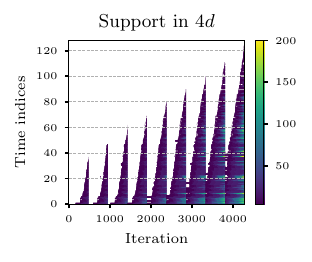}
	\endminipage\hfill
	\minipage{0.3333\textwidth}
	\includegraphics{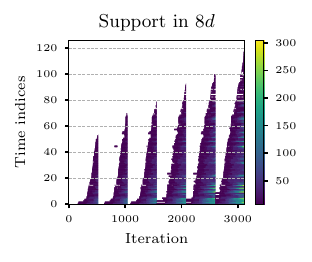}
	\endminipage\hfill
	\minipage{0.3333\textwidth}
	\includegraphics{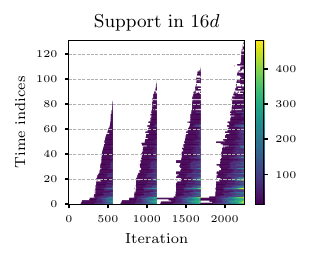}
	\endminipage
	\caption{Sum of one-dimensional support for each time index in dependence on the iteration number for the heat equation with initial value, for $d=4,8,16$.}
	\label{fig:supp-it-initial-value}
\end{figure}

\begin{figure}
	\includegraphics{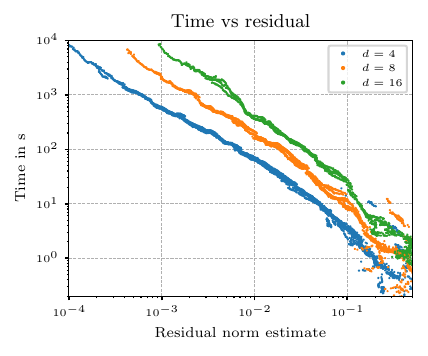}
	\caption{Time required for each inner iteration step for the heat equation with initial value for $d=4,8,16$ in dependence on the achieved residual estimate.}
	\label{fig:time-initial-value}
\end{figure}

In the following we present the same quantities as in Section \ref{sec:results-source}. Figure \ref{fig:residual-initialvalue} shows the residuals as well as the corresponding estimates for the $\cX$-norm with respect to the iteration number for the three different spatial dimensions $d \in \{4,8,16\}$, again with the expected behaviour. Figures \ref{fig:supp-res-initial-value}, \ref{fig:ranks-it-initial-value} and \ref{fig:supp-it-initial-value} show results for the one-dimensional supports as well as for the maximum ranks. Figure \ref{fig:supp-res-initial-value} compares the sum of the one-dimensional support with the computed residual norm estimates, where one can see that preasymptotic behaviour is still present.  Figure \ref{fig:time-initial-value} shows the time for each inner iteration step in dependence of the residual estimates for different spatial dimensions. As in the previous test case, we observe a polynomial dependence of the costs on the problem dimension. Altogether, the results are consistent with the previous test case with vanishing initial value. However, due to the higher computation costs, in the present case we do not enter the asymptotic regime for the one-dimensional support sizes.

\section{Conclusion and Outlook}\label{sec:Conclusion}

We have constructed a space-time adaptive solver for parabolic PDEs that combines sparse wavelet approximations in time with low-rank hierarchical tensor approximations in the spatial variables. The method yields guaranteed error bounds with respect to the exact solution in the natural norm of $L_2(0,T; H^1_0(\Omega)) \cap H^1(0,T; H^{-1}(\Omega))$. 
In addition, we obtain near-optimal bounds on hierarchical ranks and discretization sizes in terms of those of best approximations of similar accuracy, and corresponding bounds on the total computational complexity of the method in terms of the total error with respect to the exact solution of the PDE.

The present paper shows how such methods can be constructed and that they are applicable to problems in large dimensions $d$. On this basis, there is a number of directions for further work. Concerning the quantitative efficiency of the method, based on the adaptive residual approximation constructed here, several further improvements are possible. In particular, the conceptually simple but quantitatively rather expensive approximate Richardson iteration on the full sequence space can be replaced by successively refined Galerkin discretizations as in \cite{AliUrban:20}. For the arising Galerkin subproblems, solvers that are optimized for the particular structure of the combined sparse and low-rank approximations can be considered. Moreover, it will be of interest to adapt the method to convection-diffusion problems.

Another direction of future work concerns the approximability of solutions in the particular sense that is exploited by our adaptive scheme, especially in the case of large $d$. On the one hand, the performance of the method depends on best $n$-term approximations of the lower-dimensional contractions $\piti(\bu)$ for $i=1,\ldots,d$; this can be interpreted as the approximability by adaptive tensor product wavelets of densities of solutions with respect to the temporal and a single spatial variable.  On the other hand, the performance of the method depends on the ranks of spatial hierarchical low-rank approximations of the coefficient tensors $\bu_{\nu_\rt}$ for each fixed temporal basis index $\nu_\rt \in \vee_\rt$.

Since our method does not make explicit use of any assumptions on approximability of solutions, but is guaranteed to automatically produce near-best approximations, it can be regarded as a numerical test for such approximability.
Our numerical results show a (poly)logarithmic growth of hierarchical ranks with respect to the total solution error. This leads us to the conjecture that for parabolic problems of this type with data (such as initial values and source terms) having polylogarithmic growth of best approximation ranks, this polylogarithmic rank growth transfers to solutions. Such results have been obtained for elliptic problems in \cite{DahmenDeVoreGrasedyckSuli}, but we are not aware of a comparable result for parabolic problems. Note that, as such exponential-type convergence of low-rank approximations is generally a structural feature of solutions, this is not covered by singular value decay estimates for generic elements of smoothness classes. For example, the estimates for functions with dominating mixed smoothness in \cite{GriebelHarbrechtSchneider,BachmayrNouySchneider} yield only algebraic convergence with respect to ranks, which in view of the numerical results obtained here, and those for the elliptic case in \cite{AliUrban:20,bachmayr2014adaptive}, is generally far from sharp for solutions of the considered PDE problems.

\begin{appendix}

\section{Exponential Sum Approximations}\label{app:expsum}
In this section we give the proof of Theorem \ref{thm:approx-scaling}. The starting point for obtaining exponential sum approximations are integral representations based on the inverse Laplace transform. These are approximated in suitably transformed form by the trapezodial rule.
We start with an auxiliary statement on the values of the derivative of the Dawson function on specific lines in the complex plane. 
\begin{lemma}\label{lem:dawson-deriv-complex-bound}
	One has
	\begin{align*}
			\bigabs{1 - 2c(\sqrt{3} \pm i) F\bigl(c(\sqrt{3} \pm i)\bigr)} \leq 1+ \sqrt{\frac{2\pi}{e}} + \frac{2}{\sqrt{3}} x_{\min} F(x_{\min}) \eqqcolon c_{\mathrm{daw}}
	\end{align*}
	for each $c \geq 0$, where $x_{\min} $ is the global minimizer of $\frac{dF}{dx}$ \cite{oeisDawInfl}.
\end{lemma}
\begin{proof}
	We will give the proof for the ``+'' case. The other case can be shown in the same way. Let $c \geq 0$ be arbitrary and set $z = c (\sqrt{3} + i)$. First, we start by estimating $\abs{F(z)}$. Choosing an axis-parallel contour, we have
	\begin{align*}
		\int\limits_0^z e^{\xi^2} d\xi &= \int\limits_0^{\Imag z} e^{(iy)^2} dy + \int\limits_0^{\Real z} e^{(x + i \Imag(z))^2} dx \\
		&= \frac{\sqrt{\pi}}{2} \erf(c) + e^{-c^2} \int\limits_0^{\sqrt{3} c} e^{x^2} e^{ 2 i c x} dx .
	\end{align*}
	Using the definition of the Dawson function as well as the triangle inequality, we obtain
	\begin{equation*}
		\abs{F(z)} = \bigabs{e^{-z^2}} \biggabs{\int\limits_0^z e^{\xi^2} d\xi} 
		\leq e^{-2c^2} \frac{\sqrt{\pi}}{2} \erf(c) + F(\sqrt{3} c) .
	\end{equation*}
	We next take into account that the derivative of the Dawson function $\frac{dF}{dx} = 1 - 2xF$  has its global minimum at $x_{\min}$. As a consequence,
	\begin{align*}
		\abs{1-2zF(z)} &\leq 1 + 2\abs{z}\abs{F(z)} \\
		&\leq 1 + 2 c e^{-2c^2} \sqrt{\pi} \erf(c) + 2 c F(\sqrt{3} c) \\
		&\leq 1+ \sqrt{\frac{2\pi}{e}} + \frac{2}{\sqrt{3}} x_{\min} F(x_{\min}) = c_{\mathrm{daw}} . \qedhere
	\end{align*}
\end{proof}

We use the following definition and approximation error bound from \cite{Stenger:1607843}.
\begin{definition}
	For $\zeta > 0$, let $D_\zeta = \{z \in \C : \abs{\mathrm{Im} \ z} < \zeta \}$, and for $0 < \varepsilon < 1$, let
	\begin{align*}
		D_\zeta(\varepsilon) = \{z \in \C : \abs{\mathrm{Re} \ z} < \varepsilon^{-1}, \abs{\mathrm{Im} \ z } < \zeta(1-\varepsilon)\} .
	\end{align*}
For $v$ analytic in $D_\zeta$ let
\begin{align*}
	N_1(v,D_\zeta) = \lim\limits_{\varepsilon \to 0} \int\limits_{\partial D_{\zeta}(\varepsilon)} \abs{v(z)} \abs{\dd z} .
\end{align*}
\begin{theorem}[{see \cite[Theorem 3.2.1]{Stenger:1607843}}]\label{thm:trapzedioal-rule}
	Let $g$ be analytic in $D_\zeta$ with $N_1(g,D_\zeta) < \infty$, then
	\begin{align*}
		\lrabs{\int_{\R} g(x) \sdd x - h \sum\limits_{k \in \Z} g(kh)} \leq \frac{e^{-\pi \zeta / h}}{2 \sinh(\pi \zeta/h)} N_1(g,D_\zeta) .
	\end{align*}
\end{theorem}
\end{definition}

\begin{proof}[Proof of Theorem \ref{thm:approx-scaling}]
Our starting point is the representation
\begin{align*}
	\frac{\sqrt{s}}{s + \tilde{a}} = \int\limits_0^\infty\frac{1}{\sqrt{\pi}\sqrt{y}} \Bigl(1 - 2 \sqrt{ya}  F\bigl(\sqrt{ya}\bigr) \Bigr) e^{-ys} \sdd y,
\end{align*}
where we make use of the property $\frac{\dd F}{\dd x}(x) + 2xF(x) = 1$ of the Dawson function $F$ given by \eqref{eq:dawsondef}. Substituting $ya = e^x$ yields 
\begin{align*}
	\frac{\sqrt{s}}{s + a} &= \frac{1}{\sqrt{a}\sqrt{\pi}} \int\limits_{\R}  \left( 1 - 2 e^{\frac{x}{2}} F\left(e^{\frac{x}{2}}\right)\right) e^{\frac{x}{2} - \frac{1}{a} e^x s} \sdd x \\
	&= 2 \frac{1}{\sqrt{a} \sqrt{\pi}}  \int\limits_{\R}  \tfrac{dF}{dx} (e^{x/2}) e^{-\frac{1}{a}e^x s} \sdd x .
\end{align*}
The integrand is holomorphic in the strip $\{x+iy : x \in \R, \abs{y} < \pi/3\}$. In order to apply Theorem \ref{thm:trapzedioal-rule}, we need to estimate the quantity
\begin{align*}
	N_1(g,D_\zeta) = \int_{\R} \abs{g(x+i\zeta)} \sdd x + \int_{\R} \abs{g(x-i\zeta)} \sdd x
\end{align*}
where $g(x) = \frac{\sqrt{a}}{\sqrt{\pi}} \left( 1 - 2 e^{x/2} F(e^{x/2})\right) e^{\frac{x}{2} - \frac{1}{a} e^x s}$. By Lemma \ref{lem:dawson-deriv-complex-bound}, we have
\begin{align*}
	 \lrabs{1 - 2 e^{\frac{x \pm y i}{2}} F\left(e^{\frac{x \pm yi}{2}}\right)} \leq c_{\mathrm{daw}}
\end{align*}
for $y = \pi/3$ and all $x \in \R$. Moreover,
\begin{align*}
	\Bigabs{e^{\frac{x+yi}{2} - \frac{1}{a} e^{x+yi} t}} &= e^{\mathrm{Re}(\frac{x+yi}{2} - \frac{1}{a} e^{x+yi} t)} \\
	&= e^{\frac{x}{2}- \frac{1}{a}e^x \cos(y) t} .
\end{align*}
The same transformation can be done for the negative imaginary part using $\cos(y) = \cos(-y)$. Putting this together, for $\zeta \in (0, \frac{\pi}3)$, which ensures $\cos(\zeta) > 0$, we obtain
\begin{equation*}
	N_1(g,D_\zeta) \leq 2 \frac{c_{\mathrm{daw}}}{\sqrt{a}\sqrt{\pi}} \int_{\R} e^{\frac{x}{2}- \frac{1}{a}e^x \cos(\zeta) s} \sdd x 
	= \frac{2 c_{\mathrm{daw}}}{\sqrt{s \ \cos(\zeta)}} ,
\end{equation*}
where we have used that $
	{\sqrt{\pi} \, \erf(\sqrt{c} \ e^{x/2})}/ {\sqrt{c}} $
is the antiderivative of $e^{\frac{x}{2} - c e^x}$ for $c > 0$.
Theorem \ref{thm:trapzedioal-rule} now yields
\begin{align*}
	\lrabs{\frac{\sqrt{s}}{s+a} -  \sum\limits_{k \in \Z} h_a w_a(kh_a) e^{-\alpha_a(kh_a)s}} &\leq \frac{e^{-\frac{\pi^2}{3h_a}}}{2 \sinh(\frac{\pi^2}{3h_a})} 2 c_{\mathrm{daw}} \sqrt{2} \frac{1}{\sqrt{s}} \\ 
	&\leq \frac{e^{-\frac{\pi^2}{3h_a}}}{2 \sinh(\frac{\pi^2}{3h_a})} 10 (1+a) \frac{\sqrt{s}}{s+a} 
	\leq \delta \frac{\sqrt{s}}{s+a}
\end{align*}
with the choice of $h_a$ in \eqref{eq:h-epxonential-sum}, where we have used $s \geq 1$ as well as $2 c_{\mathrm{daw}} \sqrt{2}  \leq 10$ in the second line.

In the next step, we truncate the doubly infinite sum from above at the index $n^+_a$ at cost of an maximum error of $\delta$. We use that
\begin{align*}
	\frac{2}{c+1} e^{-c e^x} F(e^{x/2}) + \sqrt{\pi} \frac{\sqrt{c}}{c+1} \erf\bigl(\sqrt{c}e^{x/2}\bigr)
\end{align*}
is the antiderivative of $
	( 1 - 2 e^{\frac{x}{2}} F(e^{\frac{x}{2}})) e^{\frac{x}{2} - c e^x} $.
Setting $u_a(x) = w_a(x) e^{-\alpha_a(x)s}$ and $c = s/a$, we obtain
\begin{align*}
	\sum\limits_{k > n^+_a} h_a u_a(k h_a) &\leq \int\limits_{n^+_a}^\infty h_a w_a(x h_a)\, e^{-\alpha_a(x h_a)s} \sdd x 
	= \frac{1}{\sqrt{\pi}\sqrt{a}}\int\limits_{n^+_a h_a}^\infty\Bigl( 1 - 2 e^{\frac{x}{2}} F\bigl(e^{\frac{x}{2}}\bigr)\Bigr) e^{\frac{x}{2} - c e^x} \sdd x \\
	&= \frac{1}{\sqrt{\pi}\sqrt{a}} \left(\sqrt{\pi} \frac{\sqrt{c}}{c+1} \operatorname{erfc}\Bigl(\sqrt{c} e^{\frac{n^+_a h_a}{2}}\Bigr) - \frac{2}{c+1} e^{-ce^{n^+_a h_a}} F\Bigl(e^{\frac{n^+_a h_a}{2}}\Bigr) \right) \\
	&\leq \frac{\sqrt{s}}{s+a} \operatorname{erfc}\biggl(\frac{1}{\sqrt{a}} e^{\frac{n^+_a h_a}{2}} \biggr) < \delta \frac{\sqrt{s}}{s+a}
\end{align*}
with $n^+_a$ chosen according to \eqref{eq:nplus-eponential-sum}. In the last line we use that the complementary error function $\operatorname{erfc} = 1-\erf$ is monotonically decreasing and $s \geq 1$.

Finally, we truncate the sum in the negative range of indices. To this end, we limit the interval of possible values of $s$ to $[1,K]$ for a fixed $K > 1$. We make use of the fact that the antiderivative of
$	(1 - 2 e^{-\frac{x}{2}} F(e^{-\frac{x}{2}})) e^{-\frac{x}{2}-c e^{-x}}  $
for $c > 0$ is given by
\begin{align*}
	- \frac{2}{c+1} e^{-ce^{-x}} F\bigl(e^{-\frac{x}{2}}\bigr) - \sqrt{\pi} \frac{\sqrt{c}}{c+1} \erf\bigl(\sqrt{c} e^{-\frac{x}{2}}\bigr) .
\end{align*}
Setting $c = s/a$ as before, we arrive at
\begin{align*}
	\sum\limits_{k < -n} h_a u_a(k h_a) &= \sum\limits_{k > n} h_a u_a(-k h_a) \\
	&\leq \frac{1}{\sqrt{\pi}\sqrt{a}}\int\limits_{n h_a}^\infty \left(1- 2e^{-\frac{x}{2}F\left(e^{-\frac{x}{2}}\right)}\right) e^{-\frac{x}{2}-c e^{-x}} \sdd x \\
	&=  \frac{1}{\sqrt{\pi}\sqrt{a}} \left(\frac{2}{c+1} e^{-ce^{-nh_a}} F\bigl(e^{-\frac{nh_a}{2}}\bigr) + \sqrt{\pi} \frac{\sqrt{c}}{c+1} \erf\bigl(\sqrt{c} e^{-\frac{nh_a}{2}}\bigr)\right) \\
	&= \frac{\sqrt{s}}{s+a} \left( \erf\left(\frac{\sqrt{s}}{\sqrt{a}} e^{-\frac{nh_a}{2}}\right) + \frac{2\sqrt{a}}{\sqrt{s}\sqrt{\pi}} e^{-\frac{s}{a} e^{-nh_a}} F\left(e^{-\frac{nh_a}{2}}\right)\right) \\
	&= \frac{\sqrt{s}}{s+a} f_a(s,nh_a) .
\end{align*}
If we assume that $n$ is chosen sufficiently large to ensure $F(e^{-(nh_a)/2}) e^{-(nh_a)/2} < \frac12$, one can easily check that the function $f_a(s,x)$ is monotonically decreasing in some interval $s \in [0,z]$ and monotonically increasing in the remaining interval $s \in (z,\infty)$. We thus conclude that the maximum value of the function $f_a(s,x)$ for a given $x$ and the time interval $[1,K]$ has to be located at one of the boundary values.
\end{proof}

\section{Auxiliary proofs}\label{sec:aux}

\begin{proof}[Proof of Lemma \ref{lem:recompress}]
	For $\bu \in \ell_2(\vee)$ and $\sfr \in \mathcal{R}^{\veet}$, let $\tilde{\mathrm{P}}_{\bu,\sfr} \colon \ell_2(\vee) \to \mathcal{F}(\sfr)$ be defined by
	\[
		\norm{\bu - \tilde{\mathrm{P}}_{\bu,\sfr} \bu} = \min_{\bw \in \mathcal{F}(\sfr)} \norm{\bu - \bw} .
	\]
	The operator has the representation
	\[
	\tilde{\mathrm{P}}_{\bu,\sfr} = \sum\limits_{\nut \in \vee_\rt} \mathbf{E}_{\nut} \otimes \tilde{\mathrm{P}}_{\rx,\bu_{\nut},\sfr_{\nut}}
	\]
	with $\mathbf{E}_{\nut} = \mathbf{e}_{\nut}^\intercal \mathbf{e}_{\nut}$ and where $\tilde{\mathrm{P}}_{\rx,\bu_{\nut},\sfr_{\nut}} \colon \ell_2(\vee_{\rx}) \to \mathcal{H}(r_{\nut})$ is the linear projection such that $\tilde{\mathrm{P}}_{\rx,\bu_{\nut},\sfr_{\nut}} \bu_{\nut}$ is a best approximation in $\mathcal{H}(\sfr_{\nut})$ as in \cite[Lemma 1]{BachmayrNearOptimal}.
	A direct consequence of this representation is that the operator $\tilde{\mathrm{P}}_{\bu,\sfr}$ is also a linear projection. According to Proposition \ref{prop:recompress-upper-bound}, we obtain
	\begin{equation*}
		\begin{aligned}
			\norm{\bv - \mathrm{P}_{\bar{\sfr}(\bu,\alpha \eta)} \bv} &\leq \kappa_P \inf_{\bw \in \mathcal{F}(\bar{\sfr}(\bu,\alpha \eta))} \norm{\bv - \bw} \\
			&\leq \kappa_P \norm{\bv - \tilde{\mathrm{P}}_{\bu,\bar{\sfr}(\bu,\alpha \eta)} \bv} \\
			&\leq \kappa_P  \big(\norm{(\bI - \tilde{\mathrm{P}}_{\bu,\bar{\sfr}(\bu,\alpha \eta)}) (\bv - \bu)}  + \norm{\bu - \tilde{\mathrm{P}}_{\bu,\bar{\sfr}(\bu,\alpha \eta)} \bu}\big) \\
			&\leq \kappa_P (1+\alpha) \eta . 
		\end{aligned}
	\end{equation*}
	By definition $	\rank_\infty\bigl(\hat{\Pro}_{\kappa_{\Pro}(1+\alpha)\eta} (\bv)\bigr)$ is chosen minimally to archive the bound $\kappa(1+\alpha)\eta$, which yields \eqref{eq:recompress-rank-upper-bound}. Furthermore, \eqref{eq:recompress-norm} is an immediate consequence of the triangle inequality.
\end{proof}

\begin{proof}[Proof of Theorem \ref{thm:combinedcoarsen}]
	Combining \eqref{eq:recompress-norm} in Lemma \ref{lem:recompress} and the definition of the coarsening operator in \eqref{eq:coarsen-threshold} together with the triangle inequality yield relation \eqref{eq:recompress-coarsen-norm}.
	
	The statements in \eqref{eq:recompress-coarsen-rank} follow directly from Theorem \ref{thm:recompress}. The additional coarsening operator can at most reduce the ranks and does not affect these estimates.
	
	The main part is the proof of the estimates \eqref{eq:recompress-coarsen-supp}. We set $\hat{\bw} = \hat{\Pro}_{\kappa_{\Pro}(1+\alpha)\eta} (\bv)$. Let $N \in \N$ be minimal such that $\norm{\bu - \Res_{\bar{\Lambda}(\bu,N)} \bu} \leq \alpha \eta$. Then
	\begin{align*}
		\norm{\hat{\bw} -\Res_{\bar{\Lambda}(\bu,N)} \hat{\bw}} &\leq \norm{(\bI - \Res_{\bar{\Lambda}(\bu,N)}) (\bu - \hat{\bw})} + \norm{\bu - \Res_{\bar{\Lambda}(\bu,N)} \bu} \\
		&\leq \norm{\bu - \hat{\bw}} +  \norm{\bu - \Res_{\bar{\Lambda}(\bu,N)} \bu}  \leq (1+\kappa_{\Pro}(1+\alpha) + \alpha) \eta ,
	\end{align*}
	where we used Lemma \ref{lem:recompress} to bound the first summand on the right-hand side. Hence, by \eqref{eq:coarsen-best-ineq}, we have
	\begin{equation}\label{eq:coarsen-fixed-N-error}
	\begin{aligned}
		\norm{\hat{\bw} -\Res_{\Lambda(\hat{\bw},N)} \hat{\bw}} &\leq \kappa_{\Cor} \norm{\hat{\bw} - \Res_{\bar{\Lambda}(\hat{\bw},N)} \hat{\bw}} \\ 
		&\leq \kappa_{\Cor} \norm{\hat{\bw} - \Res_{\bar{\Lambda}(\bu,N)}\hat{\bw}} \leq \kappa_{\Cor} (1+\alpha)(\kappa_{\Pro}+1) \eta .
	\end{aligned}
\end{equation}
	Without loss of generality we may assume that $N \geq d$. Keeping the optimality \eqref{eq:coarsen-best-optimality} of the best coarsen operator in mind, property \eqref{eq:bound-restriction-contraction-mod} yields
	\begin{align*}
		\alpha \eta < \norm{\bu - \Res_{\Lambda(\bu,N-1)} \bu} &\leq \inf\limits_{\sum_i \# \Lambda^{(t,i)} \leq N-1} \left( \sum\limits_{i=1}^d \lrnorm{\piti(\bu) - \Res_{\Lambda^{(t,i)}} \piti(\bu)  }^2 \right)^{\frac{1}{2}} \\
		&\leq \sum\limits_{i=1}^d \inf\limits_{\# \Lambda^{(t,i)} \leq (N-1)/d} \lrnorm{\piti(\bu) - \Res_{\Lambda^{(t,i)}} \piti(\bu)} \\
		&\leq \left(\frac{N-1}{d}\right)^{-s} \sum\limits_{i=1}^d \norm{\piti(\bu)}_{\cAs}  \\
		&\leq 2^s \left(\frac{N}{d}\right)^{-s} \sum\limits_{i=1}^d \norm{\piti(\bu)}_{\cAs} \,.
	\end{align*}
	We note that by \eqref{eq:coarsen-fixed-N-error}, the coarsening operator $\hat{\Cor}_{\kappa_{\Cor}(\kappa_{\Pro}+1)(1+\alpha)\eta}$ retains at most $N$ terms. We combine this fact with the previous estimate, which results in
	\begin{align}\label{eq:recompress-coarsen-supp-ineq}
		\sum\limits_{i=1}^d \# \supp\left(\piti(\bw_\eta)\right) \leq N \leq 2 d \alpha^{-\frac{1}{s}} \eta^{-\frac{1}{s}} \left( \sum\limits_{i=1}^d \norm{\piti(\bu)}_{\cAs} \right)^{\frac{1}{s}} .
	\end{align}
	Hence the first statement in \eqref{eq:recompress-coarsen-supp} is shown. Now let $\hat{N}_i = \# \supp( \piti(\bu))$ for $i=1,\dots,d$ and $\hat{N} = \sum_i \hat{N}_i$ . We may assume without loss of generality $\hat{N}_i > 0$ for $i=1,\dots,d$. Resolve \eqref{eq:recompress-coarsen-supp-ineq} for $\eta$ and inserting into \eqref{eq:recompress-coarsen-norm}, we observe
	\begin{align}\label{eq:recompress-coarsen-ell2norm-contraction}
		\norm{\bu - \bw_{\eta}} \leq \hat{N}^{-s} C(\alpha) d^s \sum\limits_{i=1}^d \norm{\piti(\bu)}_{\cAs}
	\end{align}
	where $C(\alpha) = 2^s \alpha^{-1} (1+\kappa_{\Pro}(1+\alpha) + \kappa_{\Cor}(\kappa_{\Pro}+1)(1+\alpha))$.
	
	Let $\hat{\bu}_i$ be the best $\hat{N}_i$-term approximation to $\piti(\bu)$, then by using the properties of Proposition \ref{prop:cAs-properties} we observe 
	\begin{align*}
		\norm{\piti(\bw_{\eta})}_{\cAs} &\leq 2^s \left( \norm{\hat{\bu}_i}_{\cAs} + \norm{\hat{\bu}_i - \piti(\bw_{\eta})}_{\cAs}  \right) \\
		&\leq 2^s \left( \norm{\piti(\bu)}_{\cAs} + (2\hat{N}_i)^s \norm{\hat{\bu}_i - \piti(\bw_{\eta})} \right) \\
		&\leq 2^s \left( \norm{\piti(\bu)}_{\cAs} + (2\hat{N}_i)^s \left(\norm{\hat{\bu}_i - \piti(\bu)} + \norm{\piti(\bu) - \piti(\bw_{\eta})}\right) \right) \\
		&\leq 2^s \left( (1+2^s) \norm{\piti(\bu)} + (2\hat{N}_i)^s \norm{\piti(\bu) - \piti(\bw_{\eta})} \right),
	\end{align*}
	where we used that $\# \supp(\hat{\bu}_i -\piti(\bw_{\eta})) \leq 2\hat{N}_i$ in the second line. As a consequence of the Cauchy-Schwarz inequality, we have the componentwise estimate
	\begin{align*}
		\abs{\piti_{\nu_\rt,\nu_\rx}(\bu) - \piti_{\nu_\rt,\nu_\rx}(\bw_{\eta})} \leq \piti_{\nu_\rt,\nu_\rx}(\bu - \bw_{\eta}),
	\end{align*}
	which yields
	\begin{align*}
		\norm{\piti(\bu) - \piti(\bw_{\eta})} \leq \norm{\piti(\bu-\bw_{\eta})} = \norm{\bu - \bw_{\eta}} .
	\end{align*}
	Combining this fact with \eqref{eq:recompress-coarsen-ell2norm-contraction}, we obtain
	\begin{align*}
		\norm{\piti(\bw_\eta)}_{\cAs} \leq 2^s (1+2^s) \norm{\piti(\bu)}_{\cAs} + 2^s C(\alpha) d^s \hat{N}^{-s} (2\hat{N}_i)^s \biggl(\sum\limits_{k=1}^d \norm{\pi^{(\rt,k)}(\bu)}_{\cAs} \biggr).
	\end{align*}
	Summing over $i=1,\dots,d$ and noting that $
		\hat{N}^{-s} \sum\limits_{i=1}^d (2\hat{N}_i)^s  \leq 2^s d^{\max(0,1-s)} $,
 	we arrive at the second statement in \eqref{eq:recompress-coarsen-supp}.
\end{proof}

\end{appendix}

\bibliographystyle{amsplain}
\bibliography{BFalrparabolic}

\end{document}